\documentclass[letterpaper,11pt]{amsart}

\usepackage{fullpage}
\usepackage{graphicx}
\usepackage{amsmath, amssymb, amsfonts, amsthm}

\usepackage{color}
\usepackage[plainpages=false,colorlinks=true,citecolor=blue,filecolor=black,linkcolor=red]{hyperref}

\usepackage[ruled,vlined]{algorithm2e}
\newtheorem{thm}{Theorem}[section]

\newtheorem{prop}[thm]{Proposition}
\newtheorem{lem}[thm]{Lemma}

\theoremstyle{remark}
\newtheorem{rem}[thm]{Remark}

\usepackage[backend=bibtex,firstinits=true,maxcitenames=2,maxbibnames=10,style=alphabetic,natbib=true,url=false,sorting=nyt,doi=true,backref=false]{biblatex}

\renewbibmacro{in:}{\ifentrytype{article}{}{\printtext{\bibstring{in}\intitlepunct}}}

\begin{document}

\title{Interface dynamics for an Allen-Cahn-type  \\ equation governing a matrix-valued field}
\author{Dong Wang}
\address{Department of Mathematics, University of Utah, Salt Lake City, UT}
\email{dwang@math.utah.edu}
\author{Braxton Osting}
\thanks{B. Osting is partially supported by NSF DMS 16-19755 and 17-52202.}
\address{Department of Mathematics, University of Utah, Salt Lake City, UT}
\email{osting@math.utah.edu}
\author{Xiao-Ping Wang}
\address{Department of Mathematics, Hong Kong University of Science and Technology, Hong Kong}
\email{mawang@ust.hk}
\thanks{X.-P. Wang was supported in part by the Hong Kong Research Grants Council (GRF grants 16302715, 16324416, 16303318, and NSFC-RGC joint research grant N-HKUST620/15)}



\date{\today}

\begin{abstract} 
We consider the initial value problem for the generalized Allen-Cahn equation,
\[\partial_t \Phi = \Delta \Phi - \varepsilon^{-2} \Phi (\Phi^t \Phi - I), \qquad x \in \Omega, \ t \geq 0, \]
where $\Phi$ is an $n\times n$ real matrix-valued field,  $\Omega$ is a two-dimensional square with periodic boundary conditions, and $\varepsilon > 0$.  This equation is the gradient flow for the energy, 
$E(\Phi) := \int  \frac{1}{2} \|\nabla \Phi \|^2_F + \frac{1}{4 \varepsilon^2}  \| \Phi^t \Phi - I \|^2_F$,  
where $\| \cdot \|_F$ denotes the Frobenius norm. The primary contribution of this paper is to use asymptotic methods to describe the solution of this initial value problem. 

If the initial condition has single-signed determinant, at each point of the domain, at a fast $O(\varepsilon^{-2} t)$ time scale, the solution evolves towards the closest orthogonal matrix. Then, at the $O(t)$ time scale, the solution evolves according to the $O_n$ diffusion equation. Stationary solutions to the  $O_n$ diffusion equation are analyzed for $n=2$. 

If the initial condition has regions where the determinant is positive and  negative, a free interface develops. Away from the interface, in each region, the matrix-valued field behaves as in the single-signed determinant case. 
At the $O(t)$ time scale, the interface evolves in the normal direction by curvature. 
At a slow $O(\varepsilon t)$ time scale, the interface is driven by curvature and the surface diffusion of the matrix-valued field. For $n=2$, the interface is driven by curvature and the jump in the squared tangental derivative of the phase across the interface. In particular, we emphasize that the interface when $n\geq 2$ is driven by surface diffusion, while for $n=1$, the original Allen--Cahn equation, the interface is only driven by mean curvature. 

A variety of numerical experiments are performed to verify, support, and illustrate our analytical results. 
\end{abstract}

\maketitle

\section{introduction}
We consider the initial value problem for the generalized Allen--Cahn equation, 
\begin{equation} \label{e:orig}
\begin{cases}
&\partial_t A = \Delta A - \varepsilon^{-2}  A(A^tA - I), \qquad  x \in \Omega, \ t> 0 \\
& A(t=0,x) = A_0(x), 
\end{cases}
\end{equation}
where $A(t,x)\in M(n)$ is a real matrix-valued field and $\varepsilon >0$ is a small parameter. 
For simplicity, we take the domain $\Omega$ to be a two-dimensional square, $[-1/2, 1/2]^2$, with periodic boundary conditions. 
It is not difficult to show that \eqref{e:orig} is the gradient flow for the energy, 
\begin{equation} \label{e:energy}
E(A) := \int_\Omega  \frac{1}{2} \|\nabla A \|^2_F + \varepsilon^{-2} W(A),  
\qquad \textrm{where } W(A) := \frac{1}{4}  \| A^t A - I \|^2_F, 
\end{equation} 
and $\| \cdot \|_F$ denotes the Frobenius norm. 
Roughly speaking, for $\varepsilon$ small, the solution to \eqref{e:orig} is smoothed by the first term and the second term keeps the pointwise values of the matrix-valued field near $O_n$,  the $n \times n$ orthogonal matrix group. 
This problem was first introduced in \cite{osting2017} as a model problem for several applications where 
a smooth matrix-valued field arises, including  
crystallography, where the matrix-valued field describes the local crystal orientation and 
inverse problems in image analysis {\it e.g.}, diffusion tensor MRI or fiber tractography, where it is of interest to estimate a matrix- or orientation-valued function \cite{OstingWang2018}. 

\medskip

It is clear that when $n=1$, \eqref{e:orig} reduces to the original Allen--Cahn equation \cite{Cahn_2013}, which  models the behavior of two immiscible fluids. In \eqref{e:orig}, because the reaction rate is large compared to the diffusion rate, the solution at each point $x\in \Omega$ quickly tends to a stable equilibrium state of the reaction process, {\it i.e.}, a minimum of $W(A)$. For the case $n=1$, the local minima of $W(A)$ are $1$ and $-1$. There are two cases. 
\begin{itemize}
\item[(i)] If $A_0(x) >0$ (or $A_0(x) <0$ respectively) for every $x\in \Omega$,  then $A(t,x)$ will tend to $1$ (or $-1$ resp.). In this case, the effect of diffusion is only to slightly change the rate at which the solution approaches $1$ (or $-1$ resp.). 
\item[(ii)] However, if $ \Omega_+$ and $\Omega_-$  are such that $\Omega = \overline{\Omega_+} \cup \overline{\Omega_-}$ and $\Omega_+ \cap \Omega_- = \emptyset$,  with 
\begin{equation*}
\begin{cases}
A_0(x)> 0,  &  x \in \Omega_+ \\
A_0(x)< 0,  &  x \in \Omega_-
\end{cases},
\end{equation*}
then a boundary layer in the solution develops at an interface between the two subdomains. 
Through a boundary layer expansion, one can show that 
$A(t,x) = \pm 1$ for $x$ away from the interface and that the interface evolves in the normal direction by its mean curvature. 
\end{itemize}
We refer  to \cite{Bronsard_1991, Bronsard_1993} and references therein for more details on the $n=1$ case. 

\medskip

For $n\geq 2$, the minimizers of $W(A)$ are elements of $O_n$. Recall that 
$$
O_n = SO_n \cup SO_n^-
$$ 
where 
$SO_n$ denotes the special orthogonal group of the orthogonal matrices with determinant $1$  and 
$SO_n^-$ is the set of orthogonal matrices with determinant $-1$. 
In this paper, we use matched asymptotic expansion methods \cite{weinan2011principles} to show that, as in the $n=1$ case, there are two cases. 
\begin{itemize} 
\item[(i)] If the initial condition $A_0(x)$ has either positive or negative determinant for each $x \in \Omega$, then no interface develops. We show in Section~\ref{sec:part1slow} that the $O(t)$ time dynamics of the leading order solution satisfies the \emph{$O_n$ diffusion equation}, 
$$
\partial_t B(t,x) =  \frac{1}{2} \left( \Delta B(t,x) B^t(t,x) - B(t,x) \Delta B^t(t,x)  \right) B(t,x)  
$$
with initial condition given by
$$
B(0,x) = \Pi_{O_n} A_0(x).
$$
Here, and throughout this paper, we use $\Pi_{O_n} A = \arg \min_{B \in O_n} \| A - B\|_F$ to denote the closest point in the orthogonal matrix group to the matrix $A$. We show in Proposition \ref{p:OnDiffEq}  that the leading order solution remains in $O_n$ pointwise for all time $t>0$. We discuss stationary solutions to the $O_n$ diffusion equation in Section~\ref{s:Harmonic}. 

\item[(ii)] 
In the second case, the initial condition $A_0(x)$ satisfies 
\begin{equation*}
\begin{cases}
\det(A_0(x))> 0,  &  x \in \Omega_+ \\
\det(A_0(x))< 0,  &  x \in \Omega_-
\end{cases},
\end{equation*}
for $\Omega_+$ and $\Omega_-$ satisfying $\Omega = \overline{\Omega_+} \cup \overline{\Omega_-}$ and $\Omega_+ \cap \Omega_- = \emptyset$. 
In this case, when $x$ is away from the interface, the behavior is similar to the first case. 
We derive a motion law for the interface at two time scales. 
At the $O(t)$ time scale, the interface evolves in the normal direction by curvature, as in the $n=1$ case; see Section~\ref{sec:part2fast}. 
At a slow $O(\varepsilon t)$ time scale, we show in Section~\ref{sec:part2slow} that the interface is driven by the surface diffusion of the matrix-valued field and the curvature. 
For $n=2$, in Proposition~\ref{p:SlowMotionLaw}, we show that the surface diffusion term can be written as the 
 jump in the squared tangental derivative of the phase across the interface.

In particular, we emphasize that the interface when $n\geq 2$ is driven by surface diffusion, while for $n=1$, the original Allen--Cahn equation, the interface is only driven by mean curvature. 
\end{itemize}
The results obtained via asymptotics are verified, supported, and illustrated in Section~\ref{sec:num} through a wide variety of  numerical experiments.

\medskip

The model \eqref{e:orig} considered in this paper can be viewed as a special case of the general model studied in \cite{Lin_2012}. In \cite{Lin_2012}, an energy of the form in \eqref{e:energy} is considered for high-dimensional, vector-valued functions and  general assumptions on the potential $W$. General results for the phase transition of stationary solutions between minima of $W$ are derived.  
In the present paper, we consider time dynamics for our specific model. 

\subsection*{Outline} The paper is organized as follows. 
In Section~\ref{sec:part1}, we derive the behavior of the matrix-valued field satisfying \eqref{e:orig} if the initial field, $A_0(x)$, only takes values in $SO_n$ or $SO_n^-$. 
In Section~\ref{sec:part2}, we discuss the case when an interface develops between subdomains where $\det(A(x,t)) > 0$ and $\det(A(x,t)) < 0$.  We develop a boundary layer around the interface and derive the motion of the interface at different time scales. 
Some numerical experiments are performed in Section~\ref{sec:num}. 
We conclude with a discussion in Section~\ref{sec:con}.

\section{Evolution of an initial matrix-valued field with single-signed determinant} \label{sec:part1}
In this section, we discuss the case where the initial matrix-valued field, $A_0(x)$, is continuous and has positive determinant at each point $x \in \Omega$. The case where $A_0(x)$ has negative determinant everywhere is analogous. 
In this case, there is no interface appearing in the dynamics of the system. We consider the asymptotic expansion 
$$
A = \bar A_0 + \varepsilon \bar A_1 + \varepsilon^2 \bar A_2 + o(\varepsilon^2) 
$$
and the initial condition $A(t,x)|_{t=0}  = A_0(x)$, which we assume to appear at the $O(1)$ scale.
Then, we expand the nonlinear term on the right hand side of \eqref{e:orig},
\begin{align}\label{eq:ansatz1}
\bar A(\bar A^t \bar A -I) & =  \bar A_0 (\bar A_0^t \bar A_0-I) + \varepsilon( \bar A_0 \bar A_0^t \bar A_1 + \bar A_0\bar A_1^t \bar A_0+ \bar A_1 \bar A_0^t \bar A_0 -  \bar A_1) \nonumber \\
& + \varepsilon^2 (\bar A_0 \bar A_0^t \bar A_2 + \bar A_0 \bar A_2^t \bar A_0 + \bar A_2 \bar A_0^t \bar A_0 +\bar A_0 \bar A_1^t \bar A_1+\bar A_1 \bar A_0^t \bar A_1+\bar A_1 \bar A_1^t \bar A_0 -\bar A_2) + o(\varepsilon^2).
\end{align}

We take two time scales: $t$ and a fast time scale $\tau_1 := \varepsilon^{-2} t$ and write 
$$
\bar A_i = \bar A_i(x,t,\tau_1), \qquad i =1,2.
$$
We have $\partial_t = \varepsilon^{-2}  \partial_{\tau_1}$ and
\begin{equation} \label{e:TimeExpansion}
\partial_t A = \varepsilon^{-2}  \partial_{\tau_1}\bar A_0+ \varepsilon^{-1}  \partial_{\tau_1}\bar A_1+  \partial_{\tau_1}\bar A_2+ \partial_t \bar A_0 + o(1). 
\end{equation}

We insert our ansatz into \eqref{e:orig} and collect terms at each order in $\varepsilon$.  Using \eqref{eq:ansatz1} and \eqref{e:TimeExpansion}  in \eqref{e:orig} yields
\begin{align} \label{e:ExpandedOrig}
& \varepsilon^{-2}  \partial_{\tau_1}\bar A_0+ \varepsilon^{-1}  \partial_{\tau_1}\bar A_1+ o(\varepsilon^{-1}) \\
\nonumber
& \quad =  \varepsilon^{-2}  \bar A_0 (\bar A_0^t \bar A_0-I)  + \varepsilon^{-1}( \bar A_0 \bar A_0^t \bar A_1 + \bar A_0\bar A_1^t \bar A_0+ \bar A_1 \bar A_0^t \bar A_0 -  \bar A_1) + o(\varepsilon^{-1}).
\end{align}

\subsection{Behavior at the $O(\varepsilon^{-2} t)$ time scale} \label{sec:part1fast}
Collecting $O(\varepsilon^{-2})$ terms in \eqref{e:ExpandedOrig} yields
\begin{equation}\label{eq:barA1}
 \partial_{\tau_1}\bar A_0 = \bar A_0 (\bar A_0^t \bar A_0-I).
 \end{equation}
 
For $n =1$, it is well known that the solution of \eqref{eq:barA1} approaches $1$ if the initial value is positive and approaches $-1$ if the initial value is negative as $\tau_1 \rightarrow \infty$.

For $n \geq 2$, at each point $x\in \Omega$, as $\tau_1 \rightarrow \infty$, the solution of \eqref{eq:barA1} approaches a matrix in $SO_n$ if the initial matrix has a positive determinant and approaches a matrix in $SO_n^-$ if the initial matrix has a negative determinant, as shown in the following Lemma.
\begin{lem} \label{l:DynSys}
For the dynamic system 
\begin{align*} 
& \partial_t B = B (B^t B-I) \\
& B(t=0) = B_0
\end{align*}
for a non-singular initial $n \times n$ matrix $B_0$ ($n \geq 2$), as $t \rightarrow \infty$, the solution $B(t)$ approaches the nearest orthogonal matrix to $B_0$, written $\Pi_{O_n} B_0$. 
\end{lem}
\begin{proof}
Write the singular-value decomposition of $B_0$ as $B_0 = U\Sigma_0 V^t$, where the diagonal values of $\Sigma_0$ are denoted by $\sigma_{0,i}$ ($i \in [n]$).
Since the right hand side of the equation can initially be written as $U(\Sigma_0\Sigma_0^t \Sigma_0 -\Sigma_0)V^t$, the solution $B(t)$ also admits a singular-value decomposition with the same $U$ and $V$. 
Then, we can write the dynamic system as 
\[U (\partial_t \Sigma) V^t  = U(\Sigma\Sigma^t \Sigma -\Sigma)V^t .\]
For each diagonal element $\sigma_i$ $(i \in [n])$ in $\Sigma$, we have $\partial_t \sigma_i = \sigma_i^3 -\sigma_i$. 
Since the $\sigma_{i}(0) > 0$ for $i \in [n]$, we have $\sigma_i(t) \rightarrow 1$ as $t \rightarrow \infty$. That implies, as $t \rightarrow \infty$, $B(t)$ approaches $UV^t$, which is the closest orthogonal matrix to $B_0$; see, {\it e.g.}, \cite[Lemma 1.1]{osting2017}
\end{proof}
 
Collecting $O(\varepsilon^{-1})$ terms in \eqref{e:ExpandedOrig} yields
\begin{equation}\label{eq:barA2}
 \partial_{\tau_1}\bar A_1 = \bar A_0 \bar A_0^t \bar A_1 + \bar A_0\bar A_1^t \bar A_0+ \bar A_1 \bar A_0^t \bar A_0 -  \bar A_1.
 \end{equation}
 Since $\bar A_1(0,x) = 0$, we have 
$$
\bar A_1(\tau_1,x) = 0
$$ 
is the solution to \eqref{eq:barA2}.

\subsubsection{Summary of the behavior at the $O(\varepsilon^{-2} t)$ time scale} \label{sec:sumpart1fast}
\begin{enumerate}
\item If the determinant of the initial matrix-valued field is positive for all $x \in \Omega$, the leading order matrix $\bar A_0(\tau_1,x)$ at each point approaches the closest orthogonal matrix in $SO_n$. If the determinant of the initial matrix-valued field is negative for all $x \in \Omega$, the leading order matrix $\bar A_0(\tau_1,x)$ at each point will approach the closest orthogonal matrix in $SO_n^-$. 
\item The second order matrix $\bar A_1(\tau_1,x)$ is $0$ for any $\tau_1 \geq 0$ and $x \in \Omega$.
\end{enumerate}

\subsection{Behavior at the $O(t)$ time scale} \label{sec:part1slow}
Using \eqref{eq:ansatz1} and \eqref{e:TimeExpansion} in \eqref{e:orig},
 we have at the time scale $O(t)$, 
\begin{align}
\nonumber
\partial_t \bar A_0 + o(1) & =  \varepsilon^{-2}  \bar A_0 (\bar A_0^t \bar A_0-I)  + \varepsilon^{-1}( \bar A_0 \bar A_0^t \bar A_1 + \bar A_0\bar A_1^t \bar A_0+ \bar A_1 \bar A_0^t \bar A_0 -  \bar A_1) \\
\label{e:ExpandedOrig2}
& \  + (\bar A_0 \bar A_0^t \bar A_2 + \bar A_0 \bar A_2^t \bar A_0 + \bar A_2 \bar A_0^t \bar A_0 +\bar A_0 \bar A_1^t \bar A_1+\bar A_1 \bar A_0^t \bar A_1+\bar A_1 \bar A_1^t \bar A_0 - \bar A_2 +\Delta \bar A_0) +o(1).
\end{align}

Collecting $O(\varepsilon^{-2})$ terms in \eqref{e:ExpandedOrig2} yields
$$
\bar A_0(\bar A_0^t \bar A_0 - I) = 0. 
$$
Since $\bar A_0$ is non-singular, this implies that 
\begin{equation} \label{e:A0A0}
\bar A_0^t \bar A_0 = I = \bar A_0 \bar A_0^t 
\end{equation}
This means that at each point $x\in \Omega$, for any non-singular initial matrix field, the leading order $\bar A_0$ immediately approaches a matrix field with values in $O_n$. This can be interpreted as the long time behavior of the dynamics at the time scale $O(\varepsilon^{-2} t)$.

\medskip

Collecting $O(\varepsilon^{-1})$ terms in \eqref{e:ExpandedOrig2}, we obtain 
$$\bar A_0 \bar A_0^t \bar A_1 + \bar A_0\bar A_1^t \bar A_0+ \bar A_1 \bar A_0^t \bar A_0 -  \bar A_1 = 0. $$
This is also consistent with the solution $\bar A_1(t,x) = 0$ at the $O(\varepsilon^{-2} t)$ time scale.

%Using \eqref{e:A0A0}, this simplifies to 
%\begin{equation} \label{e:t1}
%\bar A_0  \bar A_1^t \bar A_0 = -  \bar A_1. 
%\end{equation}
%Multiplying \eqref{e:t1} on the left and the right by $A_0^t$,  we obtain 
%\begin{subequations} \label{e:A0A1}
%\begin{align}
%&  \bar A_1^t \bar A_0 = - \bar A_0^t \bar A_1 \\
%& \bar A_0 \bar A_1^t  = - \bar A_1 \bar A_0^t. 
%\end{align}
%\end{subequations}
%Using  \eqref{e:t1} twice, we compute 
%\begin{equation} \label{e:t2}
%\bar  A_1 \bar A_1^t = \bar A_0 \bar A_1^t\bar  A_1 \bar A_0^t. 
%\end{equation}
%
%On the other hand, we similarly compute 
%$$
% \bar A_1^t \bar A_1 = \bar A_0^t \bar A_1 \bar A_1^t \bar A_0. 
%$$
%Multiplying on the left by $\bar A_0$ and on the right by $\bar A_0^t$, we have 
%\begin{equation} \label{e:t3}
%\bar A_0 \bar A_1^t \bar A_1 \bar A_0^t = \bar A_1 \bar A_1^t. 
%\end{equation}
%Together, \eqref{e:t2} and \eqref{e:t3} impliy that
%\begin{equation} \label{e:A1A1}
%\bar A_1^t  \bar A_1 = 0  = \bar A_1 \bar A_1^t. 
%\end{equation}
%
%\subsection*{Collecting $O(\varepsilon^{0})$ terms}

\medskip

At $O(1)$ in \eqref{e:ExpandedOrig2}, we obtain
\begin{subequations} \label{e:O1}
\begin{align}
\partial_t \bar A_0 &= \Delta \bar A_0 - \bar A_2 (\bar A_0^t \bar A_0 - I) - \bar A_1 (\bar A_0^t \bar A_1 + \bar A_1^t \bar A_0) - \bar A_0 ( \bar A_0^t \bar A_2 + \bar A_2^t \bar A_0 + \bar A_1^t \bar A_1) \\
\label{e:O1b}
&= \Delta \bar A_0 - (\bar A_2 + \bar A_0 \bar A_2^t \bar A_0), 
\end{align}
\end{subequations}
where we have used the fact $\bar A_1 = 0$. 
We now take the derivative of \eqref{e:A0A0} and use \eqref{e:O1b} to obtain
\begin{align*}
0 &= \bar A_0^t (\partial_t \bar A_0 ) + (\partial_t \bar A_0)^t \bar A_0 \\
&= \bar A_0^t (\Delta \bar A_0 )  -  \bar A_0^t (\bar A_2 + \bar A_0 \bar A_2^t \bar A_0 )  + (\Delta \bar A_0 )^t\bar  A_0 - (\bar A_2^t + \bar A_0^t \bar A_2 \bar A_0^t ) \bar A_0. 
\end{align*}
Using \eqref{e:A0A0} and rearranging, we obtain 
$$
\bar A_0^t \bar A_2 + \bar A_2^t \bar A_0  = \frac{1}{2} \left(\bar  A_0^t (\Delta \bar A_0 )  + (\Delta \bar A_0 )^t \bar A_0  \right). 
$$
Multiplying on the left by $\bar A_0$ and using \eqref{e:A0A0}, we insert this back into \eqref{e:O1b} to obtain
\begin{subequations}  \label{e:dtA0}
\begin{align} \label{e:dtA0a}
\partial_t \bar A_0 &= \frac{1}{2} \Delta \bar A_0 - \frac{1}{2}\bar  A_0 ( \Delta \bar A_0 )^t \bar A_0 \\ 
\label{e:dtA0b}
&= \frac{1}{2} \left(  \Delta \bar A_0 \bar A_0^t - \bar  A_0 ( \Delta \bar A_0 )^t \right) \bar A_0
\end{align}
\end{subequations}
The following proposition shows that if initially $\partial_t \bar A_0$ is in $SO_n$ pointwise, then it will remain there for all time $t>0$. 
\begin{prop} \label{p:OnDiffEq}
We consider the initial value problem  for the $O_n$ diffusion equation,
\begin{subequations} \label{e:OnDiffEq}
\begin{align} 
\label{e:OnDiffEqa}
\partial_t B(t,x) &=  \frac{1}{2} \left( \Delta B(t,x) B^t(t,x) - B(t,x)   \Delta B^t (t,x) \right) B(t,x)   \\
B(0,x) &= B_0(x),
\end{align}
\end{subequations}
where $B_0 \colon \Omega \to O_n$ is given. 
Then $B(t,x) \in O_n$ for all $t\geq 0$. 
\end{prop}
\begin{proof}
We compute 
\begin{align*}
\partial_t (B^tB) &= (\partial_t B)^t B + B^t (\partial_t B) \\ 
&= \frac{1}{2}\left(\Delta B(t,x) \right)^t B(t,x)  -  \frac{1}{2} B^t(t,x)  \left( \Delta B(t,x) \right) B^t(t,x)  B(t,x) \\
 & \quad +  \frac{1}{2} B^t(t,x) \left( \Delta B(t,x) \right) -  \frac{1}{2} B^t(t,x)B(t,x)  \left( \Delta B(t,x) \right)^t B(t,x)   \\ 
& = 0. 
\end{align*}
\end{proof}

\begin{rem}
We refer to \eqref{e:dtA0} and \eqref{e:OnDiffEq} as the \emph{$O_n$ diffusion equation} because it can be obtained from the diffusion equation for a matrix-valued field when the matrix is constrained to be $O_n$-valued. That is, introducing a Lagrange multiplier for the constraint and then solving for it yields precisely this equation. 
\end{rem}
\begin{rem}
For $n=2$ and the ansatz, 
$$
B(x,t) = \begin{pmatrix} \cos \eta(x,t) & - \sin  \eta (x,t) \\ \sin \eta(x,t) & \cos \eta(x,t) \end{pmatrix},  
$$
we compute 
$$
\partial_t B(x,t) = - \begin{pmatrix} \sin \eta(x) &  \cos  \eta (x) \\ -\cos \eta(x) & \sin \eta(x) \end{pmatrix} \partial_t \eta
\quad \textrm{and} \quad 
\Delta B - B (\Delta B)^t B = - 2 \begin{pmatrix} \sin \eta(x) &  \cos  \eta (x) \\ -\cos \eta(x) & \sin \eta(x) \end{pmatrix} \Delta \eta. 
$$
We conclude that $B(x,t)$ satisfies the orthogonal diffusion equation \eqref{e:OnDiffEq} if $\eta_t = \Delta \eta$. 
The \emph{spherical diffusion equation} is given by 
$$
\phi_t = \Delta \phi + |\nabla \phi |^2 \phi, 
$$ 
see, {\it e.g.}, \cite{E_2000}. Making the ansatz $\phi(x,t) = e^{i \eta(t,x)}$, we find that $\eta_t = \Delta \eta$.  
Thus, we conclude that the $n=2$ orthogonal diffusion equation \eqref{e:OnDiffEq} with initial condition taking values in $SO_n$ is equivalent to the  spherical diffusion equation. Due to this connection, we refer to $\eta = \eta(x,t)$ as the \emph{phase} of the matrix-valued field, $B(x,t)$.  
\end{rem}

\subsubsection{Summary of the behavior at the $O(t)$ time scale:} \label{sec:sumpart1slow}
\begin{enumerate}
\item The leading order solution $\bar A_0$  take values in the orthogonal matrix group for all time.
\item The second order solution is $\bar A_1 = 0$.
\item The time dynamics of $\bar A_0$ is governed by the $O_n$ diffusion equation, 
\begin{align*} 
\partial_t \bar A_0(t,x) &=  \frac{1}{2} \left( \Delta \bar A_0(t,x) \bar A_0^t(t,x) - \bar A_0(t,x) \Delta \bar A_0^t(t,x)  \right) \bar A_0(t,x),   \\
\bar A_0(0,x) &= \Pi_{O_n} A_0(x).
\end{align*}
Here the initial condition is pointwise the closest point to $A_0(x)$ in $O_n$. 
\end{enumerate}

\subsection{Harmonic orthogonal matrix-valued fields} \label{s:Harmonic}
In this section, we consider stationary solutions of the $O_n$ diffusion equation \eqref{e:OnDiffEq}, which we refer to as \emph{harmonic orthogonal matrix-valued fields}, satisfying 
\begin{subequations} \label{e:Harmonic}
\begin{align}  \label{e:Harmonica}
& (\Delta B)^t B - B^t (\Delta B)  = 0 \\ 
& B^tB = I
\end{align}
\end{subequations}
Note that \eqref{e:Harmonica} just states that $ (\Delta B)^t B$ is a symmetric matrix.

For $n=1$, the only solutions are $B=\pm 1$. 

For $n=2$, for unknown phase $\eta\colon \Omega \to \mathbb R$, we consider the $SO_n$ ansatz, 
\begin{equation}\label{eq:matrixb}
B(x) = \begin{pmatrix} \cos \eta(x) & - \sin  \eta (x) \\ \sin \eta(x) & \cos \eta(x) \end{pmatrix}. 
\end{equation}
We compute 
$$
\Delta B  = \begin{pmatrix} 
- \sin \eta \Delta \eta - \cos \eta |\nabla \eta|^2 & - \cos \eta \Delta \eta + \sin \eta |\nabla \eta|^2 \\ 
\cos \eta \Delta \eta - \sin \eta |\nabla \eta |^2 & - \sin \eta \Delta \eta - \cos \eta |\nabla \eta |^2 
\end{pmatrix}
\quad \textrm{and} \quad 
(\Delta B)^t B  = \begin{pmatrix} - | \nabla \eta |^2 & \Delta \eta \\  - \Delta \eta & - |\nabla \eta|^2 \end{pmatrix}. 
$$
We observe that $(\Delta B)^t B$ is symmetric if and only if  $\Delta \eta = 0$. We conclude that there exists a family of harmonic $SO_n$-valued fields on the torus of the form \eqref{eq:matrixb} where the phase is given by
\begin{equation} \label{e:HarmSol}
\eta(x_1,x_2) = 2\pi(n_1 x_1+n_2 x_2), \qquad \qquad n_1,n_2 \in \mathbb Z.
\end{equation}
Several numerical experiments are performed in Section~\ref{sec:single} to show that such fields are stationary for \eqref{e:orig} and to investigate what happens if perturbations of such fields are taken as initial conditions.

\section{Evolution of an $O_n$-valued initial field}\label{sec:part2}
We consider an initial condition $A_0(x)$ that satisfies 
\begin{equation}
\begin{cases}
\det(A_0(x))> 0,  &  x \in \Omega_+ \\
\det(A_0(x))< 0,  &  x \in \Omega_-
\end{cases}
\end{equation}
for $\Omega_+$ and $\Omega_-$ satisfying $\Omega = \overline{\Omega_+} \cup \overline{\Omega_-}$ and $\Omega_+ \cap \Omega_- = \emptyset$. 

Denote $\Gamma = \bar \Omega_+ \cap \bar\Omega_-$ and assume $\Gamma(t)$ is a finite collection of simple, closed, smooth curves in $\mathbb{R}^2$ so that we can find a parametric representation, at least locally, of the form
\[\Gamma(t) = \{\vec \varphi(s,t)\colon s\in \mathbb{R}^1\}, \quad \vec \varphi(s,t) = (\varphi_1(s,t),\varphi_2(s,t) ).\]
We assume $s$ to be the arc-length parameter so that we have 
\[\vec T = \frac{\partial \vec\varphi}{\partial s}, \quad \frac{\partial \vec T}{\partial s} = - \kappa \vec n, \quad \frac{\partial \vec n}{\partial s} = \kappa \vec T\]
where $\vec T$ denotes the unit tangent vector, $\vec n$ denotes the unit outer normal vector, and $\kappa$ denotes the curvature (see Figure~\ref{fig:diagram1} for a diagram on the unit outer normal vector).
\begin{figure}[ht]
\includegraphics[width = 0.4 \textwidth]{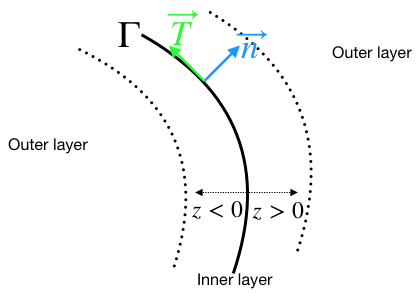}
\caption{Diagram for the unit normal vector and unit tangent vector for $\Gamma$. See Section~\ref{sec:part2}.}\label{fig:diagram1}
\end{figure}

We introduce local coordinates near $\Gamma$ as follows. 
We assume that for every point $x$ in a neighborhood of $\Gamma$, there is a unique point $\vec\varphi (s,t)$ which is the orthogonal projection of $x$ onto $\Gamma(t)$. We then define a unique normal signed distance $\rho(x,t)$ from $x$ to $\Gamma(t)$, $\rho(x,t) = (x - \vec \varphi) \cdot \vec n$.
We have a transformation from $(x,t)$ to $(s,r,t)$ defined by 
\begin{equation}\label{eq:stox}
x = \vec \varphi(s,t) + r \vec n(s,t)
\end{equation}
where $r = \rho(x,t)$.
We summarize several identities for the transformation from $(x,t)$ to $(s,r,t)$ in the following Lemma. A proof of this Lemma can be found  in \cite{Dai_2012} . 
\begin{lem}
For the transformation rule defined in \eqref{eq:stox}, we have the following equalities:\\
1.The normal velocity of $\Gamma$ at $\vec \varphi(s,t)$ is given by  $\vec{v} \cdot \vec n = -\frac{\partial \rho}{\partial t}$. \\
2. $\nabla_x s = \frac{1}{1+r\kappa} \vec T$, $\Delta_x s = -r \frac{\partial \kappa}{\partial s} \frac{1}{(1+r\kappa)^3}$.\\
3. $\nabla_x r = \vec n$, $\Delta_x r = \frac{\kappa}{1+r \kappa}$.\\
4. For any function $u(x,t) = \tilde u(s,r,t)$, 
\[\Delta_x u = \frac{1}{(1+r\kappa)^2} \frac{\partial^2 \tilde u}{\partial s^2} + \frac{\partial^2 \tilde u}{\partial r^2} +\frac{\kappa}{1+r \kappa} \frac{\partial \tilde u}{\partial r} -\left(\frac{r}{(1+r\kappa)^3} \frac{\partial \kappa}{\partial s} \right)\frac{\partial \tilde u}{\partial s} . \]
\end{lem}

Below, we study the inner layer expansion to study the behavior around the interface at the time scales $O(t)$ and $O(\varepsilon t)$.

\subsection{Behavior at the $O(t)$ time scale.} \label{sec:part2fast}
At the $O(t)$ time scale, when $x$ is away from the interface $\Gamma$, the behavior exactly reduces to the case studied in Section~\ref{sec:part1slow}. That is corresponding to the outer layer expansion for the system \eqref{e:orig}. We don't repeat the calculation
and refer the results to the summary in Section~\ref{sec:sumpart1slow}.

The inner expansion requires rescaling the normal coordinate by $z = \frac{r}{\varepsilon}$. Assume the expansion of $\vec \varphi(s,t)$ and $\kappa$ are
\begin{align}
&\vec \varphi(s,t) = \vec \varphi_0(s,t) +\varepsilon \vec \varphi_1(s,t) + \varepsilon^2 \vec \varphi_2(s,t) + o(\varepsilon^2), \\
&\kappa = \kappa_0 + \varepsilon \kappa_1 +\varepsilon^2 \kappa_2+ o(\varepsilon^2) \label{ex:kappa}.
\end{align}
Writing $\Delta_x$ in terms of $(s,z)$  using \eqref{ex:kappa} yields
\begin{equation} \label{ex:Delta}
\Delta_x = \varepsilon^{-2} \partial_{zz} +\varepsilon^{-1} \kappa_0 \partial_z + (\partial_{ss} - (z\kappa_0^2-\kappa_1) \partial_z)+ o(1).
\end{equation}
Consider the expansion 
\begin{equation}\label{eq:tscale}
A(x,t) = \tilde A(s,z,t) = \tilde A_0(s,z,t)+\varepsilon  \tilde A_1(s,z,t) +  \varepsilon^2 \tilde A_2(s,z,t) + o(\varepsilon^2)\end{equation}
and write
\begin{equation}\label{eq:tderi}
\partial_t A(x,t) = \partial_t   \tilde A(s,z,t)  + \frac{\partial s}{\partial t} \partial_s  \tilde A(s,z,t)  + \varepsilon^{-1} \frac{\partial r}{\partial t} \partial_z  \tilde A(s,z,t) .
\end{equation}
Similar to \eqref{eq:ansatz1}, we have  \begin{align}\label{eq:ansatz2}
\tilde A(\tilde A^t \tilde A -I) & =  \tilde A_0 (\tilde A_0^t \tilde A_0-I) + \varepsilon( \tilde A_0 \tilde A_0^t \tilde A_1 + \tilde A_0\tilde A_1^t \tilde A_0+ \tilde A_1 \tilde A_0^t \tilde A_0 -  \tilde A_1) \nonumber \\
&  \quad + \varepsilon^2 (\tilde A_0 \tilde A_0^t \tilde A_2 + \tilde A_0 \tilde A_2^t \tilde A_0 + \tilde A_2 \tilde A_0^t \tilde A_0 +\tilde A_0 \tilde A_1^t \tilde A_1+\tilde A_1 \tilde A_0^t \tilde A_1+\tilde A_1 \tilde A_1^t \tilde A_0 -\tilde A_2) + o(\varepsilon^2).
\end{align}
Substituting \eqref{ex:Delta}, \eqref{eq:tscale}, \eqref{eq:tderi}, and \eqref{eq:ansatz2} into \eqref{e:orig} yields 
\begin{align}
\label{e:OutExpand}
 &\varepsilon^{-1} \frac{\partial r}{\partial t} \partial_z \tilde A_0 +o(\varepsilon^{-1}) \\
\nonumber
  & =   \varepsilon^{-2} \left(\partial_{zz} \tilde A_0 - \tilde A_0 (\tilde A_0^t \tilde A_0-I) \right) \\
\nonumber
  & \quad + \varepsilon^{-1} \left( \kappa_0 \partial_z \tilde A_0 + \partial_{zz} \tilde A_1 - ( \tilde A_0 \tilde A_0^t \tilde A_1 + \tilde A_0\tilde A_1^t \tilde A_0+ \tilde A_1 \tilde A_0^t \tilde A_0 -  \tilde A_1) \right) + o(\varepsilon^{-1})
\end{align}

\medskip

Collecting the $O(\varepsilon^{-2})$ terms in \eqref{e:OutExpand}, we obtain
\begin{align}\label{eq:part2slowleading}
\partial_{zz} \tilde A_0 - \tilde A_0 (\tilde A_0^t \tilde A_0-I)= 0.
\end{align}
Matching the outer expansion gives the boundary conditions 
\begin{equation} \label{eq:part2slowleadingBC}
\lim_{z \to \pm \infty} \tilde A_0 = \bar A_0,
\end{equation}
where $\bar A_0 \in O_n$. 
Note that \eqref{eq:part2slowleading} is independent of $s$ and thus we solve  \eqref{eq:part2slowleading} for each $s$ independently.  
For $n=2$, the following proposition gives the explicit solution to \eqref{eq:part2slowleading} with the boundary conditions in \eqref{eq:part2slowleadingBC}. 

\begin{prop} \label{p:MatrixTanhProfile}
The solution to the second-order differential equation for the $2\times 2$ matrix field, $B\colon \mathbb R \to M(2)$
\begin{align*}
& \frac{d^2 B}{dz^2} = B (B^tB - I) \\
& \lim_{z\to - \infty} B(z) = \begin{bmatrix} \cos( \eta_-)   & \sin(\eta_-) \\  \sin(\eta_-)   & -\cos(\eta_-)  \end{bmatrix} \in SO_2^- \\
& \lim_{z\to + \infty} B(z) = \begin{bmatrix} \cos( \eta_+)   & -\sin(\eta_+) \\  \sin(\eta_+)   & \cos(\eta_+)  \end{bmatrix}  \in SO_2 
\end{align*}
is given by 
$$
B(z) =  
\begin{bmatrix} \cos(\xi_1)   & -\sin(\xi_1) \\  \sin(\xi_1)   & \cos(\xi_1)  \end{bmatrix} 
\begin{bmatrix} 1 & 0 \\ 0 & \tanh \left( \frac{z}{\sqrt{2}} \right) \end{bmatrix} 
 \begin{bmatrix} \cos(\xi_2)   & -\sin(\xi_2) \\  \sin(\xi_2)   & \cos(\xi_2)    \end{bmatrix}^t,
$$
where $\xi_1 = \frac{\eta_- +\eta_+}{2}$ and $\xi_2 = \frac{\eta_- -\eta_+}{2}$.
\end{prop}

\begin{proof} The proposed solution takes the form 
$B(z) =  U_1 D U_2^t$, 
where $U_i$ for $i=1,2$ are matrices in $SO_2$ that are independent of $z$ 
and $D=\begin{bmatrix} 1 & 0 \\ 0 & \tanh \left( \frac{z}{\sqrt{2}} \right) \end{bmatrix} $ is a diagonal matrix. 
Since $\sigma(z) = \tanh\left(\frac{z}{\sqrt{2}} \right)$ satisfies $ \partial_{zz} \sigma  = \sigma^3 -\sigma$, we have that $B(z)$ satisfies the differential equation.  Using the  two trigonometric identities, 
\begin{align*}
\cos \eta_{\pm} &= \cos \xi_1 \cos \xi_2 \mp \sin \xi_1 \sin \xi_2 \\
\sin \eta_{\pm} &= \sin \xi_1 \cos \xi_2 \pm \cos \xi_1 \sin \xi_2, 
\end{align*}
 $B(z)$ satisfies the boundary conditions as $z \to \pm \infty$. 
%
%When $n =2$, assume the boundary conditions are 
%\begin{equation}\label{bdc}
%\tilde A_0|_{z = -\infty} = \begin{bmatrix} \cos(\alpha)   & \sin(\alpha) \\  \sin(\alpha)   & -\cos(\alpha)  \end{bmatrix}, \ \ 
%\tilde A_0|_{z = \infty} = \begin{bmatrix} \cos(\beta)   & -\sin(\beta) \\  \sin(\beta)   & \cos(\beta)  \end{bmatrix} 
%\end{equation}
%where $\alpha$ and $\beta$ are fixed for each given $s$.  Note that $\tilde A_0(-\infty) \in SO_n^-$ and $\tilde A_0(\infty) \in SO_n$.
%
%For such boundary conditions, we choose 
%\begin{equation}\label{UVt}
%U = \begin{bmatrix} \cos(\xi_1)   & -\sin(\xi_1) \\  \sin(\xi_1)   & \cos(\xi_1)  \end{bmatrix}, \ \ 
%V^t = \begin{bmatrix} \cos(\xi_2)   & \sin(\xi_2) \\  -\sin(\xi_2)   & \cos(\xi_2)    \end{bmatrix}, \ \ 
%\Sigma(z) =  \begin{bmatrix}1   & 0 \\  0   & \sigma(z)  \end{bmatrix}
%\end{equation}
%where $\xi_1 = \frac{\alpha+\beta}{2}$ and $\xi_2 = \frac{\alpha-\beta}{2}$.
%
%Then, direct calculations imply that \eqref{eq:part2slowleading} coupled with boundary conditions \eqref{bdc} are reduced to 
%\begin{equation}
%\begin{cases}
%\partial_{zz} \sigma  = \sigma^3 -\sigma, \\
%\sigma|_{z = -\infty} = -1,\\
%\sigma|_{z = \infty} = 1
%\end{cases}
%\end{equation}
%which gives that $\sigma(z) = \tanh(\frac{z}{\sqrt{2}})$. 
%
\end{proof}

Hence, for $n = 2$, by Proposition~\ref{p:MatrixTanhProfile}, we explicitly get the solution for \eqref{eq:part2slowleading} coupled with boundary conditions \eqref{eq:part2slowleadingBC},
\begin{equation} \label{eq:prof}
\tilde A_0(s,z,t) = 
\begin{bmatrix} \cos(\xi_1)   & -\sin(\xi_1) \\  \sin(\xi_1)   & \cos(\xi_1)  \end{bmatrix} 
\begin{bmatrix} 1 & 0 \\ 0 & \tanh(\frac{z}{\sqrt{2}}) \end{bmatrix} 
 \begin{bmatrix} \cos(\xi_2)   & -\sin(\xi_2) \\  \sin(\xi_2)   & \cos(\xi_2)    \end{bmatrix}^t,
 \end{equation}
where 
 $\xi_1 = \xi_1(s,t)$ and $\xi_2 = \xi_2(s,t)$ are determined from the phase of the outer solution, $\bar A_0(x,t)$, for each $s$ and $t$.

\medskip

Collecting the $O(\varepsilon^{-1})$ terms in \eqref{e:OutExpand}  yields
\begin{align}\label{eq:part2slowsecond1}
\frac{\partial r}{\partial t} \partial_z \tilde A_0 = \kappa_0 \partial_z \tilde A_0 + \partial_{zz} \tilde A_1 - ( \tilde A_0 \tilde A_0^t \tilde A_1 + \tilde A_0\tilde A_1^t \tilde A_0+ \tilde A_1 \tilde A_0^t \tilde A_0 -  \tilde A_1)
\end{align}
Taking the Frobenius inner product with $\partial_z \tilde A_0$ on both sides of  \eqref{eq:part2slowsecond1} yields
$$
\left( \frac{\partial r}{\partial t} -\kappa_0 \right) \langle \partial_z \tilde A_0, \partial_z \tilde A_0\rangle_F  = \langle \partial_{zz} \tilde A_1 , \partial_z \tilde A_0\rangle_F- \langle  \tilde A_0 \tilde A_0^t \tilde A_1 + \tilde A_0\tilde A_1^t \tilde A_0+ \tilde A_1 \tilde A_0^t \tilde A_0 -  \tilde A_1, \partial_z \tilde A_0 \rangle_F.
$$
We integrate the above equation with respect to $z$ from $-\infty$ to $\infty$. Integrating by parts, we can rewrite the first term on the right hand side as 
\begin{align*}
\int_{-\infty}^{\infty}\langle \partial_{zz} \tilde A_1 ,\partial_z \tilde A_0 \rangle_F dz = \int_{-\infty}^{\infty} \langle \tilde A_1 ,\partial_{zzz} \tilde A_0 \rangle_F dz
\end{align*}
where the boundary terms vanish because  
$\partial_z \tilde A_0 = 0 $ at $z = \pm \infty$   from the solution to the  leading order expansion and 
 $\bar A_1 = 0$ in the outer layer which corresponds to $\tilde A_1 = 0$ at  $z = \pm \infty$ from the asymptotic matching.
The second term on the right hand side can be rewritten as 
\[\int_{-\infty}^{\infty} \langle  \tilde A_0 \tilde A_0^t \tilde A_1 + \tilde A_0\tilde A_1^t \tilde A_0+ \tilde A_1 \tilde A_0^t \tilde A_0 -  \tilde A_1, \partial_z \tilde A_0 \rangle_F dz = \int_{-\infty}^{\infty} \langle \tilde A_1, \partial_z (\tilde A_0\tilde A_0^t \tilde A_0-\tilde A_0)\rangle_F dz\]
where the boundary terms vanish because of $\bar A_1 = 0$ in the outer layer.
We then obtain 
\begin{equation} \label{eq:part2slowsecond2}
\left(\frac{\partial r}{\partial t} -\kappa_0 \right) \int_{-\infty}^{\infty}\langle  \partial_z \tilde A_0, \partial_z \tilde A_0 \rangle_F dz 
=   
\int_{-\infty}^{\infty} \left\langle \tilde A_1, \partial_{z}\left(\partial_{zz} \tilde A_0- (\tilde A_0\tilde A_0^t \tilde A_0-\tilde A_0)\right) \right\rangle_F dz. 
\end{equation}
Denote \begin{equation} \label{gamma}
\gamma(s) := \int_{-\infty}^{\infty}\|\partial_z \tilde A_0\|_F^2 \ dz.
\end{equation}
We note that when $n = 1$, $\gamma$ is the surface tension on the interface between two different phases \cite{Pego_1989}.
Since we assume $\tilde A_0 \in SO_n$ on  one side of $\Gamma(t)$ and $\tilde A_0 \in SO_n^-$ on the other, we have $\gamma(s) >0 $. 
Using \eqref{eq:part2slowleading}, the right hand side of \eqref{eq:part2slowsecond2} vanishes and we have the normal velocity is given by 
\[ \vec{v} \cdot \vec n  =  -\kappa_0 . \] 
It follows that the interface evolves according to the mean curvature flow along its normal direction at the $O(t)$ time scale.

\subsubsection{Summary of the behavior near the interface at  the $O(t)$ time scale:} \label{sec:sumpart2fast}
\begin{enumerate}
\item The leading order solution, $\tilde A_0$, transitions from a matrix in $SO_n^-$ to a matrix in $SO_n$ in the boundary layer of the interface. 
It satisfies \eqref{eq:part2slowleading} with the boundary conditions in \eqref{eq:part2slowleadingBC}. 
For $n=2$, the leading order solution is explicitly given by \eqref{eq:prof}. 

\item The interface moves in the normal direction by the leading order of  curvature, {\it i.e.}, $\vec{v} \cdot \vec n  =  -\kappa_0$.
\end{enumerate}

\subsection{Behavior at the $O(\varepsilon t)$ time scale.}\label{sec:part2slow} 
Now, we study the behavior at the time scale $O(\tau_2)$ where $\tau_2 = \varepsilon t$. Then, we have $\partial_t = \varepsilon \partial_{\tau_2}$.

First, we consider the behavior when $x$ is away from the interface $\Gamma$.  Consider the expansion 
\begin{equation}\label{eq:tauscaleout}
A(x,t) = \bar A_0(x,\tau_2)+\varepsilon  \bar A_1(x,\tau_2) +  \varepsilon^2 \bar A_2(x,\tau_2) + o(\varepsilon^2)\end{equation}
and insert it into \eqref{e:orig} to obtain  
\begin{align*}
o(1) & =  \varepsilon^{-2}  \bar A_0 (\bar A_0^t \bar A_0-I)  + \varepsilon^{-1}( \bar A_0 \bar A_0^t \bar A_1 + \bar A_0\bar A_1^t \bar A_0+ \bar A_1 \bar A_0^t \bar A_0 -  \bar A_1) \\
& \quad + (\bar A_0 \bar A_0^t \bar A_2 + \bar A_0 \bar A_2^t \bar A_0 + \bar A_2 \bar A_0^t \bar A_0 +\bar A_0 \bar A_1^t \bar A_1+\bar A_1 \bar A_0^t \bar A_1+\bar A_1 \bar A_1^t \bar A_0 - \bar A_2 +\Delta \bar A_0) +o(1).
\end{align*}
Collecting the terms at different orders of $\varepsilon$ and using the behavior at the $O(t)$ time scale, we obtain that at the $O(\tau_2)$ time scale,
\begin{equation} \label{epstouter}
 (\Delta \bar A_0)^t \bar A_0 - \bar A_0^t (\Delta \bar A_0)  = 0, \quad
\bar A_0^t \bar A_0 = I, \quad
\bar A_1 = 0, \ \ \textrm{and} \ \
\bar A_2 = 0.
\end{equation}
For $n=2$,  as shown in Section~\ref{s:Harmonic}, there is a $O_2$ harmonic leading order matrix field, $\bar A_0$ of the form 
\begin{equation} \label{barA0epst}
\bar A_0(x) = \begin{bmatrix} \cos(\eta(x)) & \mp \sin(\eta(x)) \\  \sin(\eta(x)) & \pm \cos(\eta(x))   \end{bmatrix}
\end{equation} 
where the phase $\eta(x)$ satisfies $\Delta \eta =0$. Here, the signs in the second column are chosen depending on whether $x \in \Omega_+$ or $x \in \Omega_-$. 

\medskip

To study the behavior near the interface, we consider the expansion 
\begin{equation}\label{eq:tauscale}
A(x,t) = \tilde A(s,z,\tau_2) = \tilde A_0(s,z,\tau_2)+\varepsilon  \tilde A_1(s,z,\tau_2) +  \varepsilon^2 \tilde A_2(s,z,\tau_2) + o(\varepsilon^2)
\end{equation}
and rewrite \eqref{eq:tderi} using $\tau_2 = \varepsilon t$ to obtain 
\begin{equation}\label{eq:tauderi}
\partial_t A(x,t) = \varepsilon \partial_{\tau_2} \tilde A(s,z,\tau_2) + \varepsilon \frac{\partial s}{\partial \tau_2} \partial_s \tilde A(s,z,\tau_2) + \frac{\partial r}{\partial \tau_2} \partial_z \tilde A(s,z,\tau_2).
\end{equation}
Substituting \eqref{ex:Delta}, \eqref{eq:ansatz2}, \eqref{eq:tauderi}, and  \eqref{eq:tauscale} into \eqref{e:orig}, we have
\begin{align} \label{e:t2Inner}
 & \frac{\partial r}{\partial \tau_2} \partial_z \tilde A_0 +o(1) \\
 \nonumber
 &  =    \varepsilon^{-2} \left(\partial_{zz} \tilde A_0 - \tilde A_0 (\tilde A_0^t \tilde A_0-I) \right) + \varepsilon^{-1} \left( \kappa_0 \partial_z \tilde A_0 + \partial_{zz} \tilde A_1 - ( \tilde A_0 \tilde A_0^t \tilde A_1 + \tilde A_0\tilde A_1^t \tilde A_0+ \tilde A_1 \tilde A_0^t \tilde A_0 -  \tilde A_1) \right) \\
\nonumber
  & \quad +\left[\partial_{ss} \tilde A_0- (z\kappa_0^2-\kappa_1) \partial_z\tilde A_0+\kappa_0 \partial_z \tilde A_1+\partial_{zz} \tilde A_2 \right.\\
\nonumber
  & \quad - \left.(\tilde A_0 \tilde A_0^t \tilde A_2 + \tilde A_0 \tilde A_2^t \tilde A_0 + \tilde A_2 \tilde A_0^t \tilde A_0 +\tilde A_0 \tilde A_1^t \tilde A_1+\tilde A_1 \tilde A_0^t \tilde A_1+\tilde A_1 \tilde A_1^t \tilde A_0 -\tilde A_2)\right]+ o(1) 
\end{align}

Collecting the $O(\varepsilon^{-2})$ terms in \eqref{e:t2Inner} yields
\begin{align} \label{eq:tau2leading}
\partial_{zz} \tilde A_0 - \tilde A_0 (\tilde A_0^t \tilde A_0-I)= 0
\end{align}
which is same as  \eqref{eq:part2slowleading}. The boundary conditions at $z = \pm \infty$ are as in \eqref{eq:part2slowleadingBC}. Hence, $\tilde A_0(s,z,\tau_2)$ has the same transition profile as obtained from \eqref{eq:prof}  in the boundary layer.

\medskip

At $O(\varepsilon^{-1})$ in \eqref{e:t2Inner}, we have
\begin{equation}\label{eq:tau2second}
\kappa_0 \partial_z \tilde A_0 + \partial_{zz} \tilde A_1 - ( \tilde A_0 \tilde A_0^t \tilde A_1 + \tilde A_0\tilde A_1^t \tilde A_0+ \tilde A_1 \tilde A_0^t \tilde A_0 -  \tilde A_1)= 0.
\end{equation}
Taking the Frobenius inner product with $\partial_z \tilde A_0$ on both sides of  \eqref{eq:part2slowsecond1} and integrating both sides with respect to $z$ from $-\infty$ to $\infty$ yields:
\begin{align*}\kappa_0 \int_{-\infty}^{\infty}\langle  \partial_z A_0, \partial_z A_0 \rangle_F dz +   \int_{-\infty}^{\infty} \left\langle \tilde A_1, \partial_{z}\left(\partial_{zz} \tilde A_0- (\tilde A_0\tilde A_0^t \tilde A_0-\tilde A_0)\right) \right\rangle_F dz = 0.
\end{align*}
Using \eqref{eq:tau2leading} again leads us to 
\begin{equation}\label{eq:kappaleading}
\kappa_0 = 0.
\end{equation}
Inserting \eqref{eq:kappaleading} back to \eqref{eq:tau2second} yields
\[\partial_{zz} \tilde A_1 - ( \tilde A_0 \tilde A_0^t \tilde A_1 + \tilde A_0\tilde A_1^t \tilde A_0+ \tilde A_1 \tilde A_0^t \tilde A_0 -  \tilde A_1)= 0.\]
Coupling with the boundary conditions from the outer expansion: $ \tilde A_1(s, z, \tau_2) = 0 $ at $z = \pm \infty$ implies that 
\begin{equation} \label{eq:A1}
\tilde A_1(s, z,\tau_2) = 0 \ \ \textrm{for any}  \ \ z \in (-\infty, \infty).
 \end{equation}
 
\medskip

Collecting the $O(1)$ terms in \eqref{e:t2Inner}, we obtain
\begin{align}\label{eq:tau2third}
\frac{\partial r}{\partial \tau_2} \partial_z \tilde A_0 & = \partial_{ss} \tilde A_0- (z\kappa_0^2-\kappa_1) \partial_z\tilde A_0+\kappa_0 \partial_z \tilde A_1+\partial_{zz} \tilde A_2 \nonumber \\
& \quad -(\tilde A_0 \tilde A_0^t \tilde A_2 + \tilde A_0 \tilde A_2^t \tilde A_0 + \tilde A_2 \tilde A_0^t \tilde A_0 +\tilde A_0 \tilde A_1^t \tilde A_1+\tilde A_1 \tilde A_0^t \tilde A_1+\tilde A_1 \tilde A_1^t \tilde A_0 -\tilde A_2).
\end{align}
Using \eqref{eq:kappaleading} and \eqref{eq:A1} simplifies \eqref{eq:tau2third} to 
\begin{equation} \label{eq:tau2third2}
\frac{\partial r}{\partial \tau_2} \partial_z \tilde A_0 = \partial_{ss} \tilde A_0+\kappa_1 \partial_z\tilde A_0+\partial_{zz} \tilde A_2-(\tilde A_0 \tilde A_0^t \tilde A_2 + \tilde A_0 \tilde A_2^t \tilde A_0 + \tilde A_2 \tilde A_0^t \tilde A_0 -\tilde A_2).
\end{equation}
Taking the Frobenius inner product with $\partial_z \tilde A_0$ on both sides of  \eqref{eq:tau2third2} and integrating both sides with respect to $z$ from $-\infty$ to $\infty$ yields:
\begin{align*}
&(\frac{\partial r}{\partial \tau_2} -\kappa_1 )\int_{-\infty}^{\infty}\langle  \partial_z\tilde A_0, \partial_z\tilde A_0 \rangle_F dz - \int_{-\infty}^{\infty}\langle  \partial_{ss} \tilde A_0, \partial_z \tilde A_0 \rangle_F dz \\
& \quad =   \int_{-\infty}^{\infty} \left\langle \tilde A_2, \partial_{z}\left(\partial_{zz} \tilde A_0- (\tilde A_0\tilde A_0^t \tilde A_0-\tilde A_0)\right) \right\rangle_F dz. 
\end{align*}
Combining with \eqref{eq:tau2leading} and using the definition of $\gamma$ in \eqref{gamma} leads us to 
\begin{equation} \label{e:MotionLaw}
\vec{v} \cdot \vec{n} = - \kappa_1    - \frac{1}{\gamma(s)}\int_{-\infty}^{\infty}\langle  \partial_{ss} \tilde A_0, \partial_z \tilde A_0 \rangle_F \ dz, 
\end{equation}
which gives the motion law for the interface at the $O(\tau_2)$ time scale.

\begin{prop} \label{p:SlowMotionLaw}
For $n=2$,  $\tilde A_0$ admits the transition profile in \eqref{eq:prof} where 
 $\xi_1(s,t)  = \frac{\eta_- (s,t) +\eta_+(s,t)}{2}$ and 
 $\xi_2(s,t) =\frac{\eta_-(s,t) -\eta_+(s,t)}{2}$. 
 Here,  $\eta_+(s,t)$ and $\eta_-(s,t)$  are determined from the outer solution, $\bar A_0(x,t)$, as in  \eqref{barA0epst}. 
The motion law in \eqref{e:MotionLaw} simplifies to 
\begin{equation} \label{eq:motionlawslow}
\vec{v} \cdot \vec{n}  = - \kappa_1 + \frac{1}{\bar \gamma} \left[ \eta_s^2 \right]_\Gamma 
\end{equation}
where
$ \left[ \eta_s^2 \right]_\Gamma  = \left( 
\left( \partial_s \eta_+(s,t) \right)^2  - 
\left( \partial_s \eta_-(s,t) \right)^2 \right)
$ is the jump in the squared tangental derivative of the phase across the interface $\Gamma$ and $\bar \gamma = \int_{-\infty}^{\infty} \left(1-\tanh^2(z/\sqrt{2})\right)^2 \ dz$.
\end{prop}
\begin{proof}
In the transition profile \eqref{eq:prof}, we write 
$$
\tilde A_0(s,z,t) = 
U_1
D  
U_2^t 
 \qquad \textrm{where }
D = \begin{bmatrix} 1 & 0 \\ 0 & \sigma(z) \end{bmatrix}
\ \textrm{and} \
U_i =  \begin{bmatrix} \cos(\xi_i)   & -\sin(\xi_i) \\  \sin(\xi_i)   & \cos(\xi_i)  \end{bmatrix}, \ \ i =1,2. 
$$
Here we write $\sigma(z) = \tanh \left( \frac{z}{\sqrt{2}} \right)$. 
We compute 
\[ \partial_s U_i =  
\begin{bmatrix} -\sin(\xi_i) & -\cos (\xi_i) \\ \cos(\xi_i) & -\sin(\xi_i) \end{bmatrix}  \xi_{is}, 
\]
and
\[ \partial_s^2 U_i =   
\begin{bmatrix} -\sin(\xi_i) & -\cos (\xi_i) \\ \cos(\xi_i) & -\sin(\xi_i) \end{bmatrix} 
 \xi_{iss}
+
\begin{bmatrix} -\cos(\xi_i) & \sin (\xi_i) \\ -\sin(\xi_i) & -\cos(\xi_i) \end{bmatrix} 
\xi_{i s}^2. 
\]
Then we compute 
\[ (\partial_s U_1) D (\partial_s U^t_2) = 
\begin{bmatrix}
\sin(\xi_1) \sin(\xi_2) + \cos(\xi_1)  \cos(\xi_2) \sigma&- \sin(\xi_1) \cos(\xi_2) +   \cos(\xi_1)\sin(\xi_2) \sigma\\
 - \cos(\xi_1)\sin(\xi_2)+ \sin(\xi_1) \cos(\xi_2)   \sigma  &  \cos(\xi_1) \cos(\xi_2)+\sin(\xi_1) \sin(\xi_2)\sigma
\end{bmatrix} \xi_{1s}\xi_{2s} ,
\]
\begin{align*}
(\partial_s^2 U_1) D U_2^t &=   
\begin{bmatrix}
-\sin(\xi_1) \cos(\xi_2) + \cos(\xi_1)  \sin(\xi_2) \sigma&- \sin(\xi_1) \sin(\xi_2) -   \cos(\xi_1)\cos(\xi_2) \sigma\\
 \cos(\xi_1)\cos(\xi_2)+ \sin(\xi_1) \sin(\xi_2)   \sigma  &  \cos(\xi_1) \sin(\xi_2)-\sin(\xi_1) \cos(\xi_2)\sigma
\end{bmatrix} 
\xi_{1ss}  \\
& \quad +   
\begin{bmatrix}
-\cos(\xi_1) \cos(\xi_2) - \sin(\xi_1)  \sin(\xi_2) \sigma&- \cos(\xi_1) \sin(\xi_2) + \sin(\xi_1)\cos(\xi_2) \sigma\\
 -\sin(\xi_1)\cos(\xi_2)+ \cos(\xi_1) \sin(\xi_2)   \sigma  &  -\sin(\xi_1) \sin(\xi_2)-\cos(\xi_1) \cos(\xi_2)\sigma
\end{bmatrix} 
\xi_{1s}^2,
\end{align*}
\begin{align*}
U_1 D  (\partial^2_s U_2^t )
& =  \begin{bmatrix}
-\sin(\xi_2) \cos(\xi_1) + \cos(\xi_2)  \sin(\xi_1) \sigma&     \cos(\xi_1)\cos(\xi_2) +\sin(\xi_1) \sin(\xi_2) \sigma\\
  -\sin(\xi_1) \sin(\xi_2) - \cos(\xi_1)\cos(\xi_2)  \sigma  &  \cos(\xi_2) \sin(\xi_1)-\sin(\xi_2) \cos(\xi_1)\sigma
\end{bmatrix} 
\xi_{2ss}  \\
& \quad +   
\begin{bmatrix}
-\cos(\xi_1) \cos(\xi_2) - \sin(\xi_1)  \sin(\xi_2) \sigma&
- \cos(\xi_1) \sin(\xi_2) + \sin(\xi_1)\cos(\xi_2) \sigma \\
 -\sin(\xi_1)\cos(\xi_2)+ \cos(\xi_1) \sin(\xi_2)   \sigma &  -\sin(\xi_1) \sin(\xi_2)-\cos(\xi_1) \cos(\xi_2)\sigma
\end{bmatrix}
\xi_{2s}^2,
\end{align*}
and
\[ 
\partial_z \tilde A_0 = 
\begin{bmatrix} \sin (\xi_1) \sin (\xi_2)   & -\sin (\xi_1) \cos (\xi_2)  \\ 
-\cos (\xi_1) \sin (\xi_2)  &  \cos (\xi_1) \cos (\xi_2) \end{bmatrix} \sigma_z .
\]
Using 
\[
\int_{-\infty}^{\infty} \sigma \sigma_z \ dz 
= \frac{1}{2}\int_{-\infty}^{\infty} (\sigma^2)_z \ dz 
= 0 
\quad \textrm{and} \quad 
\int_{-\infty}^{\infty} \sigma_z \ dz = 2 
\]
along with the above identities, we have
\begin{align*}
\int_{-\infty}^{\infty}\langle  \partial_{ss} \tilde A_0, \partial_z \tilde A_0 \rangle_F \ dz 
&= 4 \xi_{1s} \xi_{2s}  \\
&= 4 \  \partial_s \left(  \frac{\eta_- (s,t) +\eta_+(s,t)}{2}  \right) \partial_s \left( \frac{\eta_- (s,t) - \eta_+(s,t)}{2} \right)  \\
&= (\partial_s \eta_-)^2- (\partial_s \eta_+)^2  \\
&=  [\eta_s^2]_\Gamma 
\end{align*}
and 
\begin{align*}
\gamma(s) = \int_{-\infty}^{\infty}\langle  \partial_z \tilde A_0, \partial_z \tilde A_0 \rangle_F \ dz  = \int_{-\infty}^{\infty} \sigma_z^2 \ dz = \int_{-\infty}^{\infty} \left(1-\tanh^2(z/\sqrt{2})\right)^2 \ dz
\end{align*}
which is independent of $s$ and denoted by $\bar \gamma$.
\end{proof}

\subsubsection{Summary of the behavior at the $O(\varepsilon t)$ time scale:} \label{sec:sumpart2slow}
\begin{enumerate}
\item Away from the interface, $\bar A_0$, $\bar A_1$, and $\bar A_2$ satisfy \eqref{epstouter}. In particular, when $n =2$, $\bar A_0$ takes the form in \eqref{barA0epst} where  the phase $\eta$ satisfies $\Delta \eta = 0$.

\item The interface moves in the normal direction according to the motion law given in \eqref{e:MotionLaw}. In the case where $n=2$, the second term of the motion law is the jump in the squared tangental derivative of the phase across the interface $\Gamma$ as in \eqref{eq:motionlawslow}. 
\end{enumerate}

\section{Numerical experiments}\label{sec:num}
In this section, we perform a variety of numerical experiments to support, verify, and illustrate our analytical results in Section~\ref{sec:part1} and Section~\ref{sec:part2}. The algorithm we use is summarized in Section~\ref{sec:alg} and the numerical examples are described in Section~\ref{sec:numex}.

\subsection{Algorithm to solve \eqref{e:orig} and implementation details.}\label{sec:alg}
To numerically solve  \eqref{e:orig}, we use an efficient diffusion generated method recently developed in \cite{osting2017}. This method generalizes the Merriman-Bence-Osher method for mean curvature flow  \cite{merriman1994motion}  and methods for the Ginzburg-Landau energy \cite{Ruuth_2001,Viertel_2019}. 
The algorithm alternates a diffusion and a projection step as summarized in Algorithm~\ref{a:MBO}. 
In \cite{osting2017}, the Lyapunov function of Esedoglu and Otto \cite{esedoglu2015threshold} was extended to show that the method is non-increasing on iterates and hence, unconditionally stable. 
It was also proven that the spatially discretized iterates converge to a stationary solution in a finite number of iterations.
We refer to \cite{osting2017} for more details and properties of the algorithm.

\begin{algorithm}[t!]
\DontPrintSemicolon
 \KwIn{a time step $\tau > 0$ and initial condition $A_0 \in H^1(\Omega; O_n)$.}
 \KwOut{a sequence of matrix-valued functions $A_s \in H^1(\Omega; O_n)$, $s=1, 2, \ldots$ that approximately solve \eqref{e:orig} at times, $s\tau$.}
 Set $s=1$\;
 \While{not converged}{
{\bf 1.  Diffusion Step.} Solve the initial value problem for the diffusion equation  until time $\tau$ with initial value given by $A_{s-1}(x)$:
\begin{align*}
&\partial_t A(t,x) = \Delta A (t,x) \\
&A(0,x) = A_{s-1}(x).
\end{align*}
Let $\tilde A(x) = A(\tau,x)$\;
{\bf 2. Projection Step.} Set $A_s(x) = \Pi_{O_n} \tilde A(x) $\;
Set $s = s+1$\;
 }
\caption{A diffusion generated method for solving  \eqref{e:orig} \cite{osting2017}. } 
\label{a:MBO}
\end{algorithm}

We implemented the algorithm in MATLAB. 
In all experiments, we consider the case when $n = 2$ on a flat torus $\Omega = [-1/2, 1/2]^2$ discretized using $1024 \times 1024$ uniform grid points and set $\tau = 0.015625$. 
The heat diffusion equation in Algorithm~\ref{a:MBO} is efficiently solved using the fast Fourier transform (FFT).
The convergence criteria of the algorithm is taken to be 
$$
\int_\Omega \|A_s(x)-A_{s-1}(x)\|_F \ dx \leq tol,
$$ 
for  $tol = 10^{-6}$. 
All reported results were obtained on a laptop with a 2.7GHz Intel Core i5 processor and 8GB of RAM.  

Here we visualize an $O_2$ valued field by plotting the vector field generated by the first column vector. The second column vector is orthogonal to the first and the direction is indicated by color, when necessary. 

\newpage
\subsection{Numerical examples}\label{sec:numex}
\subsubsection{Evolution of $SO_n$-valued fields}\label{sec:single}
We first perform a numerical experiment to verify the results in Section~\ref{sec:part1} for the time evolution of a single-signed determinant initial matrix-valued field. Without loss of generality, we consider the case where the initial matrix-valued field takes values  in $SO_n$.

Figures~\ref{fig:1} and \ref{fig:2} display the evolution of an $SO_2$ matrix-valued field with the initial condition given by 
\begin{equation}\label{eq:initialfield}
A_0(x) = \begin{bmatrix}
\cos \eta(x) &  -\sin \eta(x) \\
\sin \eta(x) & \cos \eta(x) \\
\end{bmatrix},
\end{equation}
for different choices of $\eta \colon \Omega \to \mathbb R$.

In Figure~\ref{fig:1}, we take 
$$
\eta(x) = \frac{\pi}{2} \sin(2\pi (3x_1+2x_2)) \qquad \textrm{for} \ x = (x_1,x_2) \in \Omega. 
$$
From Figure~\ref{fig:1}, we see that the matrix-valued field evolves toward a uniform matrix-valued field, which, as discussed in Section~\ref{s:Harmonic}, is a stationary state of the $O_n$ diffusion equation.

 \begin{figure}[t]
\includegraphics[width = 0.24 \textwidth,clip,trim= 5cm 1cm 5cm 0cm]{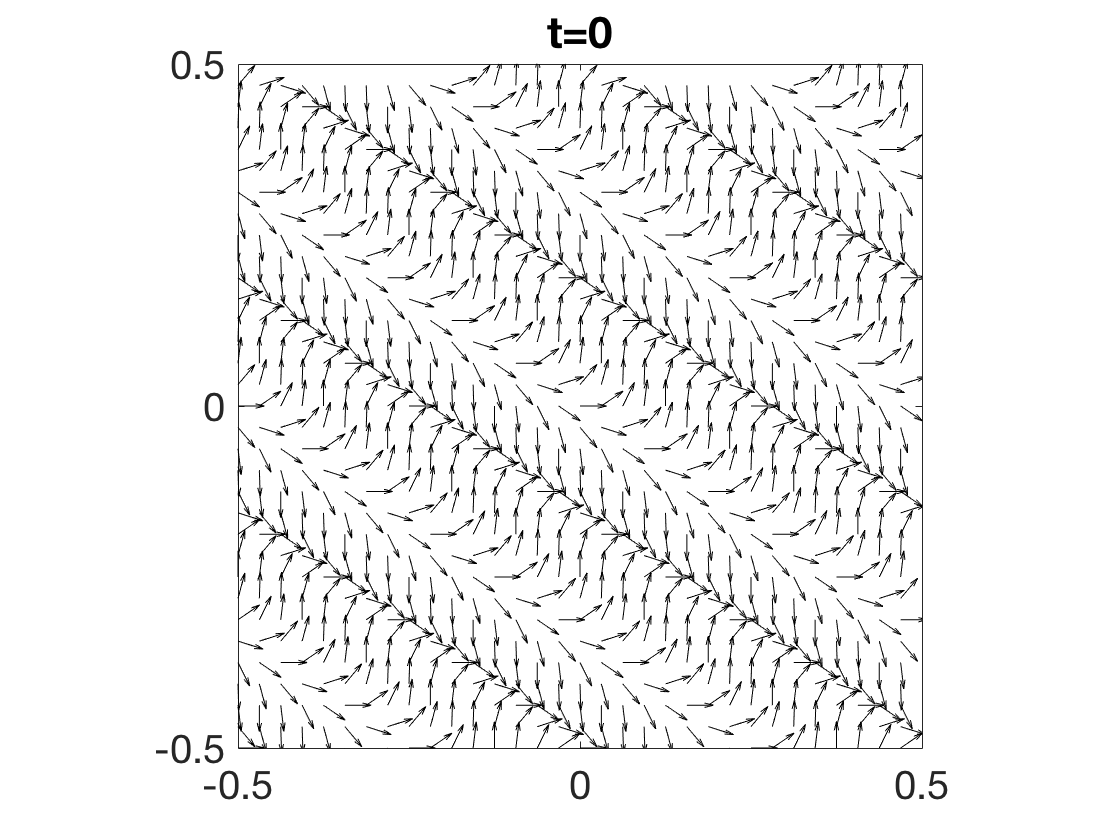}
\includegraphics[width = 0.24 \textwidth,clip,trim= 5cm 1cm 5cm 0cm]{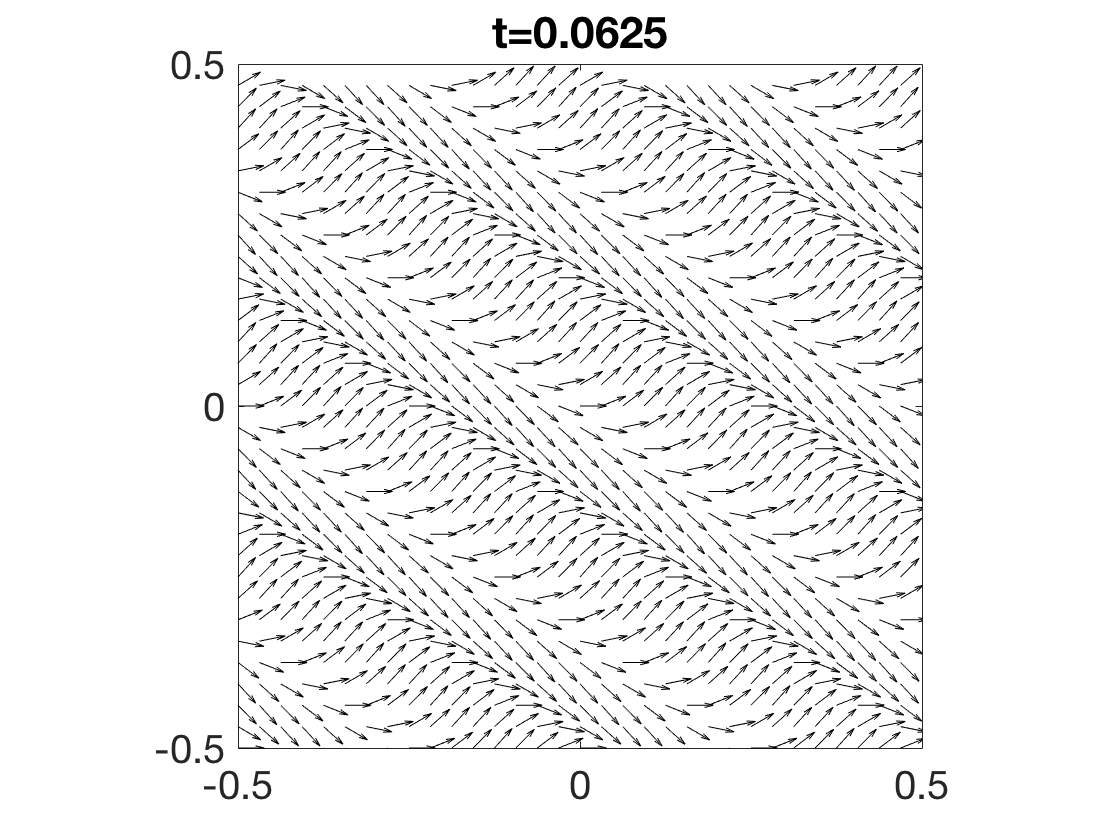}
\includegraphics[width = 0.24 \textwidth,clip,trim= 5cm 1cm 5cm 0cm]{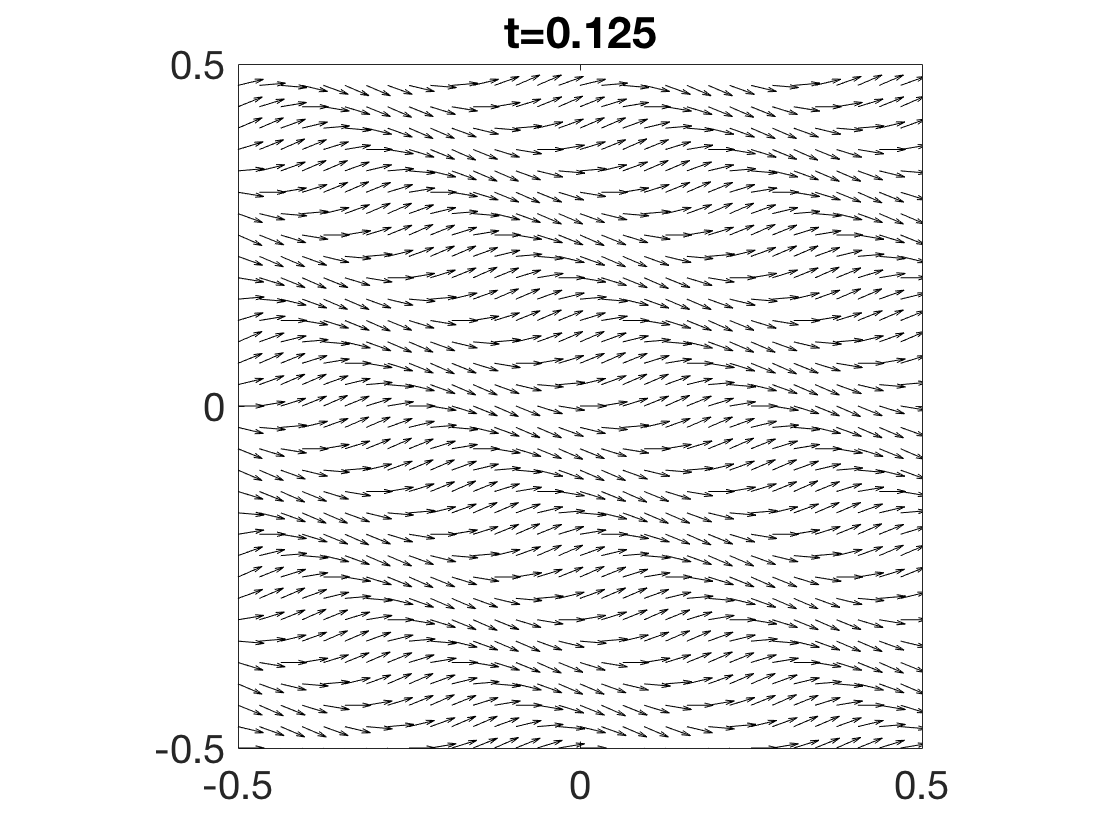} 
\includegraphics[width = 0.24 \textwidth,clip,trim= 5cm 1cm 5cm 0cm]{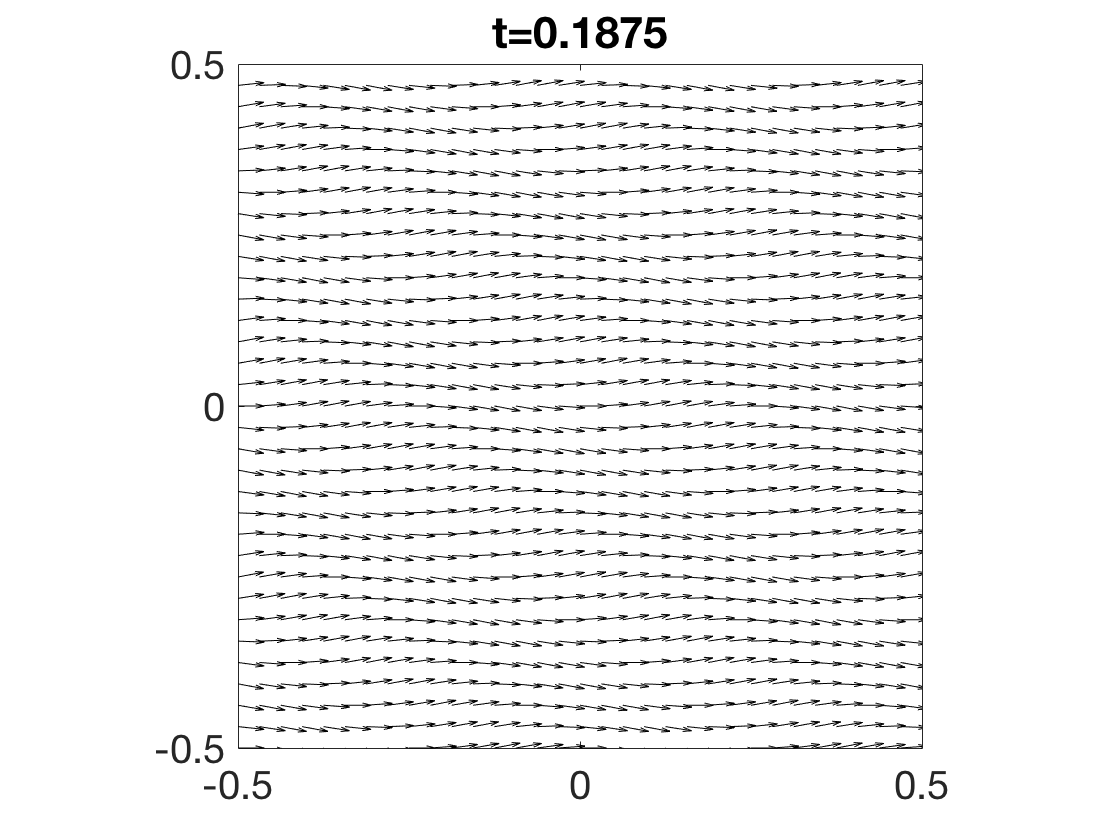}
\includegraphics[width = 0.24 \textwidth,clip,trim= 5cm 1cm 5cm 0cm]{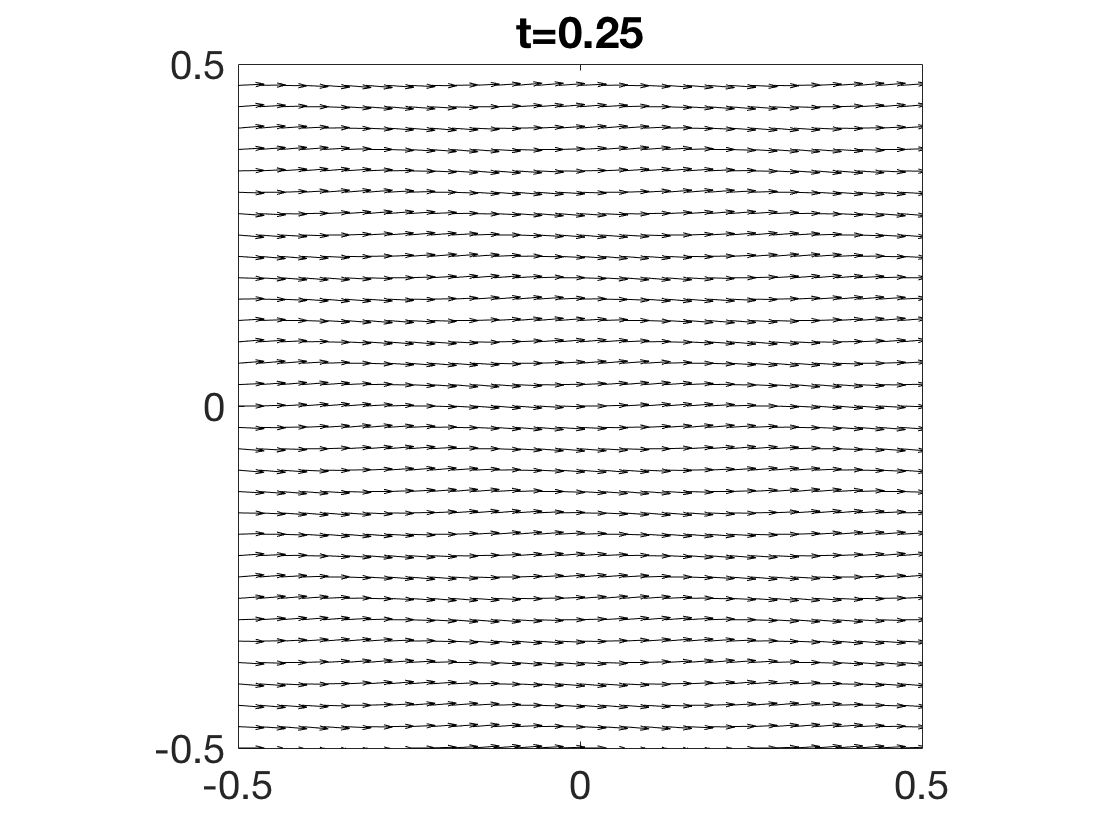}
\includegraphics[width = 0.24 \textwidth,clip,trim= 5cm 1cm 5cm 0cm]{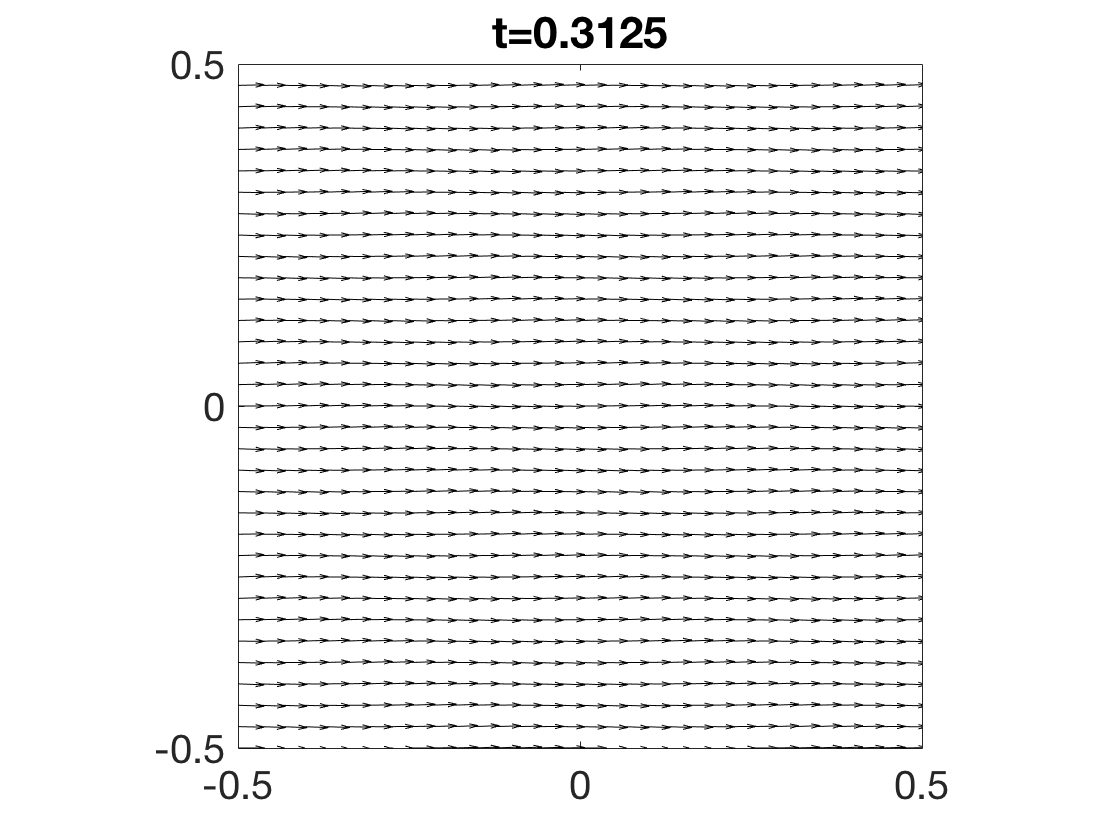}
\includegraphics[width = 0.24 \textwidth,clip,trim= 5cm 1cm 5cm 0cm]{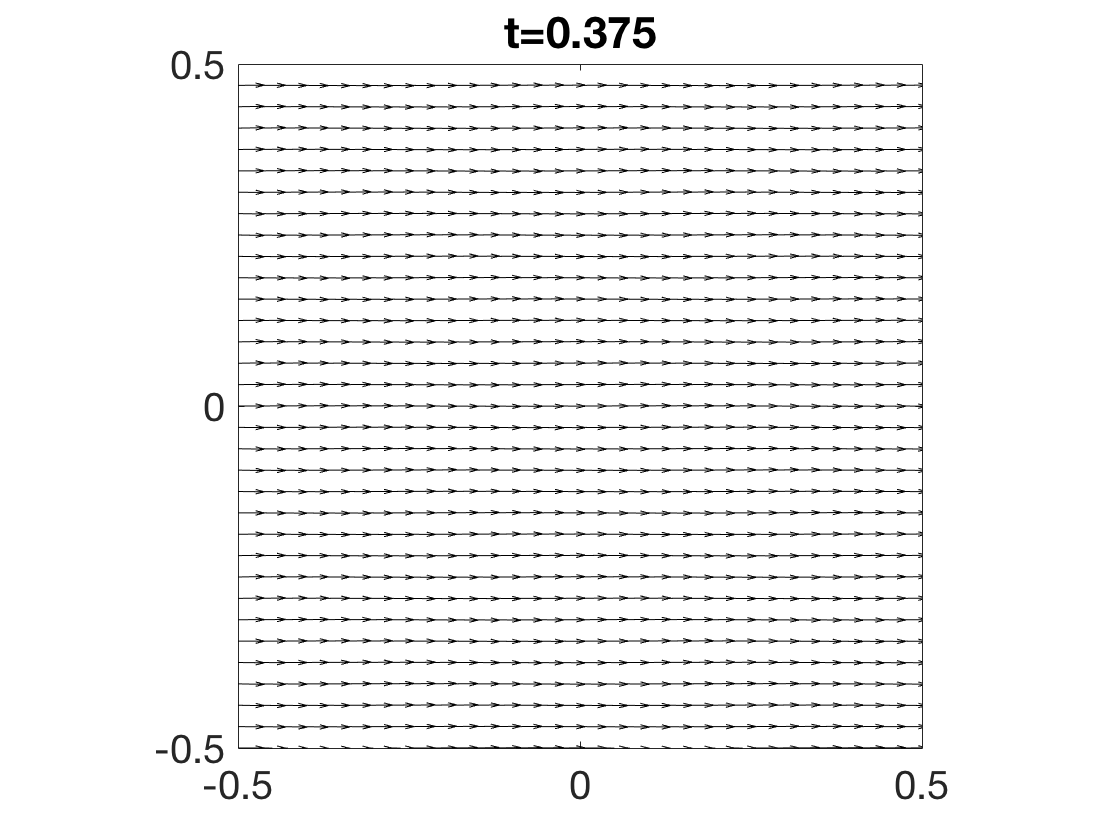}
\includegraphics[width = 0.24 \textwidth,clip,trim= 5cm 1cm 5cm 0cm]{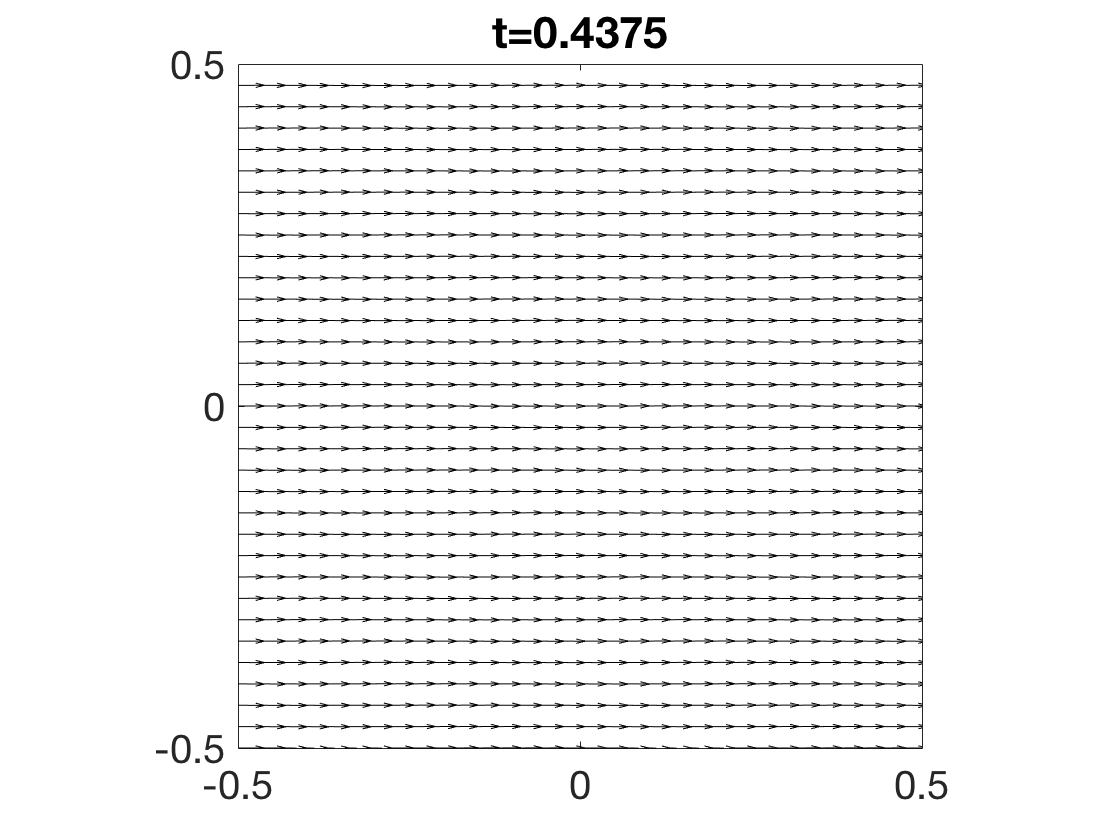}
\caption{Snapshots of the time evolution of an initial $SO_2$ matrix-valued field. The initial field is given in \eqref{eq:initialfield} with $\eta(x) = \frac{\pi}{2} \sin(2\pi (3x_1+2x_2))$. See Section~\ref{sec:single}.} \label{fig:1}
\end{figure}

In Figure~\ref{fig:2}, we set 
$$
\eta(x) = 2 \pi x_1+\frac{\pi}{2} \sin(2\pi x_1). 
$$
We observe that the field evolves toward a field with $\eta(x) = 2 \pi x_1$. Again, since $\Delta \eta = 0$, this is a harmonic orthogonal matrix-valued field; see Section~\ref{s:Harmonic}.

 \begin{figure}[ht]
\includegraphics[width = 0.24 \textwidth,clip,trim= 2cm 0cm 3cm 0cm]{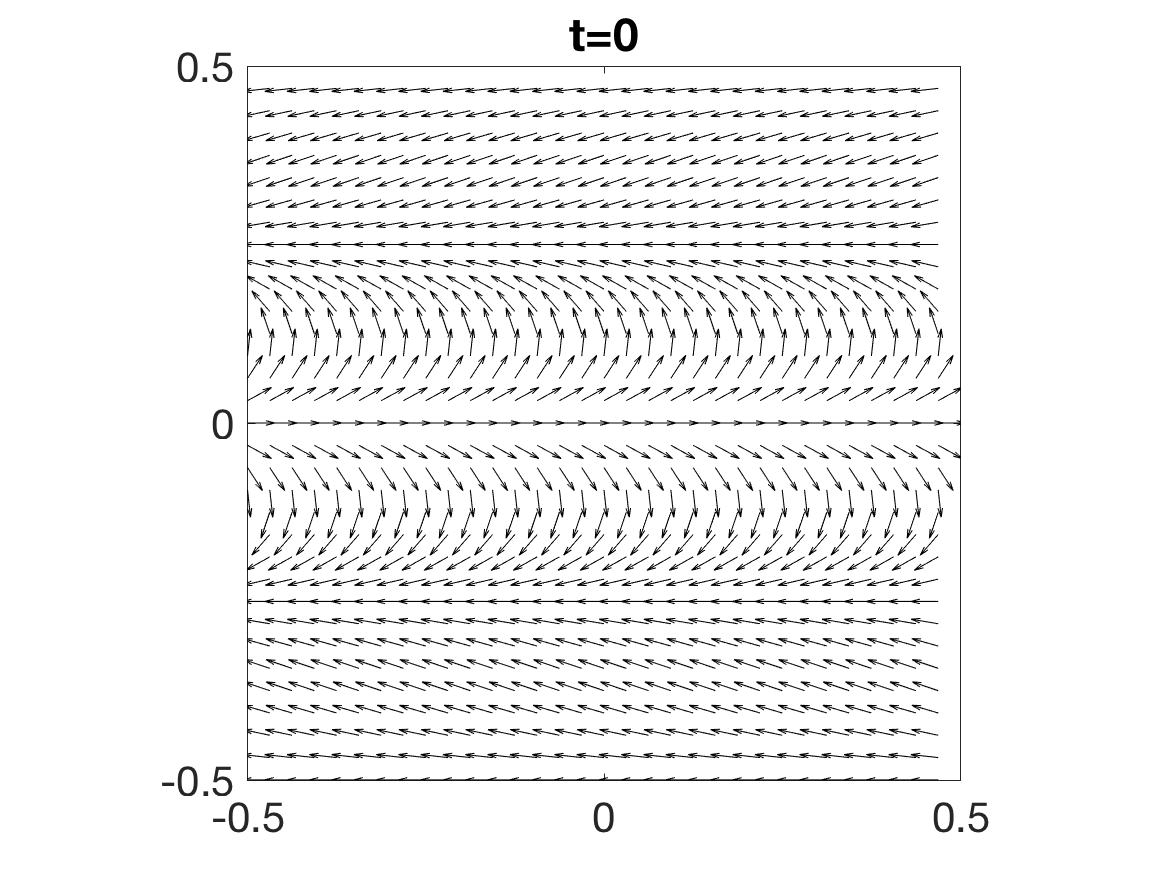}
\includegraphics[width = 0.24 \textwidth,clip,trim= 2cm 0cm 3cm 0cm]{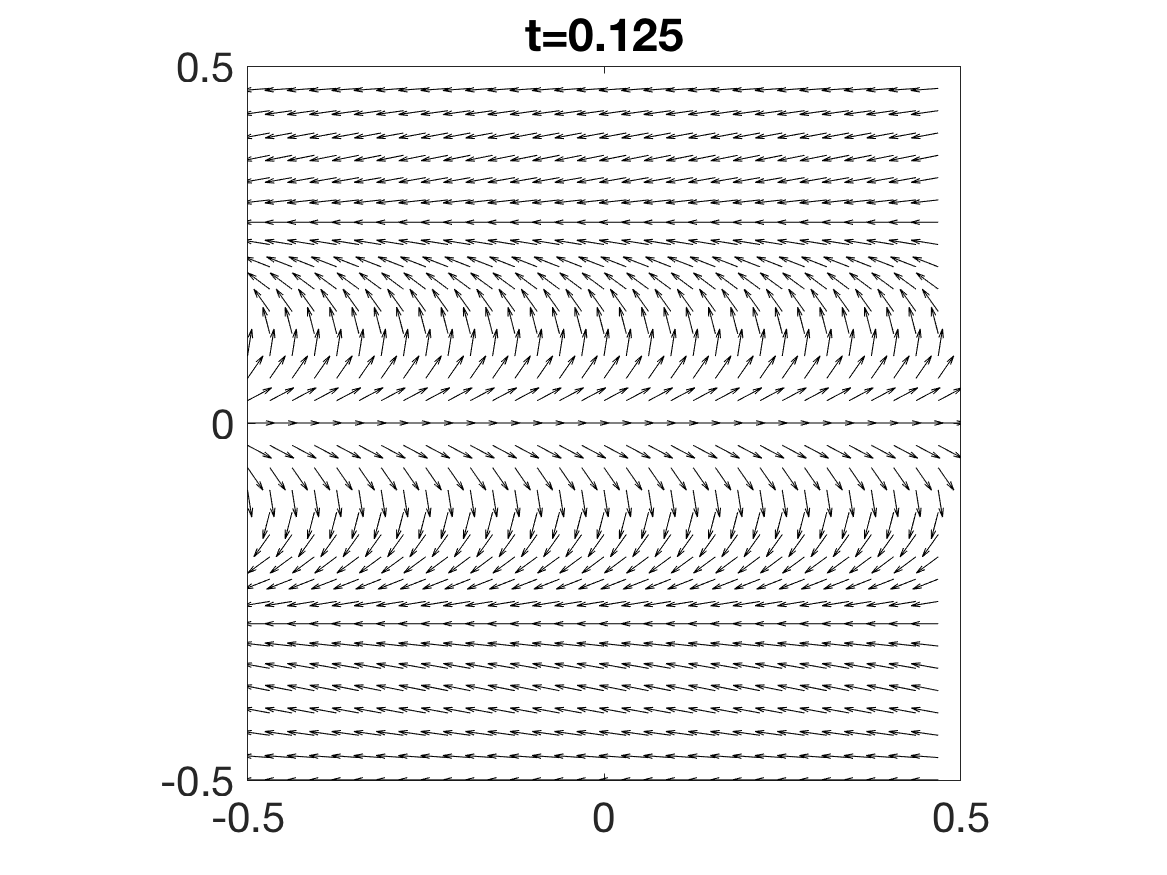}
\includegraphics[width = 0.24 \textwidth,clip,trim= 2cm 0cm 3cm 0cm]{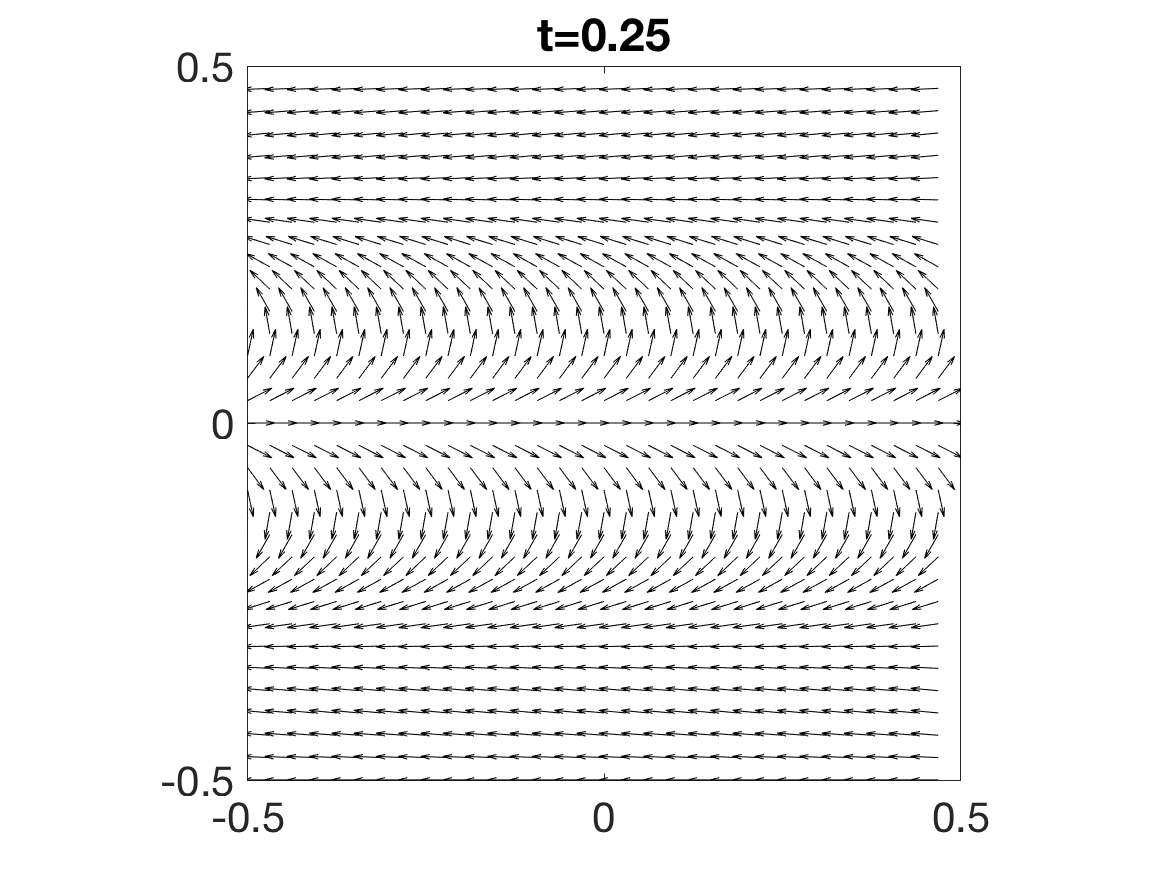} 
\includegraphics[width = 0.24 \textwidth,clip,trim= 2cm 0cm 3cm 0cm]{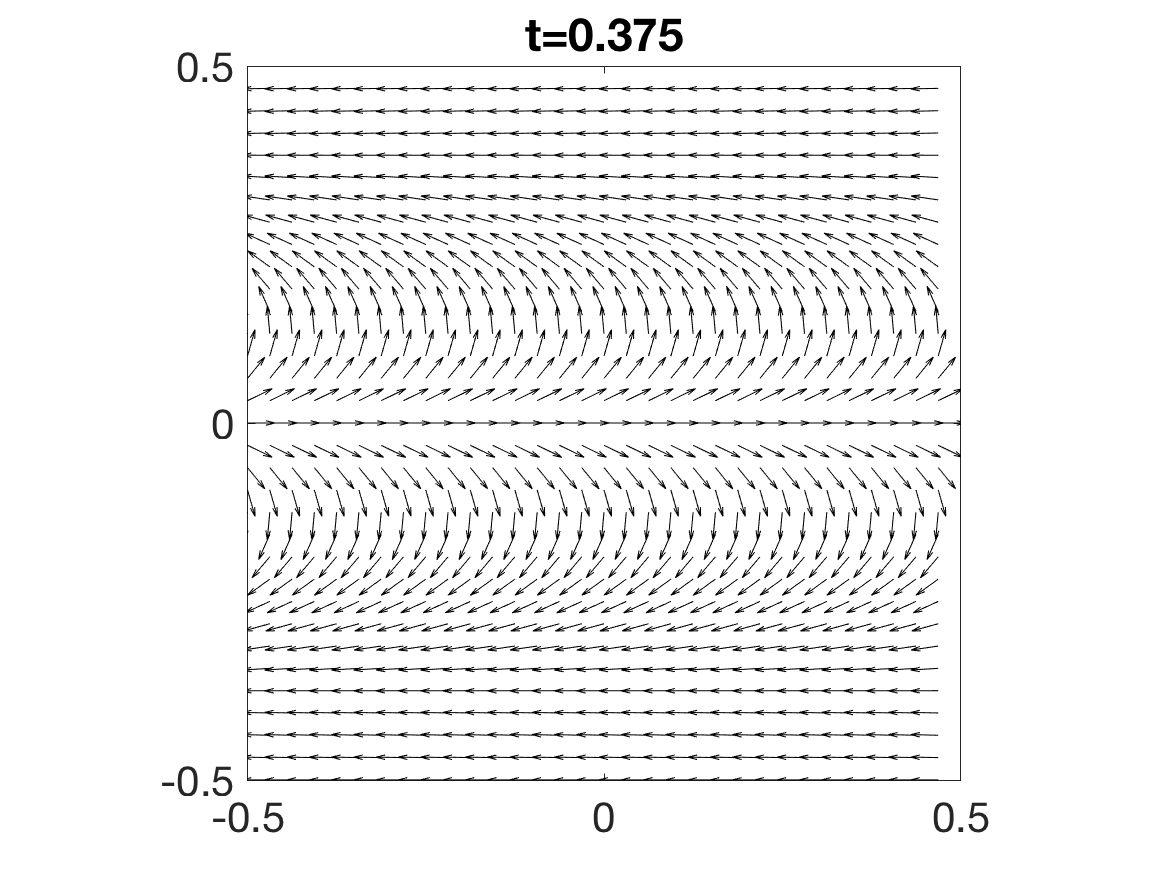}
\includegraphics[width = 0.24 \textwidth,clip,trim= 2cm 0cm 3cm 0cm]{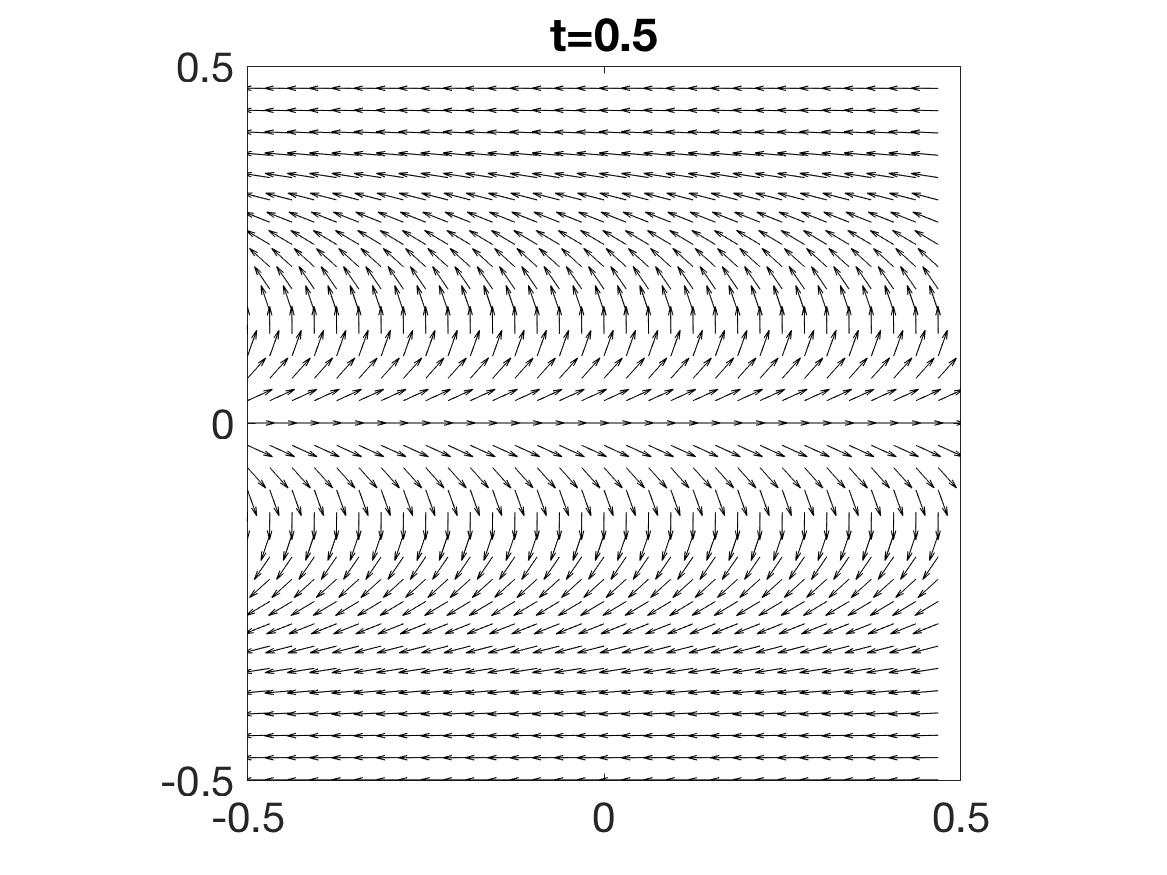}
\includegraphics[width = 0.24 \textwidth,clip,trim= 2cm 0cm 3cm 0cm]{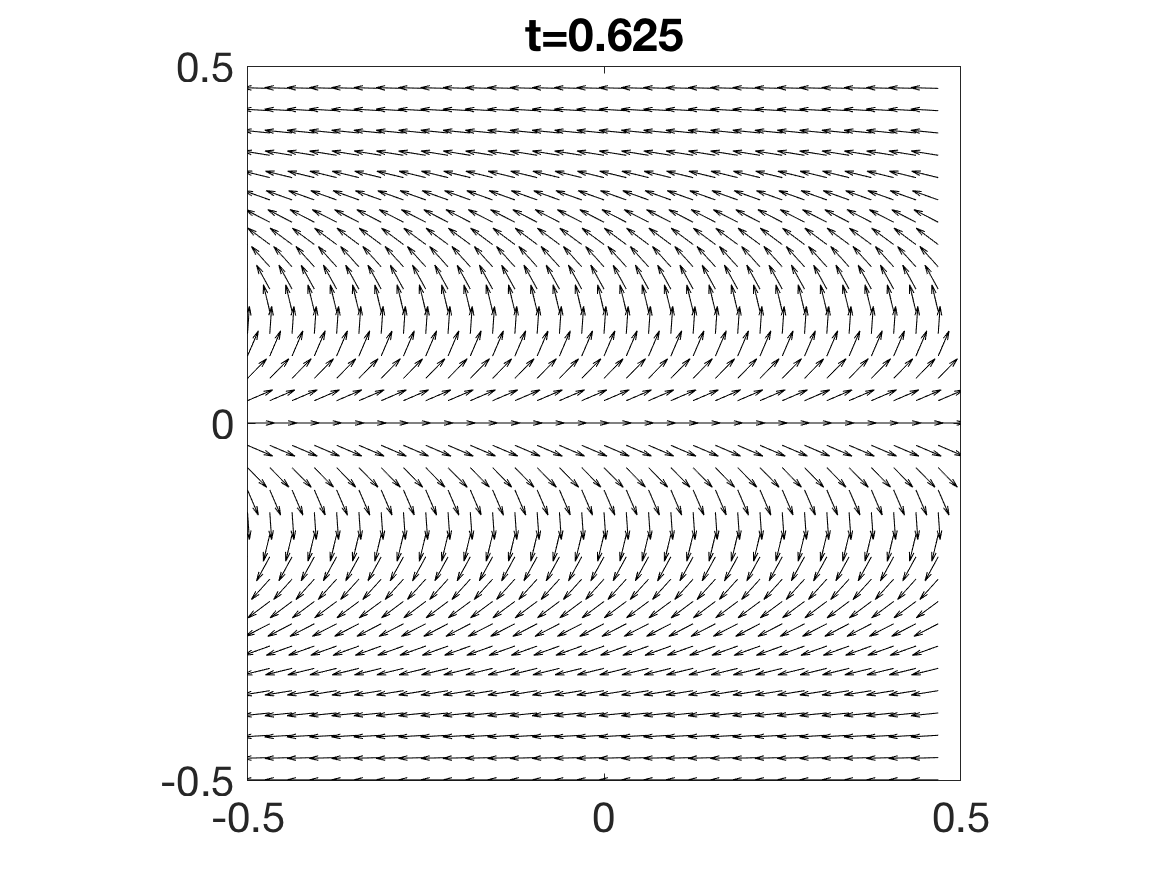}
\includegraphics[width = 0.24 \textwidth,clip,trim= 2cm 0cm 3cm 0cm]{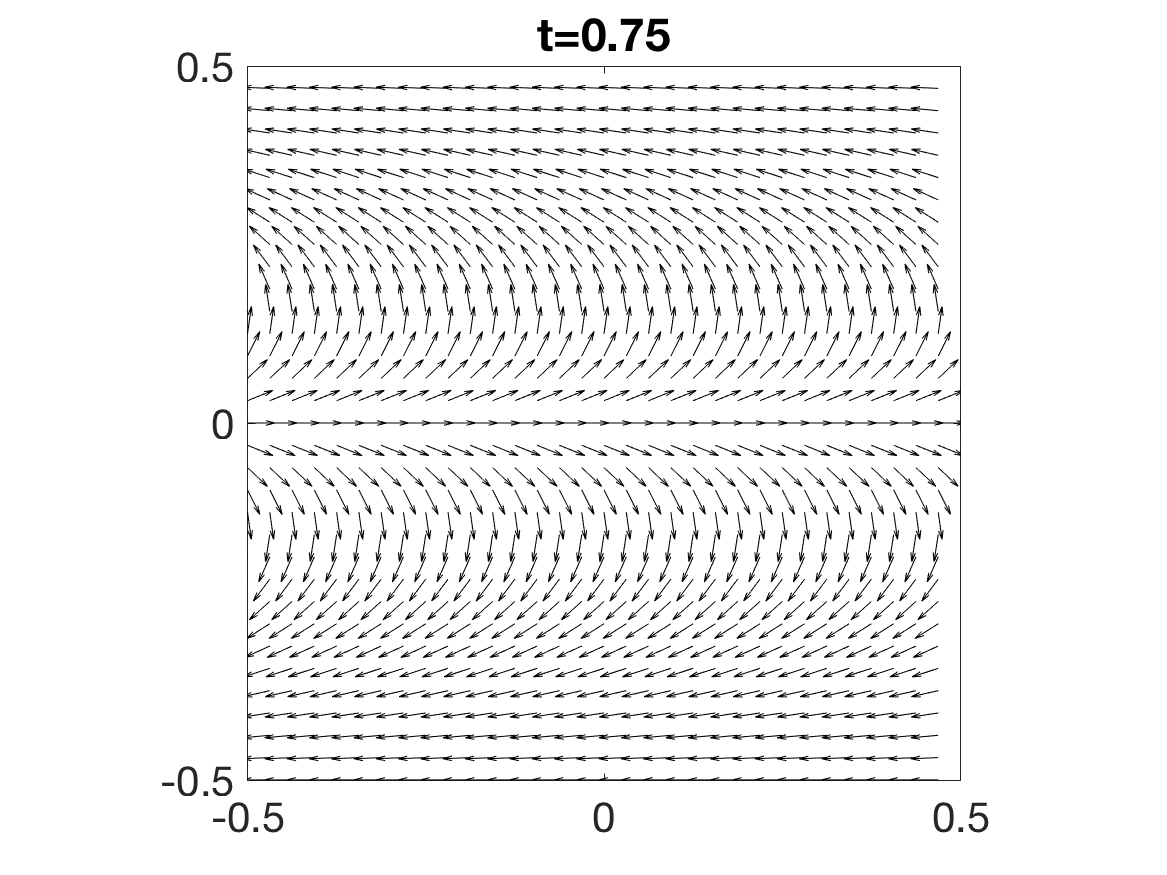}
\includegraphics[width = 0.24 \textwidth,clip,trim= 2cm 0cm 3cm 0cm]{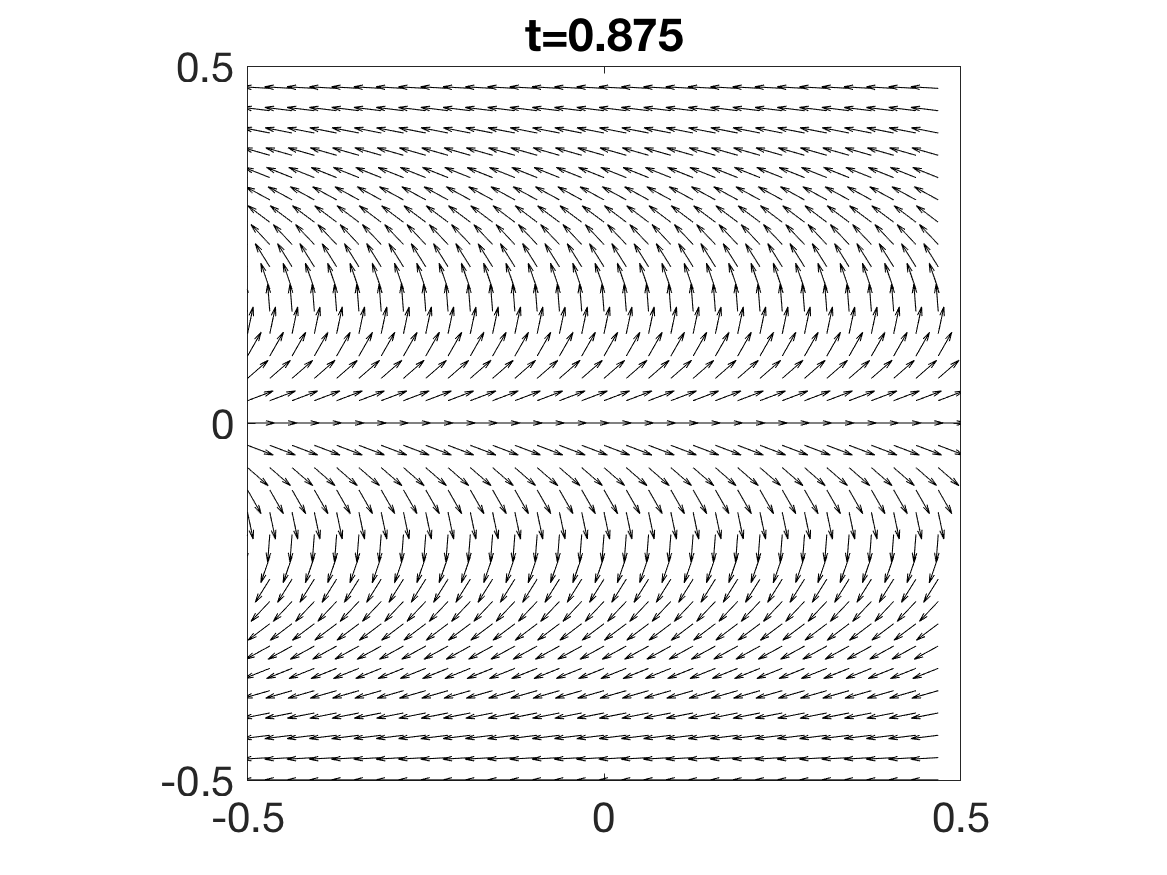}
\includegraphics[width = 0.24 \textwidth,clip,trim= 2cm 0cm 3cm 0cm]{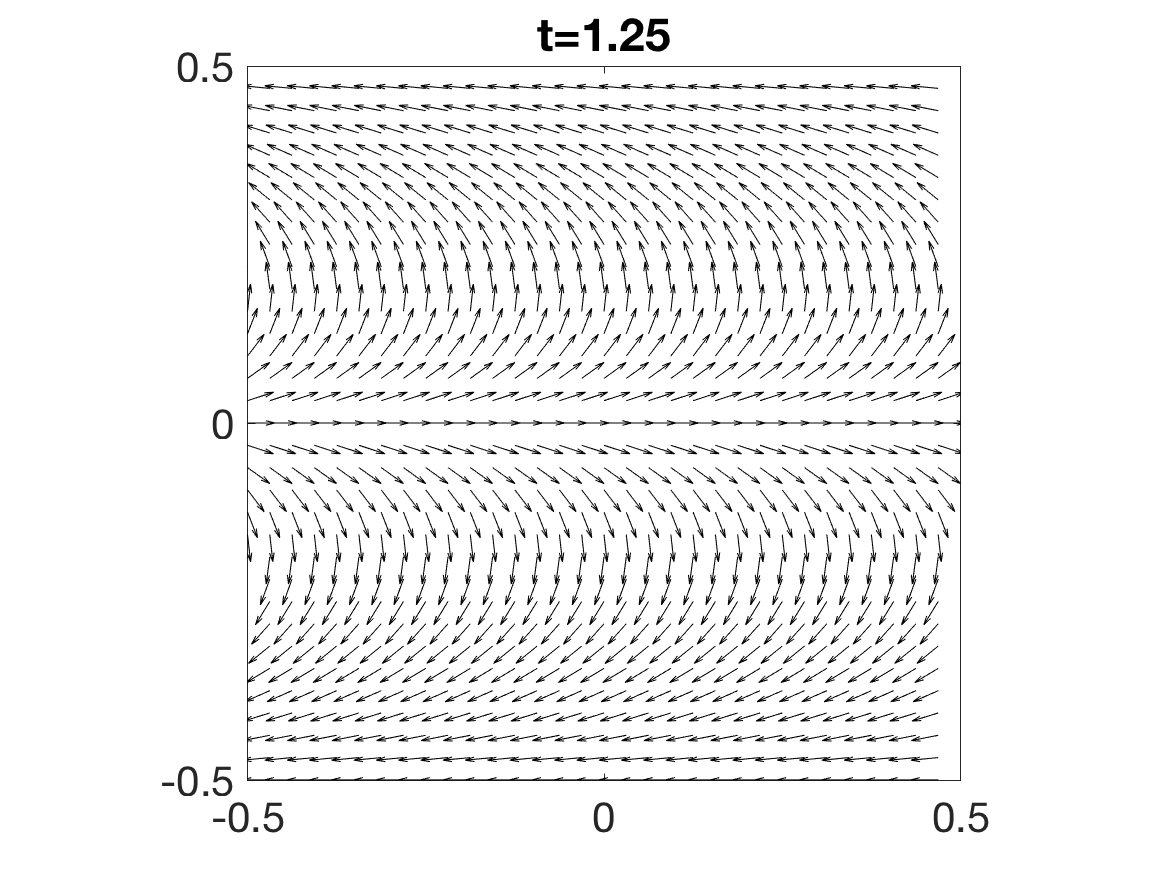}
\includegraphics[width = 0.24 \textwidth,clip,trim= 2cm 0cm 3cm 0cm]{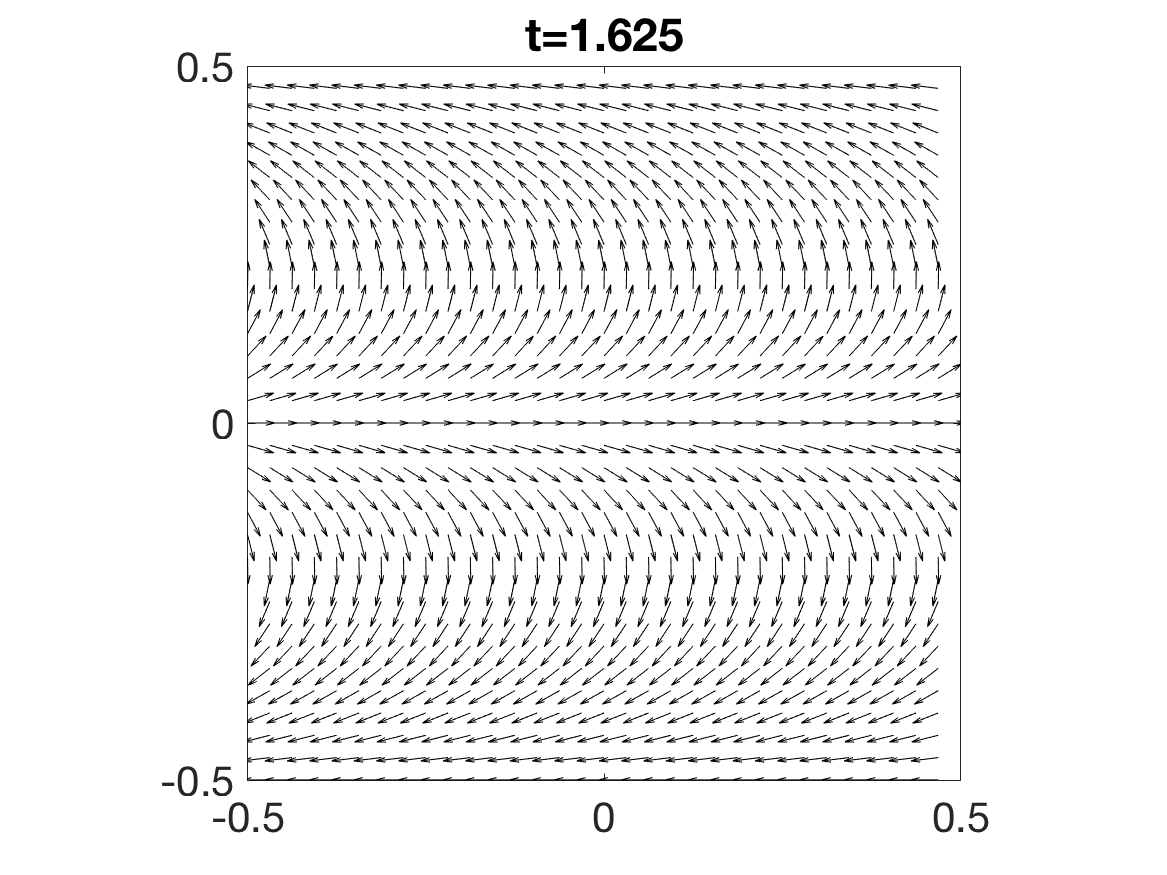}
\includegraphics[width = 0.24 \textwidth,clip,trim= 2cm 0cm 3cm 0cm]{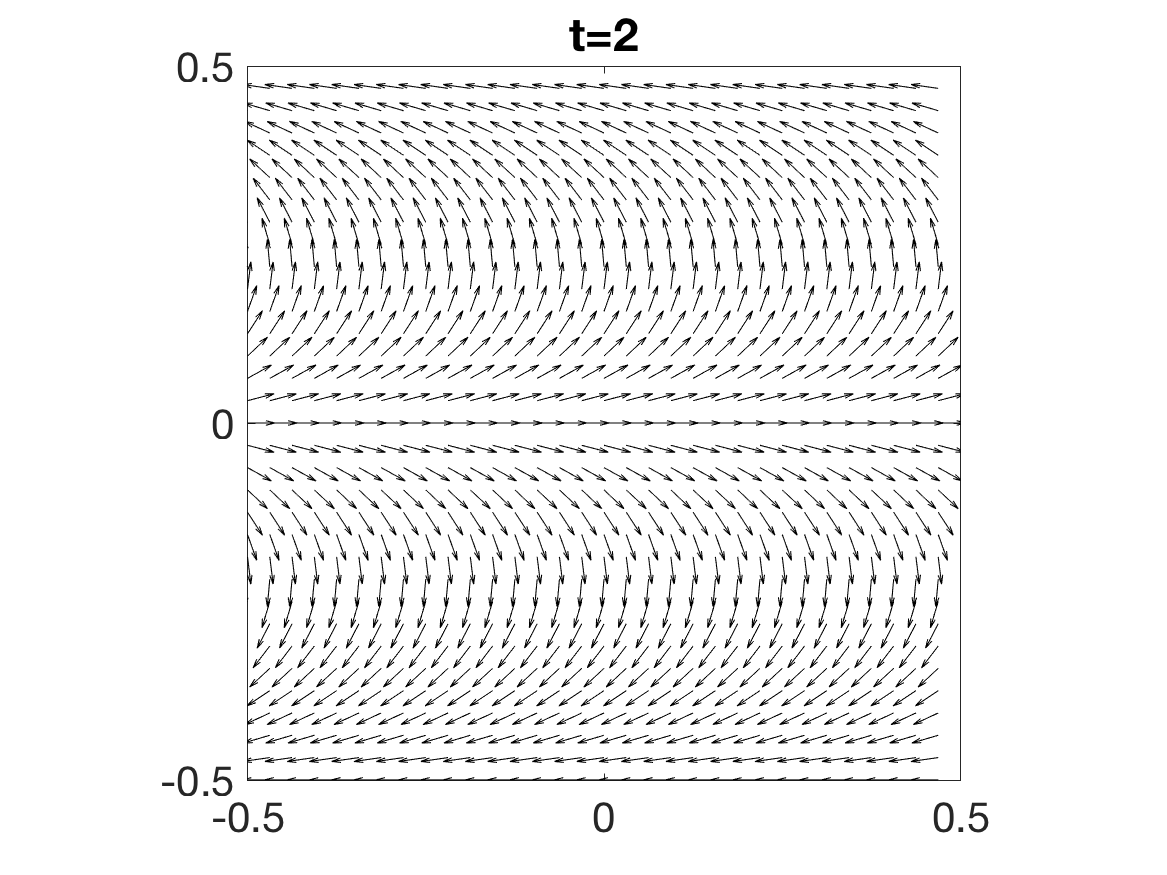}
\includegraphics[width = 0.24 \textwidth,clip,trim= 2cm 0cm 3cm 0cm]{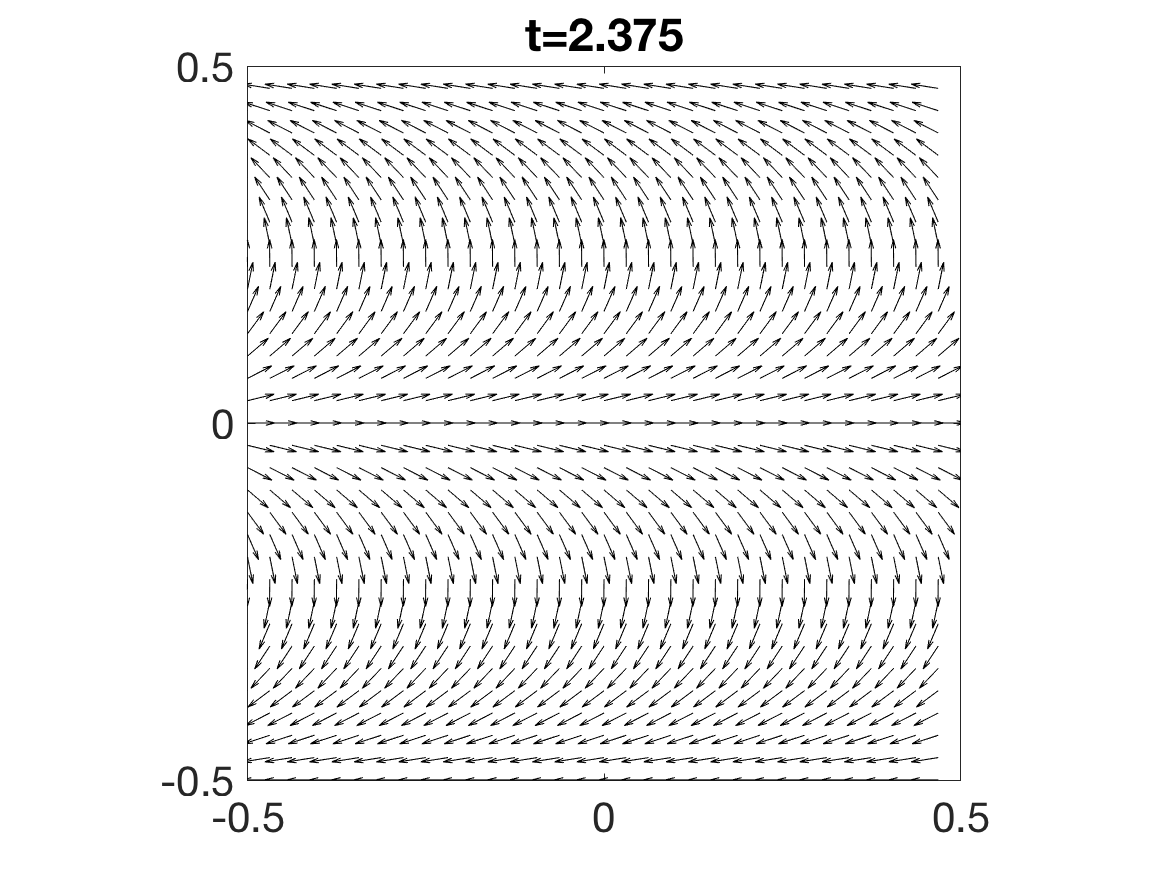}
\caption{Snapshots of the time evolution of an initial $SO_2$ matrix-valued field. The initial field is given in \eqref{eq:initialfield} with $\eta(x) = 2 \pi x_1+\frac{\pi}{2} \sin(2\pi x_1)$. See Section~\ref{sec:single}.} \label{fig:2}
\end{figure}

\medskip

To better understand the behavior in Figures~\ref{fig:1} and \ref{fig:2}, we recall the definition of the pair of indices of a matrix-valued field discussed in \cite{osting2017}. 
Let $v\colon \Omega \to \mathbb C$ be a complex-valued field with no zeros. 
Let $\gamma \colon [0,1] \rightarrow \Omega$ be a closed curve. 
We define the \emph{index of $\gamma$ with respect to $v$} to  be
\[\mathrm{ind}_v(\gamma) := \frac{1}{2\pi} \left[ \arg{v(\gamma(1))} - \arg{ v(\gamma(0))} \right].  \]
Clearly the index of $\gamma$ is an integer and varies continuously with deformations to $\gamma$, so it depends only on the homotopy class of $\gamma$. For a torus, we can parameterize the homotopy classes by the number of times the curve wraps around $\Omega$ in the $x_1$- and $x_2$-directions. Furthermore, if we let $[\gamma]_{m,n}$ denote the equivalence class of curves that wraps around $\Omega$ $m$ times in the $x_1$-direction and $n$ times in the $x_2$-direction, then it is not difficult to see that 
\[\mathrm{ind}_v([\gamma]_{m,n}) = \mathrm{ind}_v([\gamma]_{1,0})^m + \mathrm{ind}_v([\gamma]_{0,1})^n. \]
So we can characterize the index of any curve in terms of the indices of $[\gamma]_{1,0}$ and $[\gamma]_{0,1}$. For a given field $v$, we let 
$$
I = \left( \mathrm{ind}_v([\gamma]_{1,0}), \mathrm{ind}_v([\gamma]_{0,1}) \right) 
$$ 
be the \emph{index pair} corresponding to curves that wrap around $\Omega$ once in the $x_1$- and $x_2$-directions. For a matrix-valued field $A \colon \Omega \to SO_2$ or $A \colon \Omega \to SO_2^-$, we define the \emph{index pair}, $I$, to be the index pair for the first column of $A$. For example, for the harmonic orthogonal matrix fields in \eqref{e:HarmSol}, the index pair is $I = (n_1,n_2)$.

In Figure~\ref{fig:1}, the index pair for the initial condition is $(0,0)$ and the field evolves toward the harmonic orthogonal matrix field with index pair $(0,0)$, the uniform matrix field. 
In Figure~\ref{fig:2}, the index pair for the initial condition is $(1,0)$ and the field evolves toward the harmonic orthogonal matrix field with index pair $(1,0)$.  We observe and generally expect that the index pair is invariant under flow by the $O_n$ diffusion equation.

\subsubsection{Evolution of $O_n$-valued fields at the $O(t)$ time scale.}\label{sec:doublefast}
In this section, we check the motion law we derived in Section~\ref{sec:part2fast}. That is, at the $O(t)$ time scale, if there is a line defect initially, the motion of the interface is driven by the curvature at each point. Note that at this time scale, we don't see the effect from the matrix-valued field on the motion law of the interface. So we perform two experiments where the initial condition has the same line defect, but different initial matrix-valued fields. 
Specifically,  we choose the following initial condition for different choices of $\eta \colon \Omega \to \mathbb R$, 
\begin{align}\label{eq:initialslow1}
A(r,\theta)  = 
\begin{cases}
\begin{bmatrix}
\cos \eta & -\sin \eta \\
\sin \eta & \cos \eta
\end{bmatrix},   & \text{if} \  r < 0.15+0.03 \sin(12 \theta), \\
\begin{bmatrix}
\cos \eta & \sin \eta \\
\sin \eta & -\cos \eta
\end{bmatrix}, & \text{otherwise}
\end{cases}
\end{align}
where $(r, \theta)$ is the corresponding polar coordinate of $x = (x_1,x_2) \in \Omega$. 

In all subsequent figures, the domain is colored by the sign of the determinant of the matrix. For a matrix field $A \in H^1\left(\Omega; O_n \right)$, we use the convention 
\begin{center}
\begin{tabular}{l c l c l }
$x$ is yellow &  $\iff$  &  $\mathrm{det}(A(x)) =1$ & $\iff$ & $A(x) \in SO_n$ \\
$x$ is green  & $\iff$  & $\mathrm{det}(A(x)) =-1$ & $\iff$ & $A(x) \in SO_n^-$.
\end{tabular}
\end{center}

In Figure~\ref{fig:3},  we display several snapshots of the time evolution for two different initial conditions. 
In the first column of Figure~\ref{fig:3},  the initial field is chosen as in \eqref{eq:initialslow1} with $\eta(x) = \frac{\pi}{2} \sin(2\pi x_1)$ and, in the second column, the initial field has $\eta(x)  = 2\pi x_1$. 
Hence the pair of indices of the initial field in the first column is $(0,0)$ and 
the pair of indices of the initial field in the second column is $(1,0)$. 
In both columns of Figure~\ref{fig:3}, we observe that the region where $A \in SO_n$ shrinks with the interface becoming a circle before vanishing. 
We observe that the time dynamics of the line defect for the two different initial conditions are very close. 
This is consistent with our analytical results in Section~\ref{sec:part2fast}, that is, at the $O(t)$ time scale, the motion law is the leading order of the curvature of the line defect, which is independent of the matrix-valued field.  
For the evolution in the left column of  Figure~\ref{fig:3}, the field continues to evolve toward the uniform solution for longer times than shown in the Figure. 

\begin{figure}[ht]
\begin{tabular}{c|c}
\includegraphics[width = 0.24 \textwidth,clip,trim= 5cm 1cm 5cm 0cm]{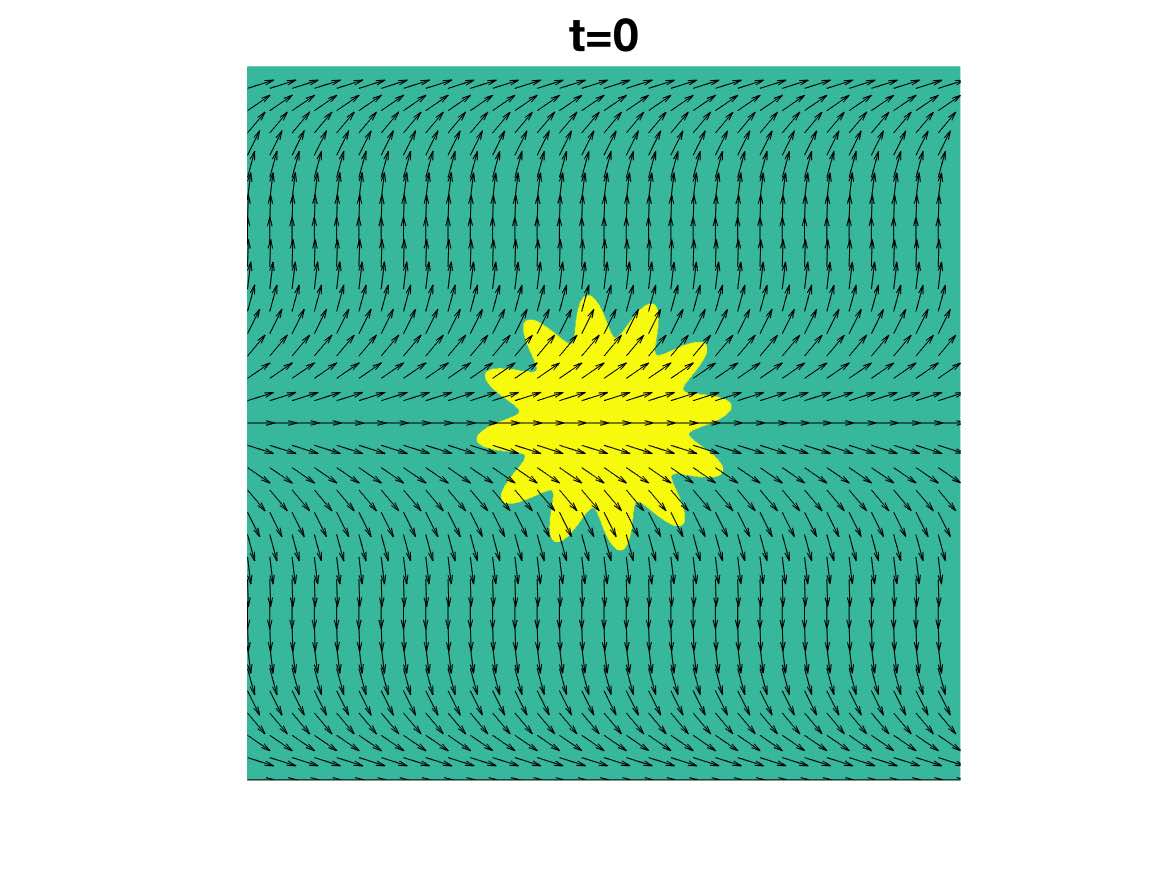} 
\includegraphics[width = 0.24 \textwidth,clip,trim= 5cm 1cm 5cm 0cm]{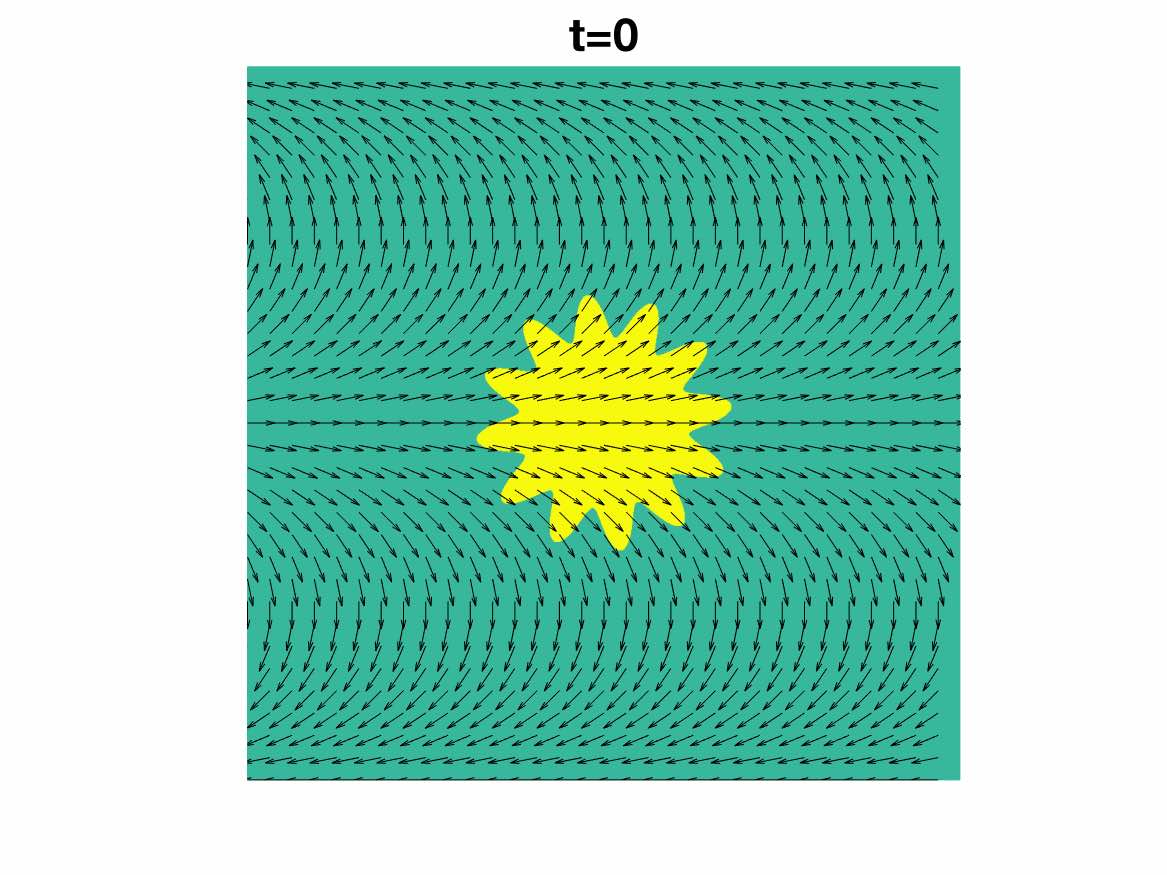} &
\includegraphics[width = 0.24 \textwidth,clip,trim= 5cm 1cm 5cm 0cm]{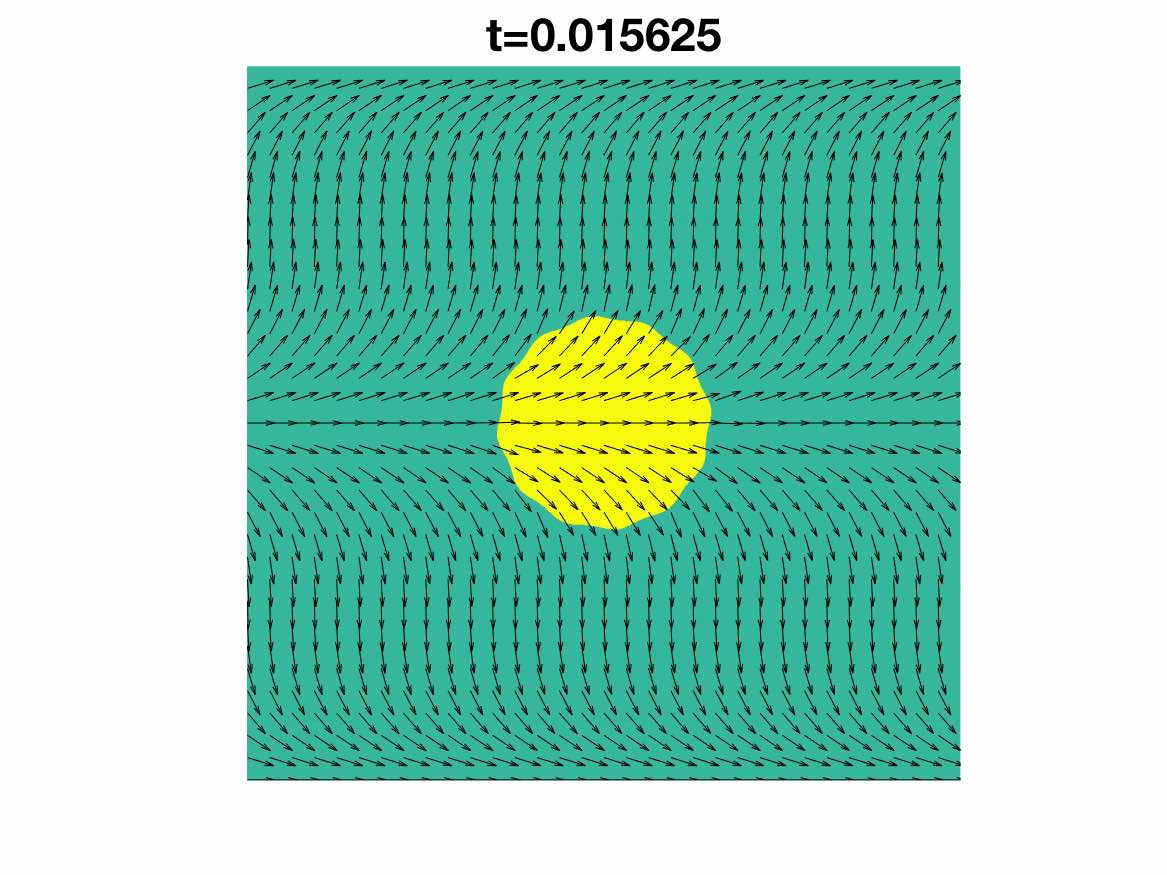}
\includegraphics[width = 0.24 \textwidth,clip,trim= 5cm 1cm 5cm 0cm]{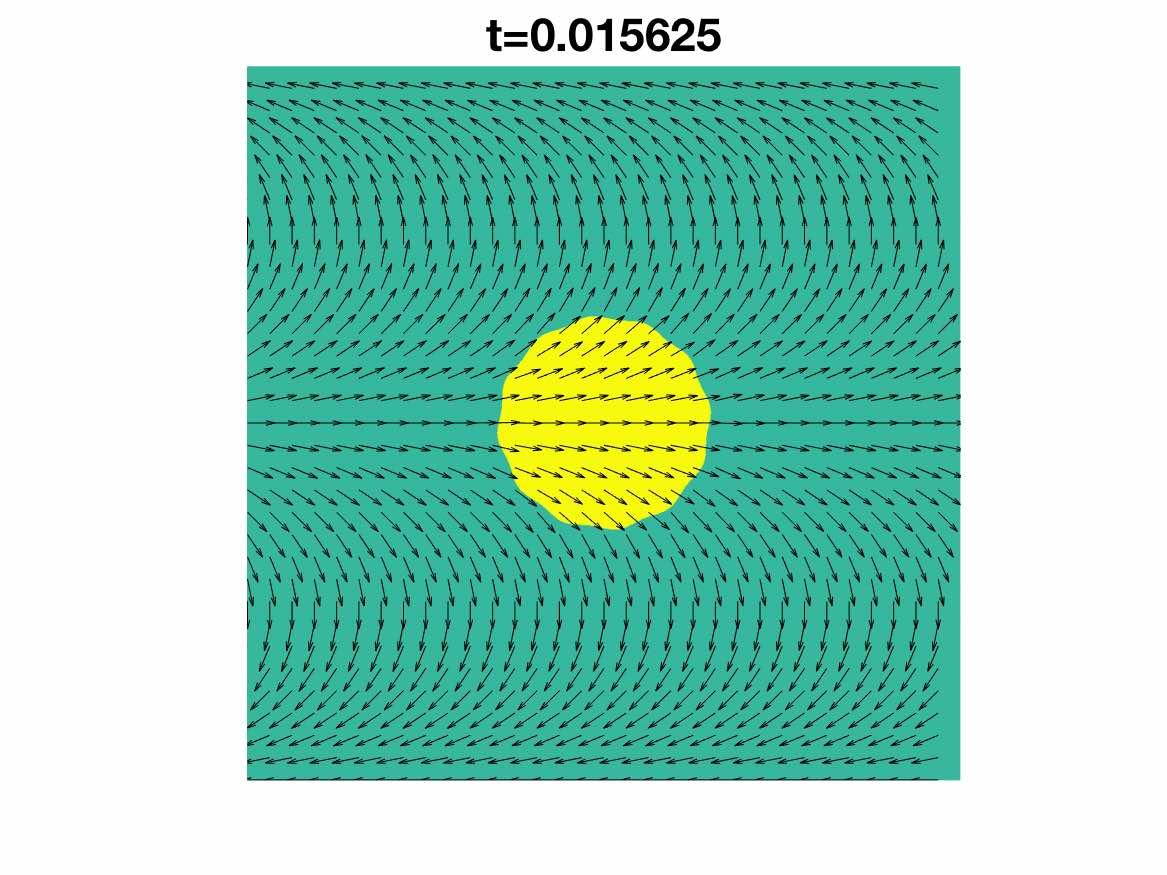} \\ 
\includegraphics[width = 0.24 \textwidth,clip,trim= 5cm 1cm 5cm 0cm]{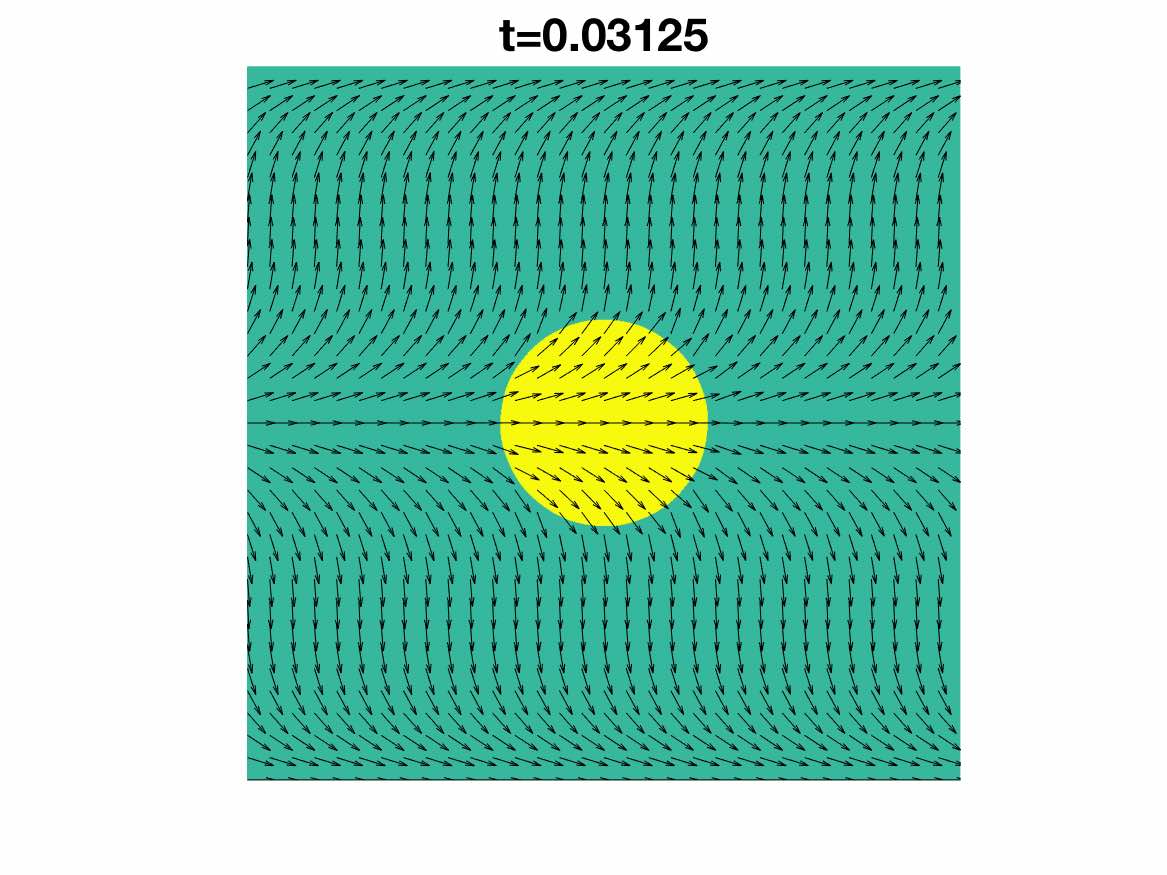} 
\includegraphics[width = 0.24 \textwidth,clip,trim= 5cm 1cm 5cm 0cm]{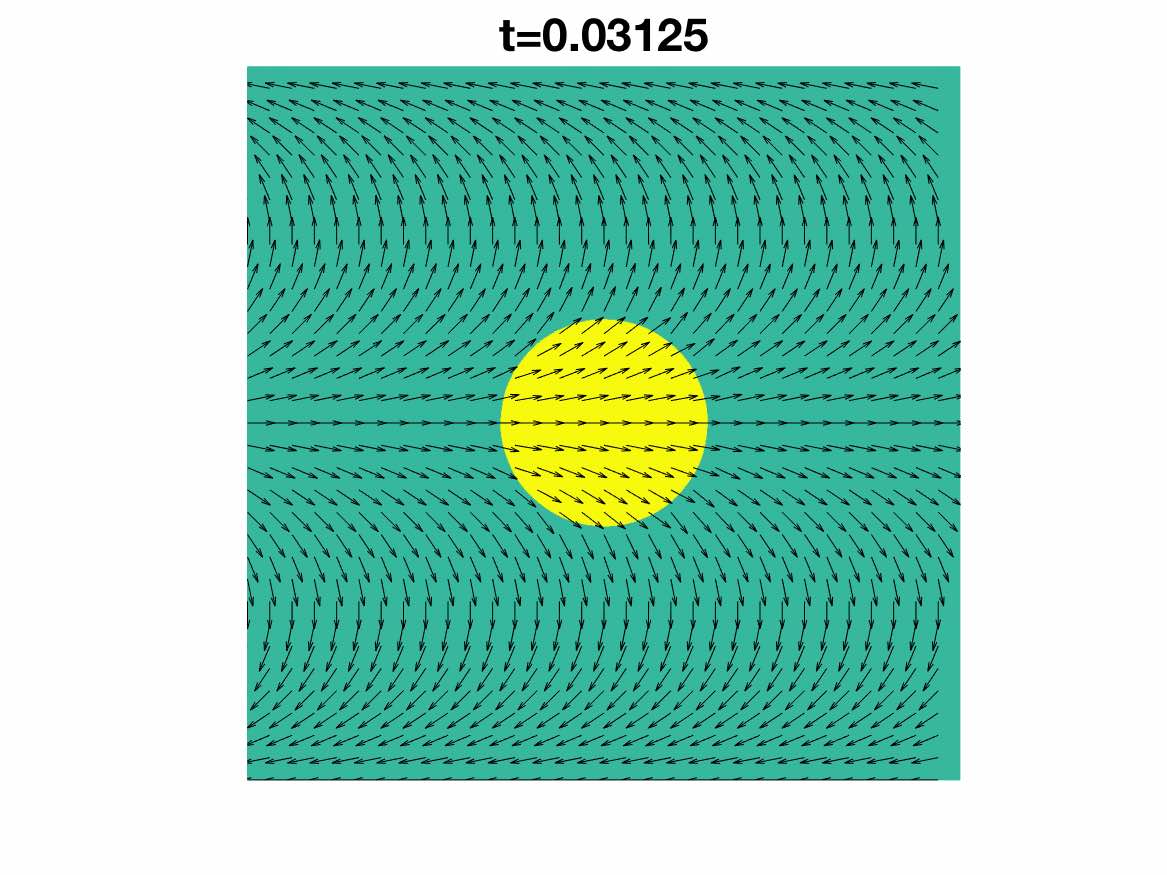} &
\includegraphics[width = 0.24 \textwidth,clip,trim= 5cm 1cm 5cm 0cm]{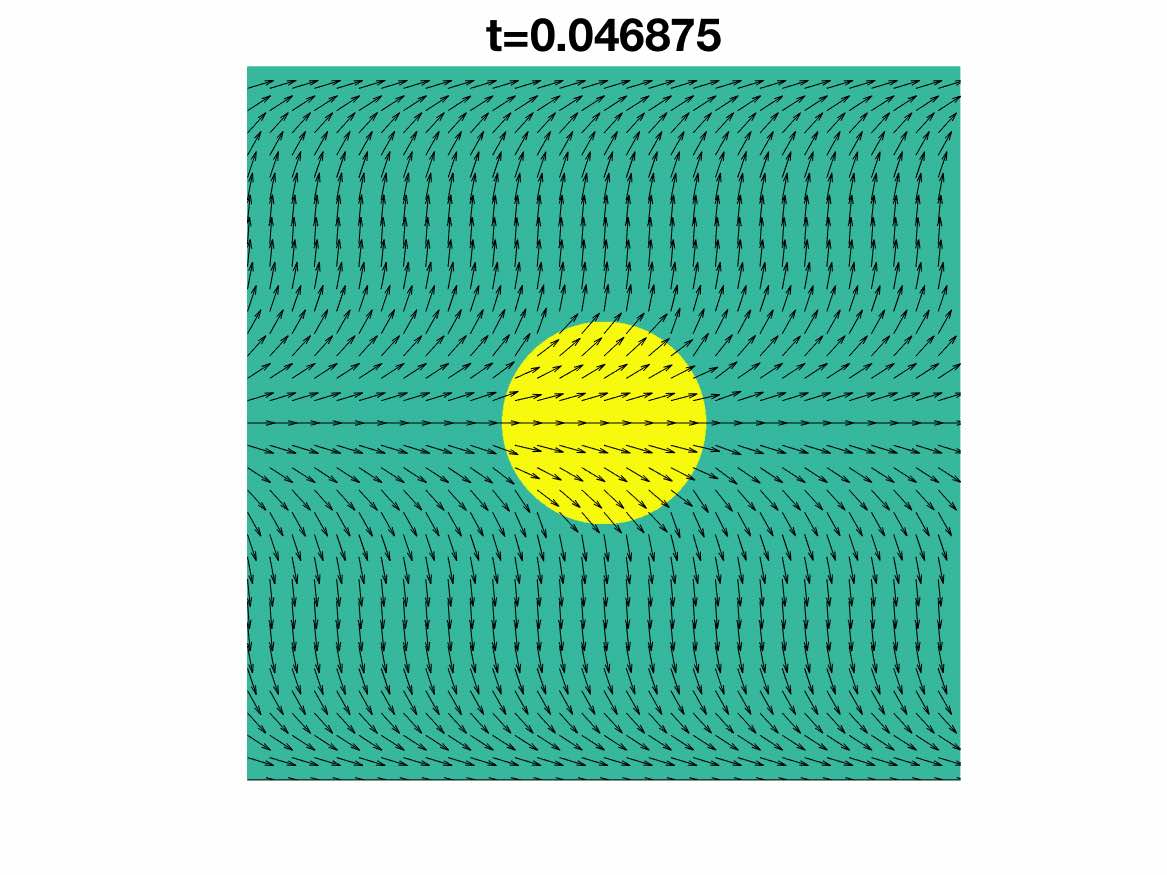}
\includegraphics[width = 0.24 \textwidth,clip,trim= 5cm 1cm 5cm 0cm]{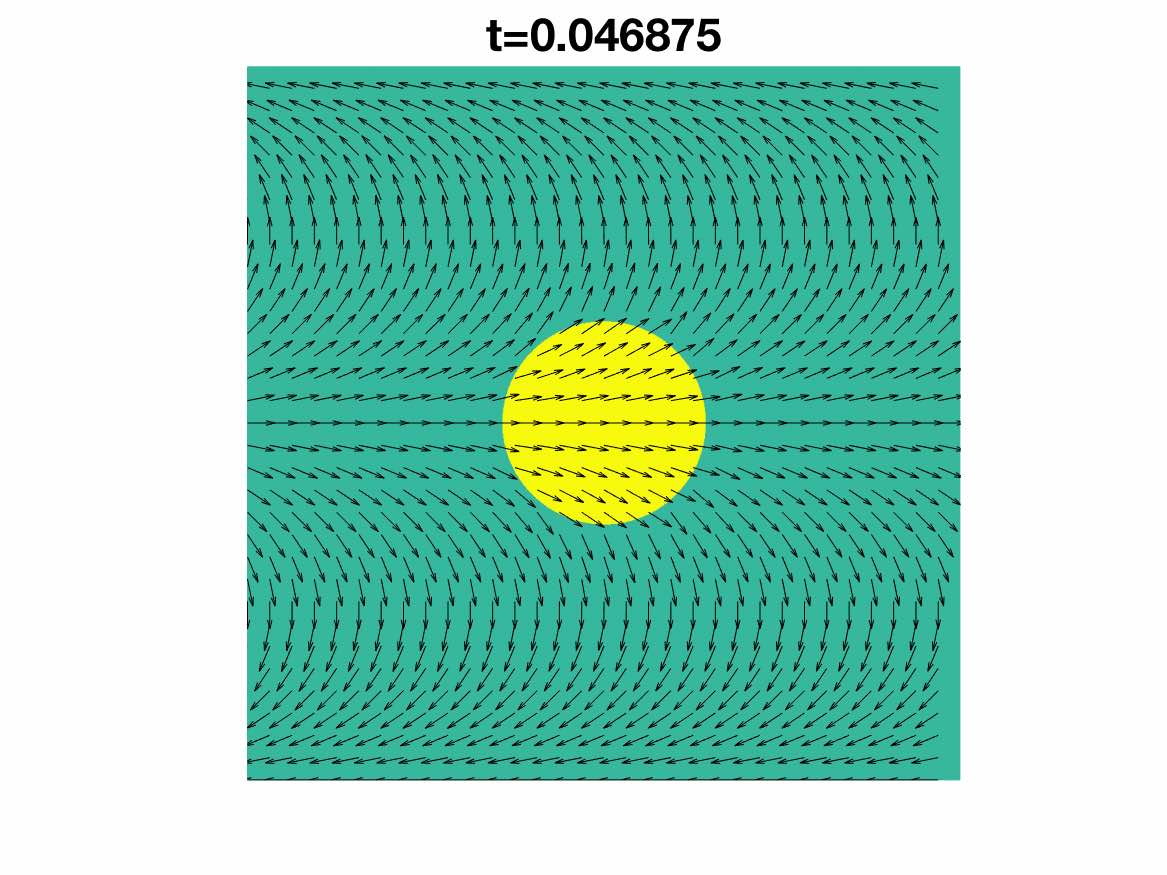} \\
\includegraphics[width = 0.24 \textwidth,clip,trim= 5cm 1cm 5cm 0cm]{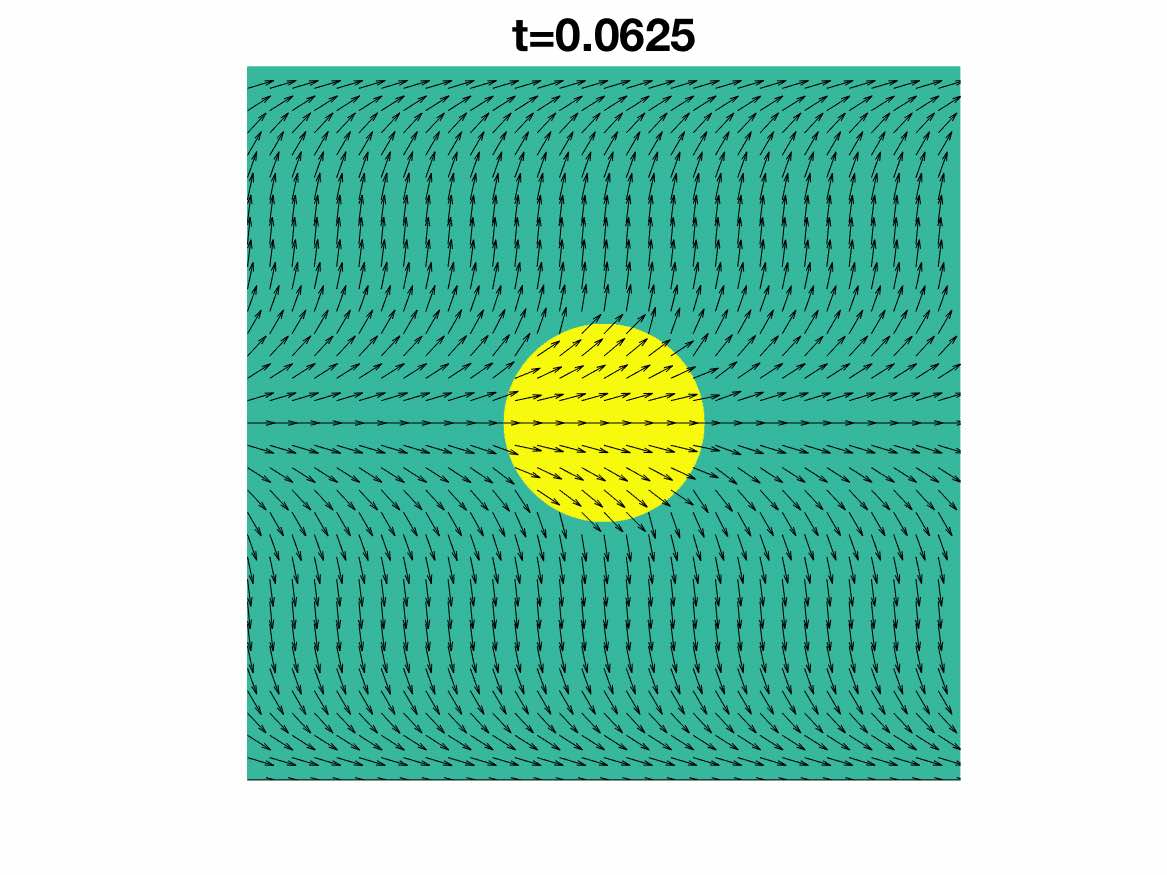} 
\includegraphics[width = 0.24 \textwidth,clip,trim= 5cm 1cm 5cm 0cm]{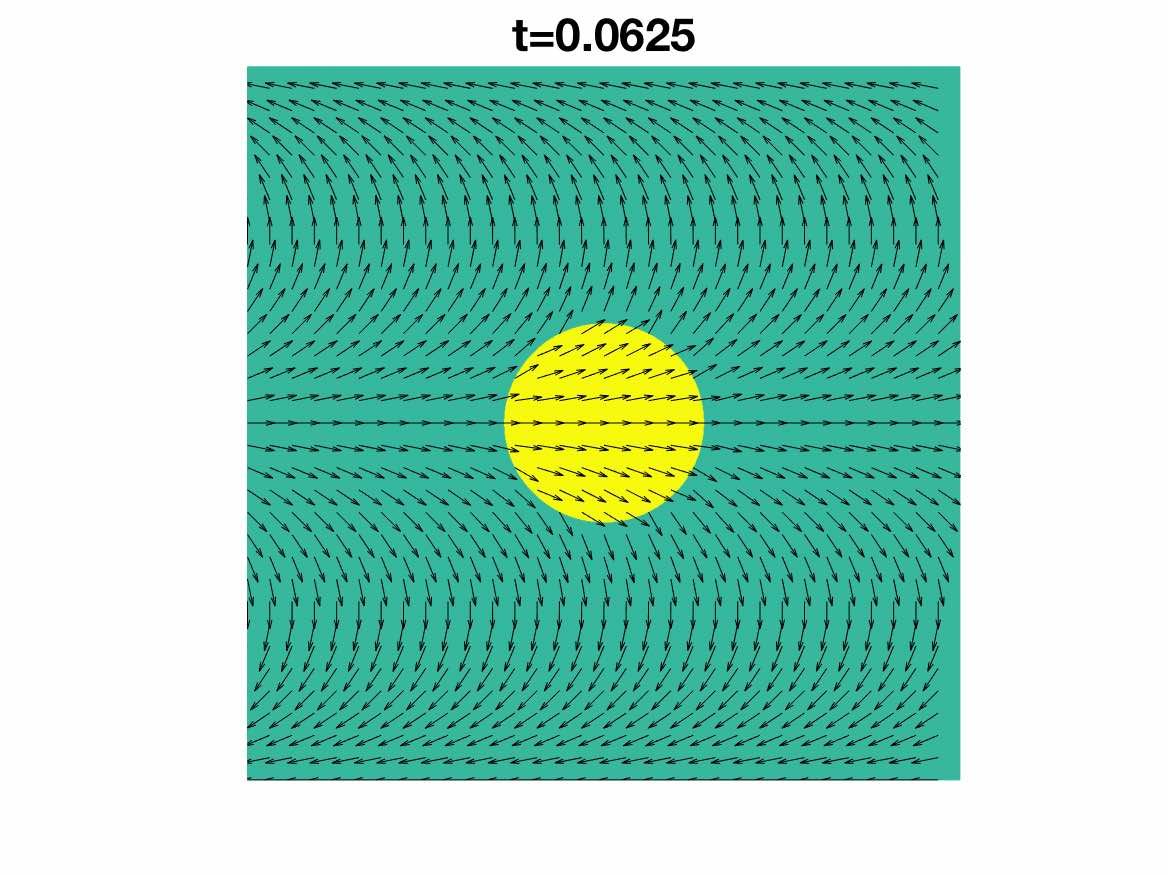} &
\includegraphics[width = 0.24 \textwidth,clip,trim= 5cm 1cm 5cm 0cm]{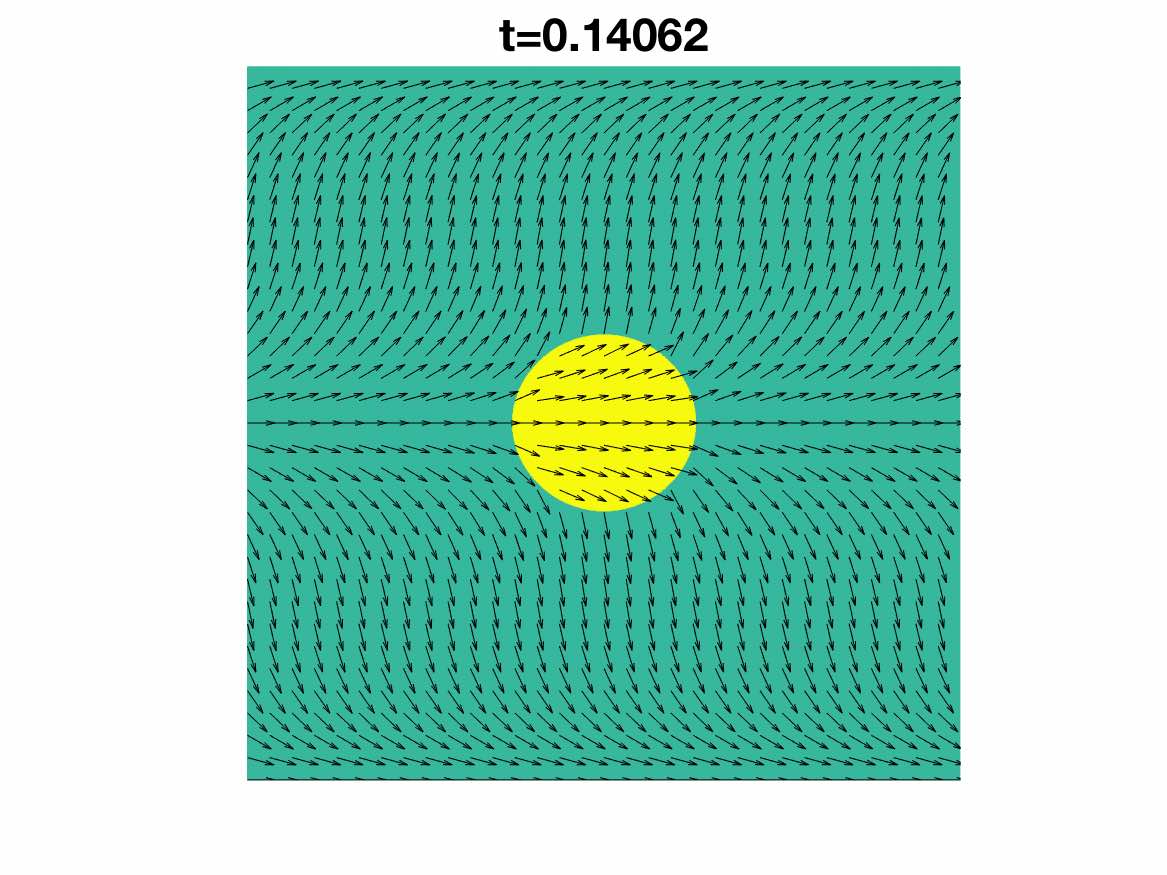}
\includegraphics[width = 0.24 \textwidth,clip,trim= 5cm 1cm 5cm 0cm]{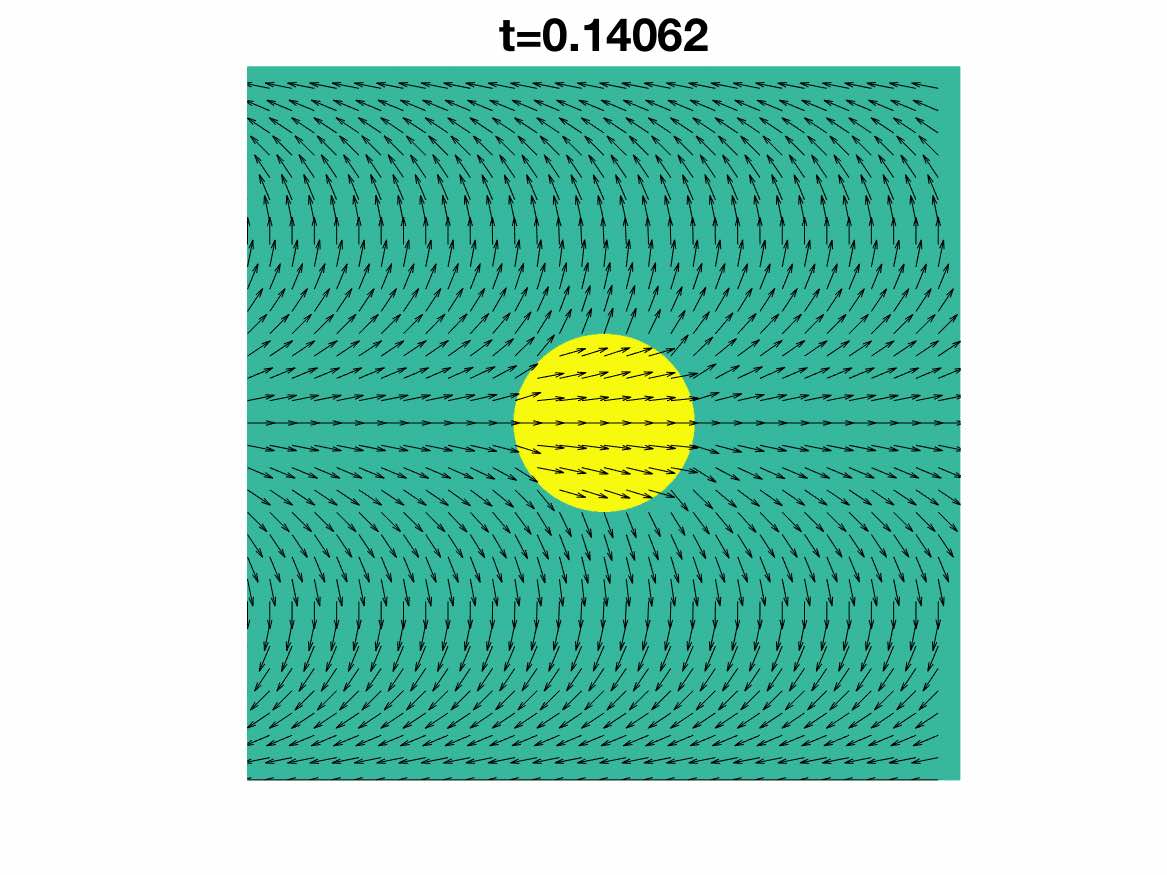} \\
\includegraphics[width = 0.24 \textwidth,clip,trim= 5cm 1cm 5cm 0cm]{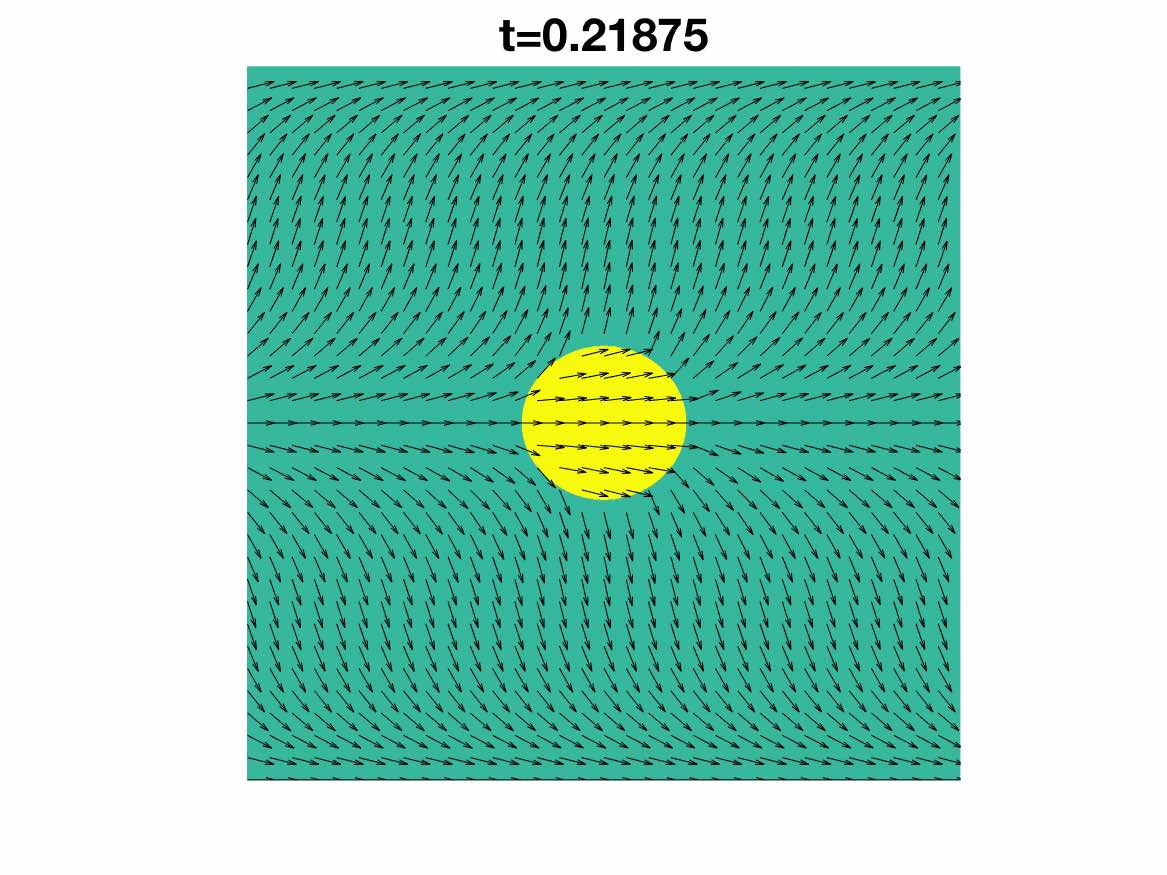} 
\includegraphics[width = 0.24 \textwidth,clip,trim= 5cm 1cm 5cm 0cm]{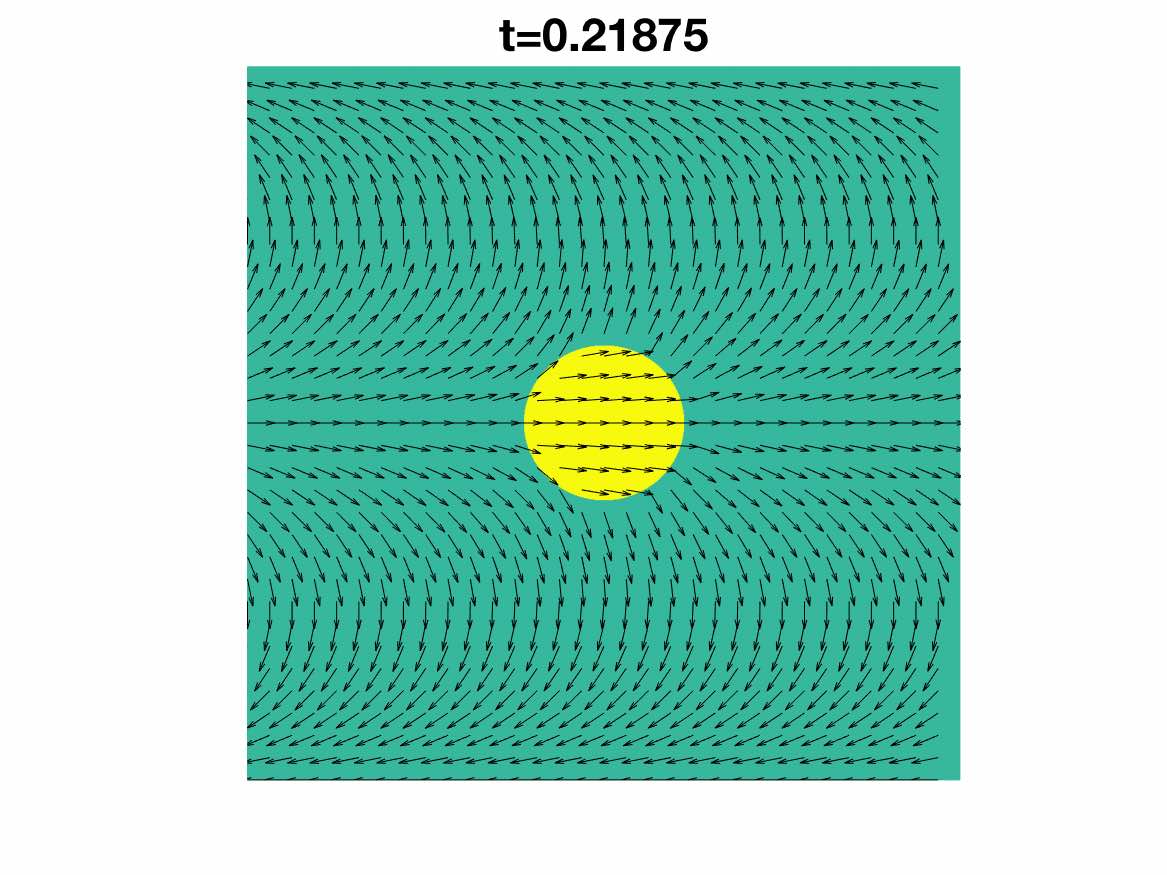} &
\includegraphics[width = 0.24 \textwidth,clip,trim= 5cm 1cm 5cm 0cm]{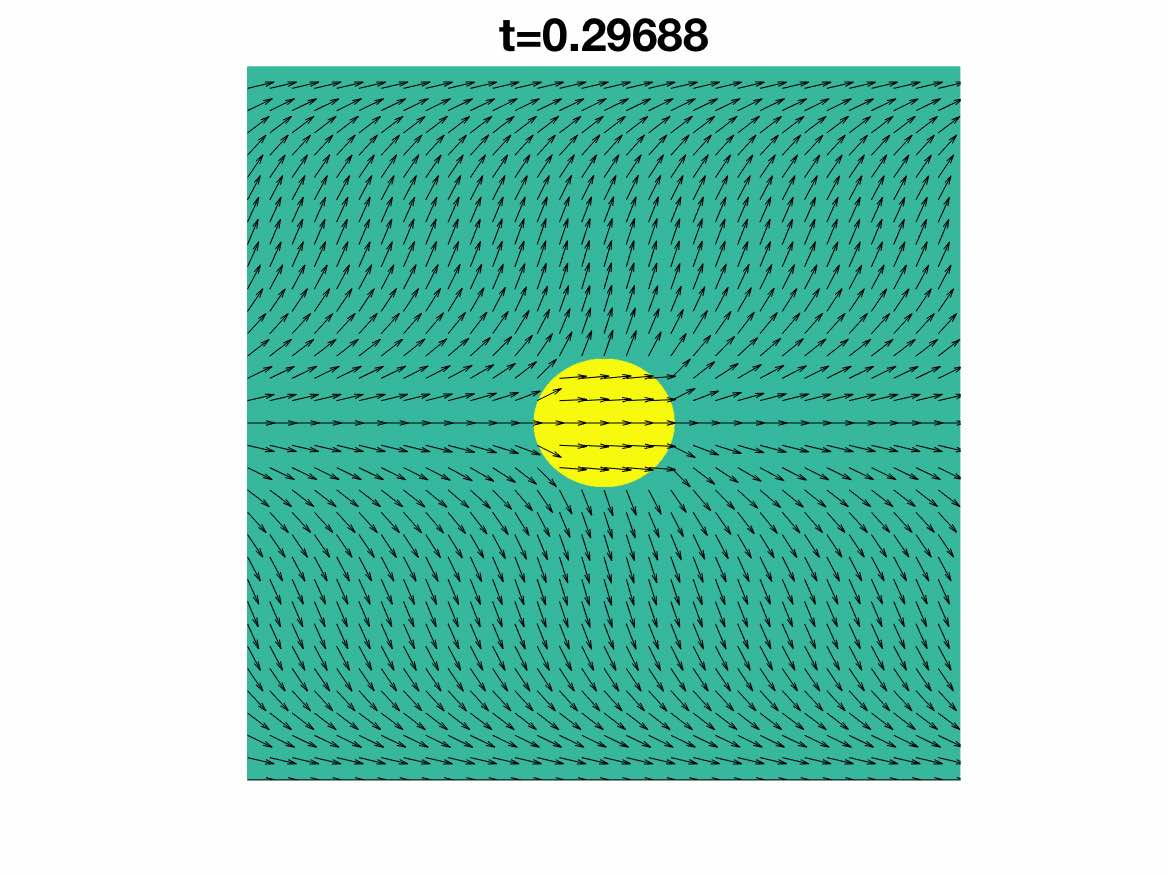}
\includegraphics[width = 0.24 \textwidth,clip,trim= 5cm 1cm 5cm 0cm]{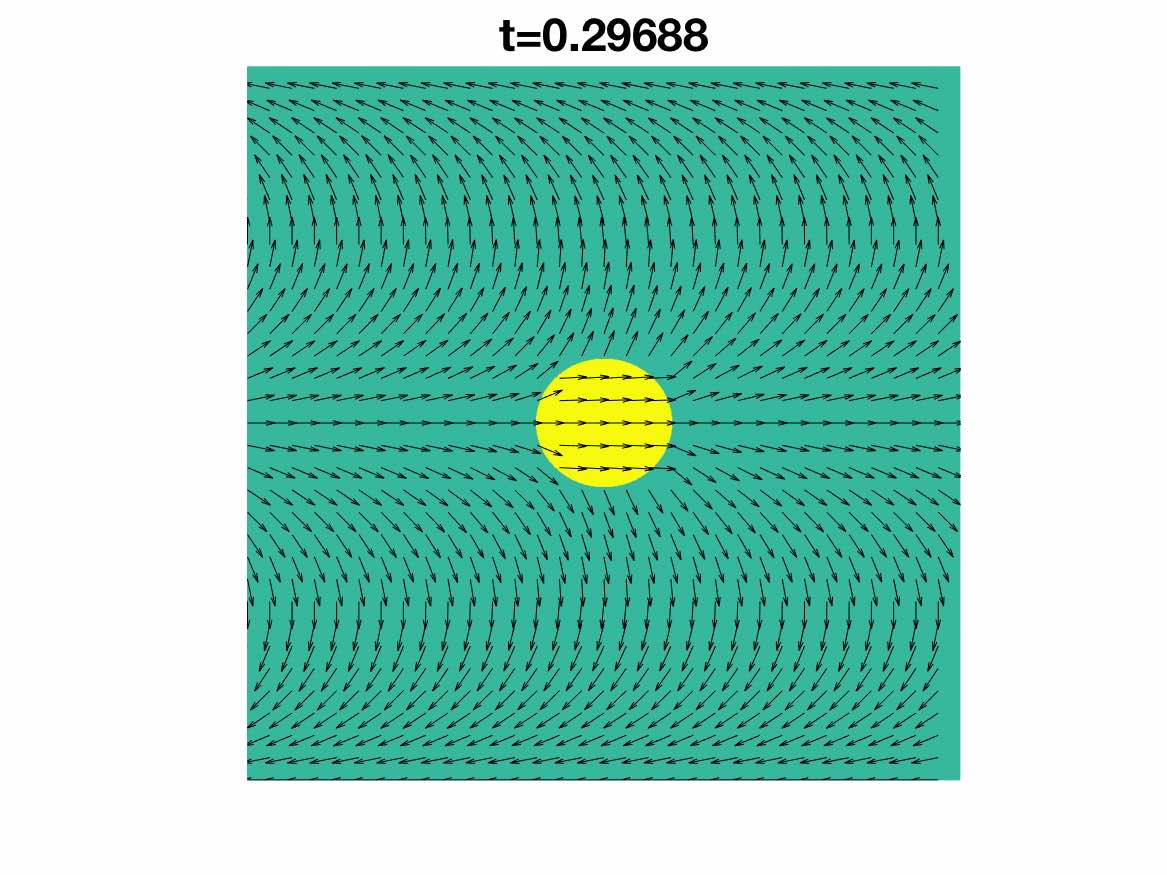} \\
\includegraphics[width = 0.24 \textwidth,clip,trim= 5cm 1cm 5cm 0cm]{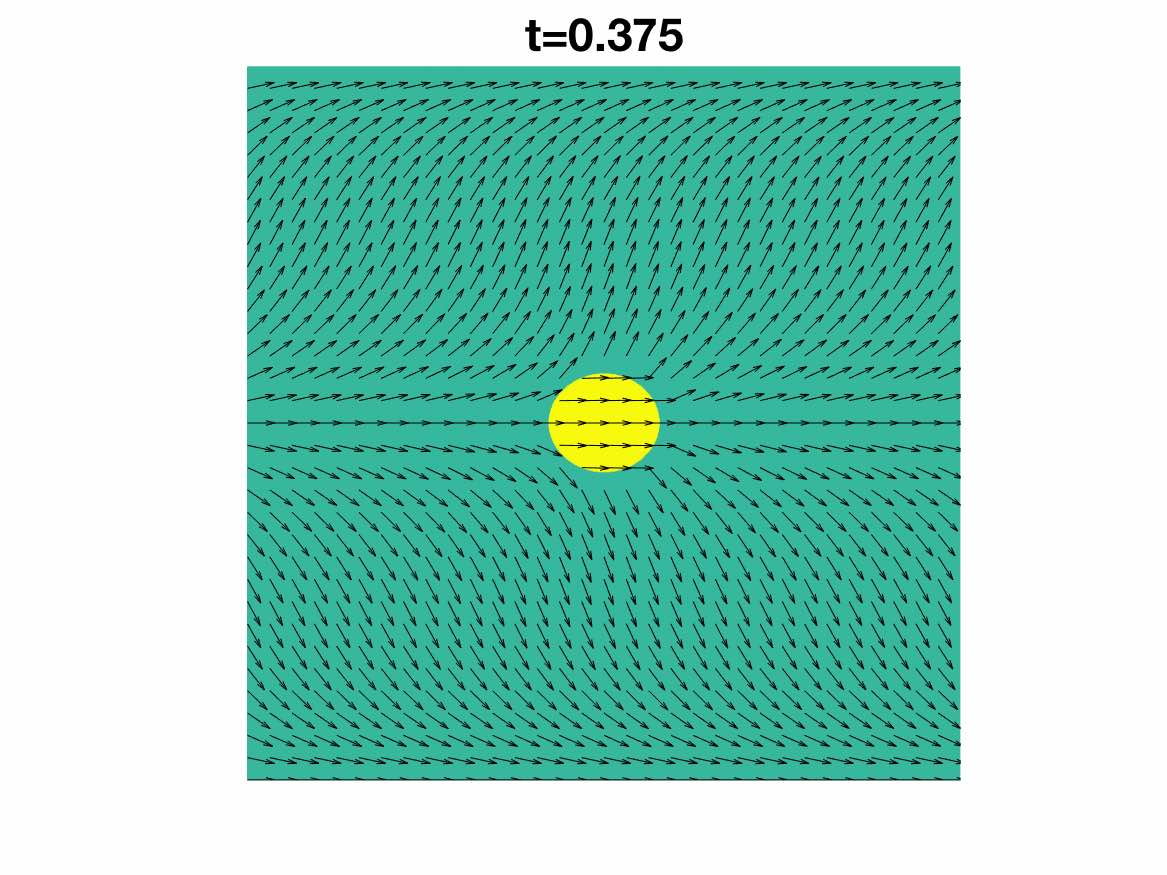} 
\includegraphics[width = 0.24 \textwidth,clip,trim= 5cm 1cm 5cm 0cm]{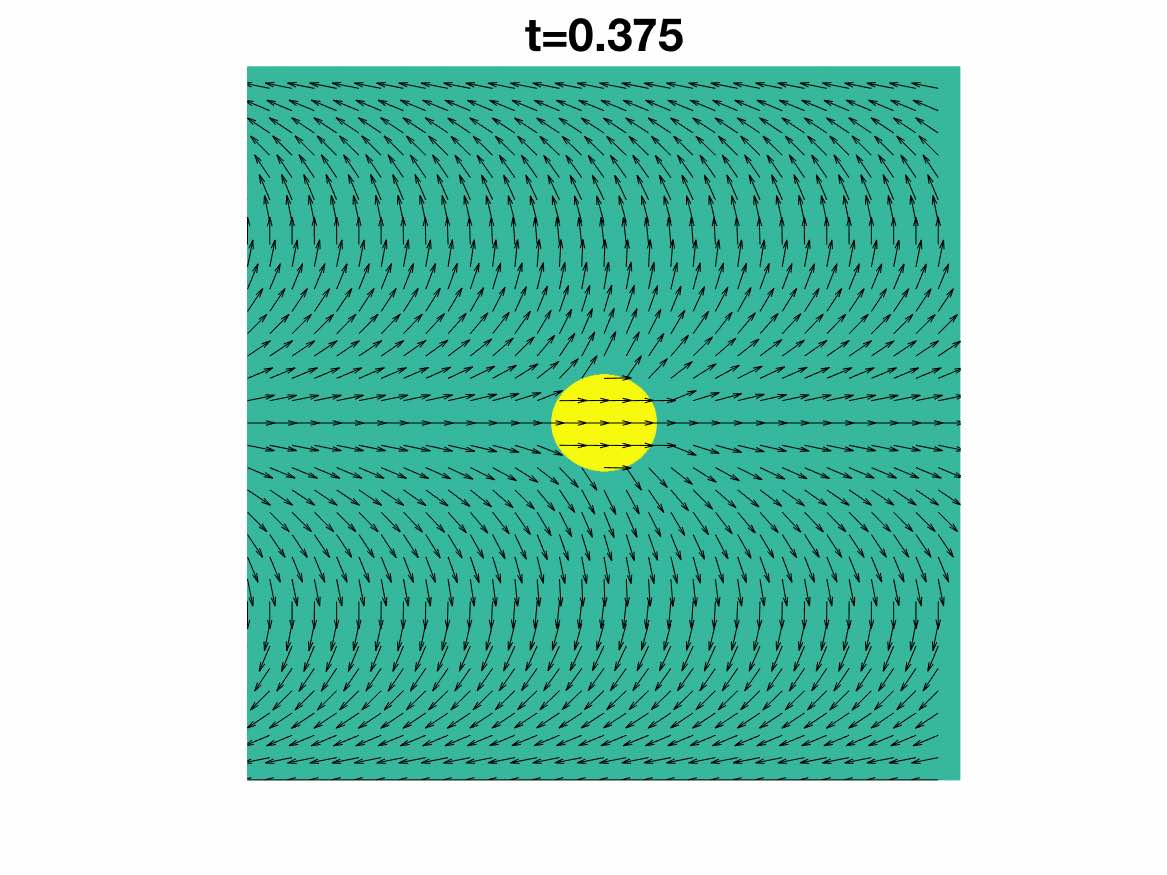} &
\includegraphics[width = 0.24 \textwidth,clip,trim= 5cm 1cm 5cm 0cm]{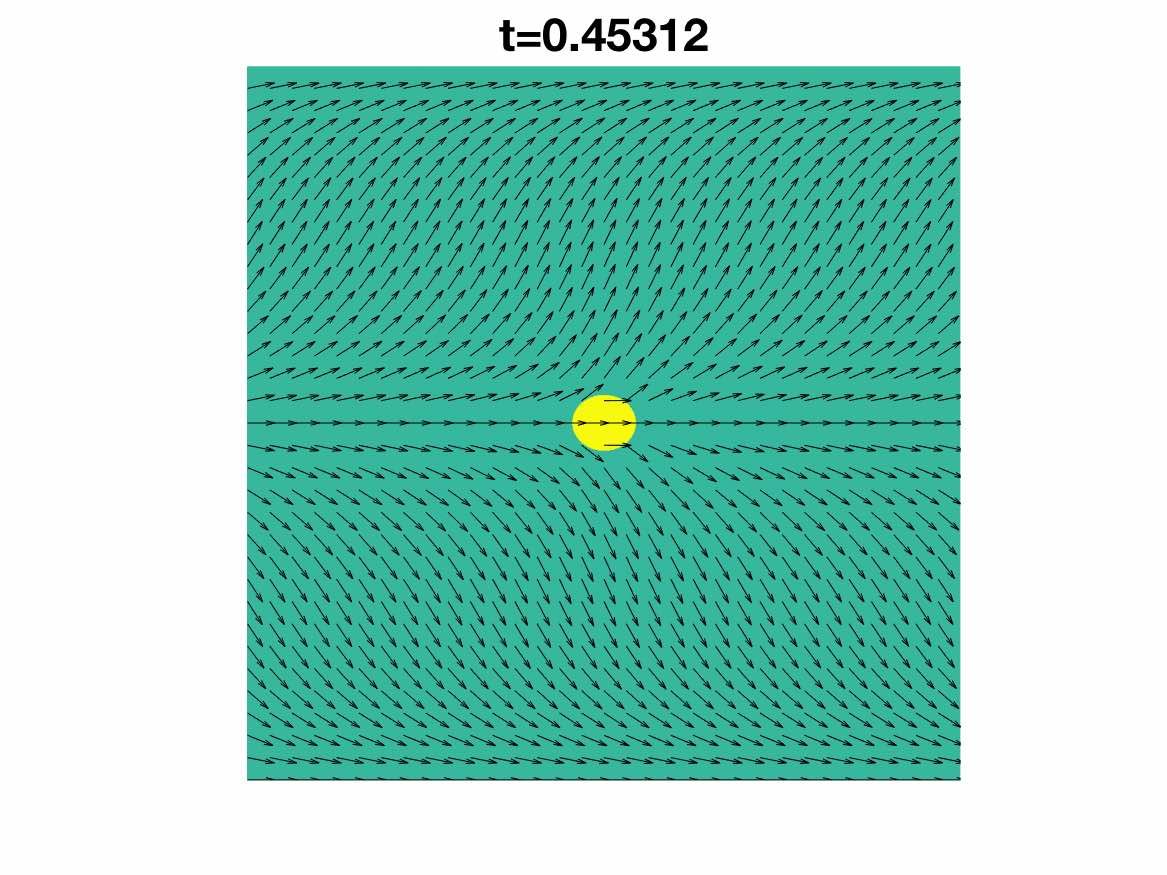}
\includegraphics[width = 0.24 \textwidth,clip,trim= 5cm 1cm 5cm 0cm]{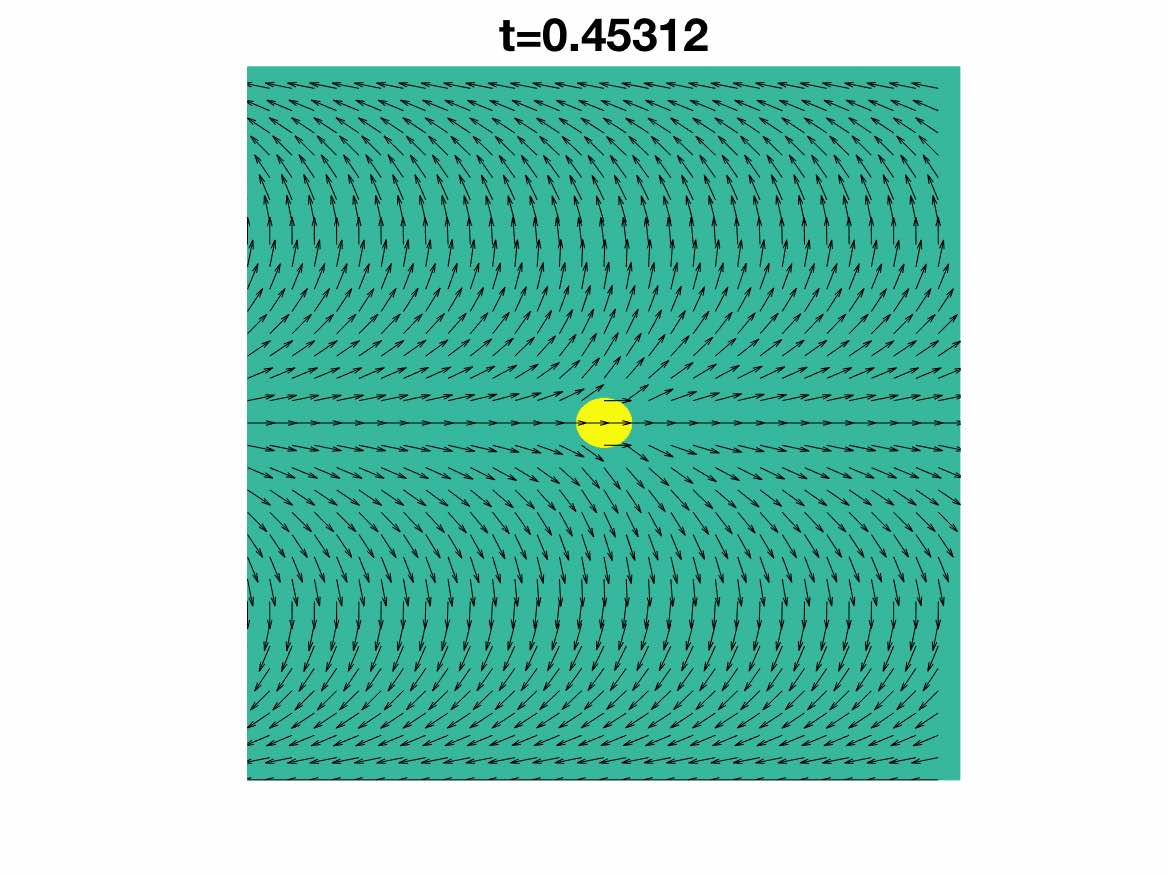} \\
\end{tabular}
\caption{Snapshots of the time evolution of an initial $O_2$ matrix-valued field. The initial line defect is given by $r = 0.15+0.03\sin(12 \theta)$ where $(r,\theta)$ is the corresponding polar coordinate of $(x_1,x_2)$.  The initial field is given in \eqref{eq:initialslow1} with $\eta(x) = \frac{\pi}{2} \sin(2\pi x_1)$ in the first column and $\eta(x) =  2\pi x_1$ in the second column. See Section~\ref{sec:doublefast}.  } \label{fig:3}
\end{figure}

\subsubsection{Evolution of $O_n$-valued fields at the $O(\varepsilon t)$ time scale.}\label{sec:doubleslow} 
In this section, we check the motion law we derived in Section~\ref{sec:part2slow} at the $O(\varepsilon t)$ time scale. 
Note that, at the $O(\varepsilon t)$ time scale, we have the leading order of the curvature of the line defect satisfies $\kappa_0 = 0$ and we have $\Delta \eta = 0$. 
Hence, we perform several experiments where the initial condition has two straight parallel line defects.
Specifically, we choose the following initial condition 
\begin{align}\label{eq:initialslow2}
A(x_1,x_2)  = 
\begin{cases}
\begin{bmatrix}
\cos \eta_1 & -\sin \eta_1 \\
\sin \eta_1 & \cos \eta_1
\end{bmatrix},   & \text{if} \  |x_2|>0.25, \\
\begin{bmatrix}
\cos \eta_2 & \sin \eta_2 \\
\sin \eta_2 & -\cos \eta_2
\end{bmatrix}, & \text{otherwise},
\end{cases}
\end{align}
for different choices of $\eta_1, \eta_2\colon \Omega \rightarrow \mathbb{R}$ satisfying $\Delta \eta_1 = 0$ and $\Delta \eta_2 = 0$.

We first choose $\eta_1$ and $\eta_2$ so that $\eta_{1s}^2 = \eta_{2s}^2$; the parallel line defects are stationary according to our analytical results in \eqref{eq:motionlawslow} in Section~\ref{sec:part2slow}. 
In Figure~\ref{fig:4}, from the left to the right, we set 
\begin{align*}
& \eta_1(x_1,x_2)= \eta_2(x_1,x_2)= 1 \\
& \eta_1(x_1,x_2)= \eta_2(x_1,x_2)= 2\pi x_1 \\
& \eta_1(x_1,x_2)= \eta_2(x_1,x_2)= 4\pi x_1 \\
& \eta_1(x_1,x_2) = 2\pi x_1,  \quad \eta_2(x_1,x_2) = -2\pi x_1.
\end{align*}
Indeed, in Figure~\ref{fig:4}, all four columns show that the parallel line defects are stationary. 

\begin{figure}[ht]
\includegraphics[width = 0.24 \textwidth,clip,trim= 2cm 1cm 3cm 0cm]{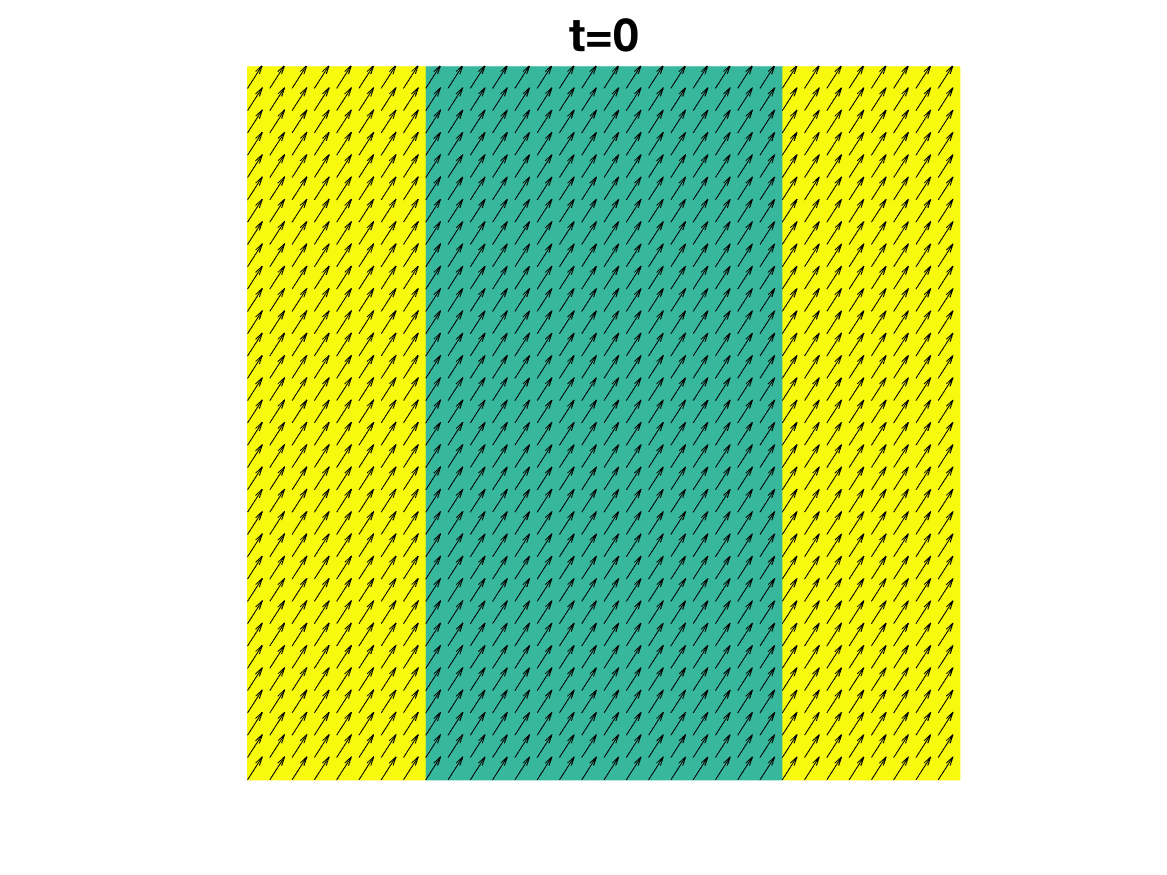} 
\includegraphics[width = 0.24 \textwidth,clip,trim= 2cm 1cm 3cm 0cm]{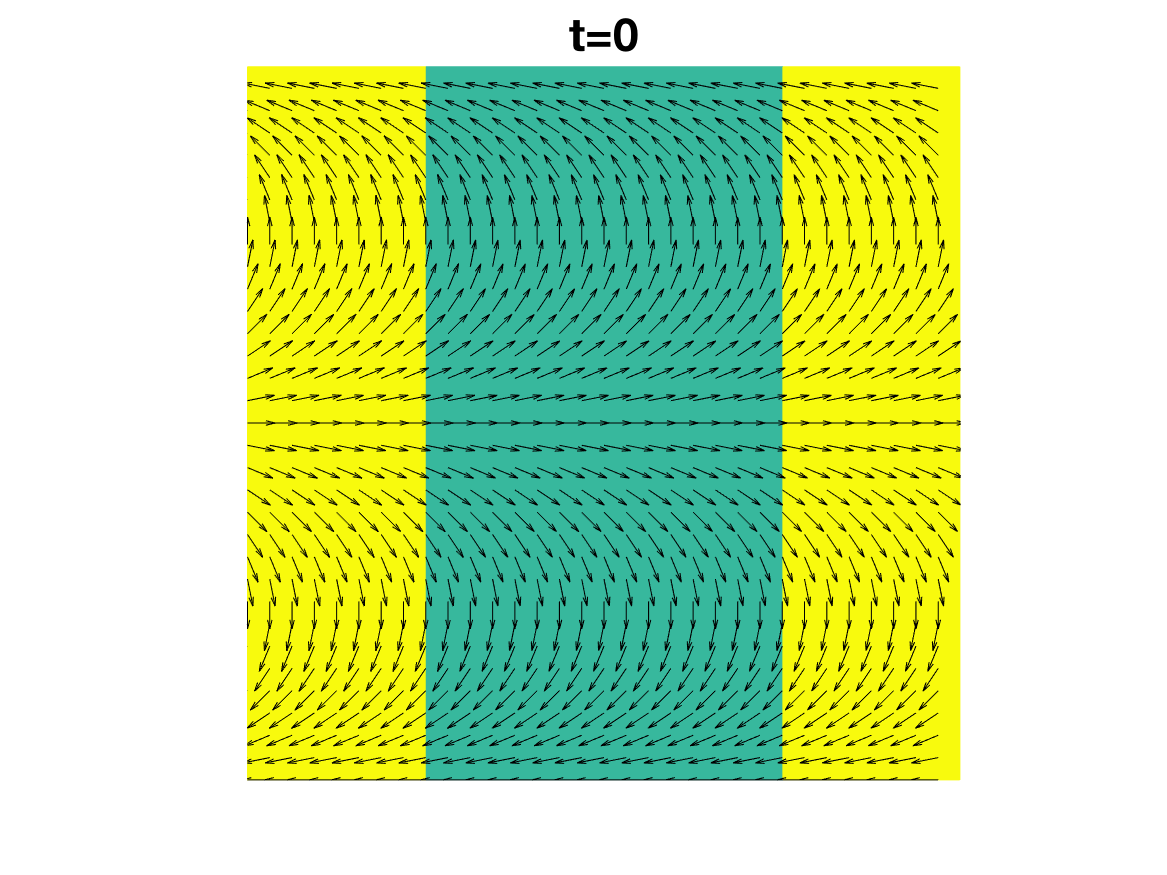} 
\includegraphics[width = 0.24 \textwidth,clip,trim= 2cm 1cm 3cm 0cm]{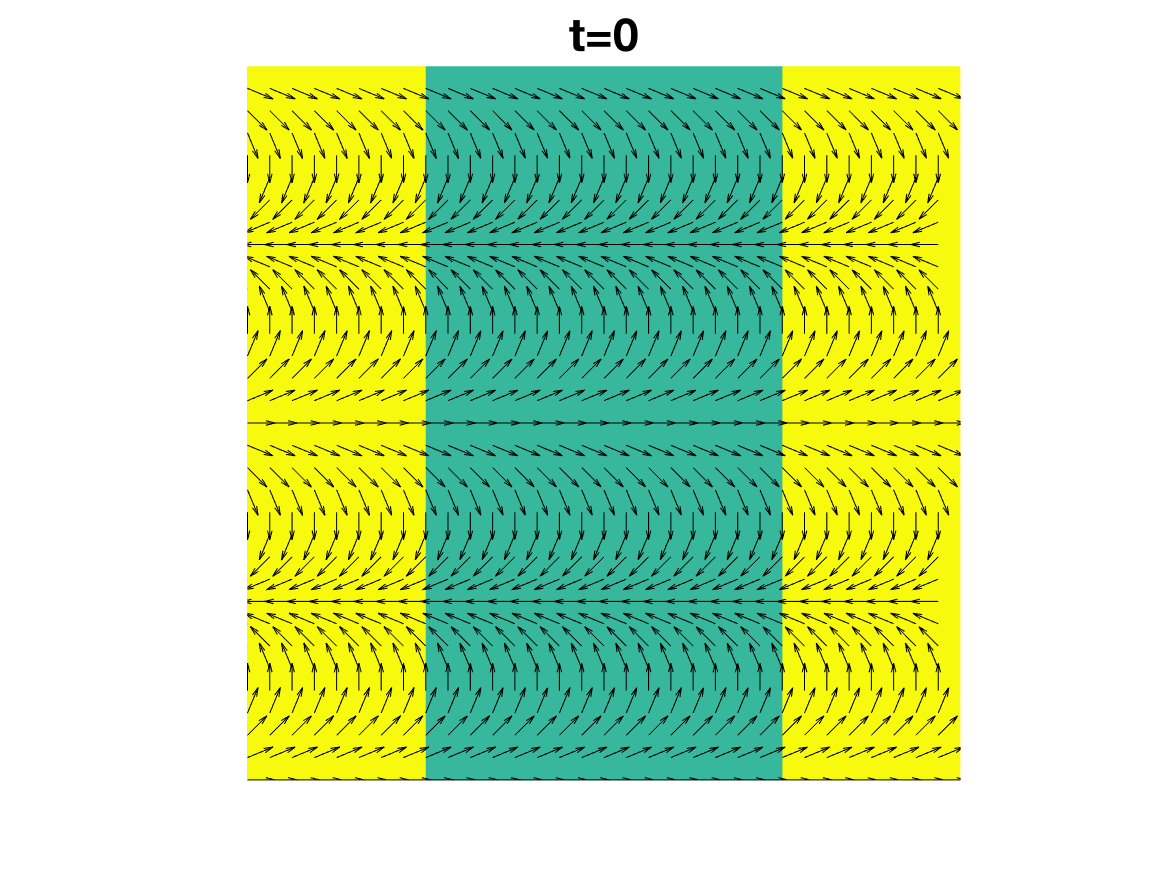}
\includegraphics[width = 0.24 \textwidth,clip,trim= 2cm 1cm 3cm 0cm]{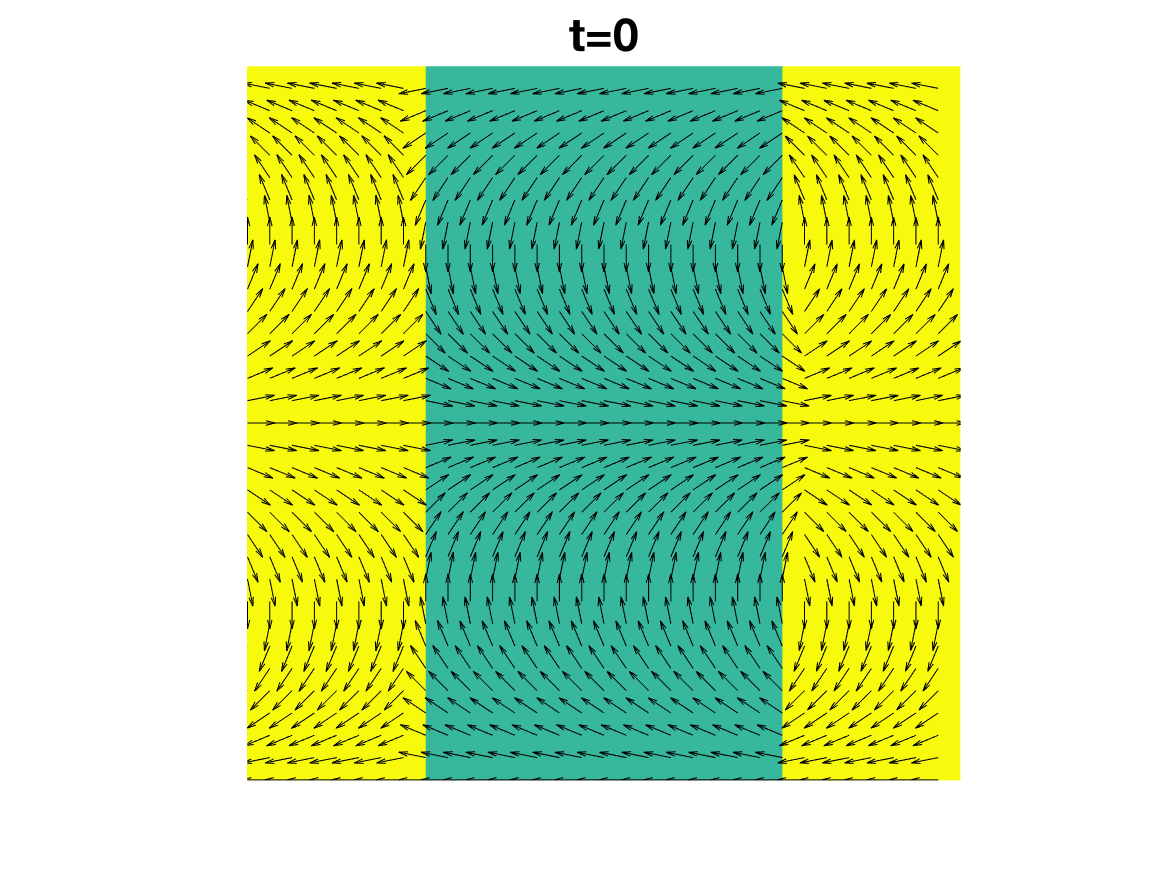}\\
\includegraphics[width = 0.24 \textwidth,clip,trim= 2cm 1cm 3cm 0cm]{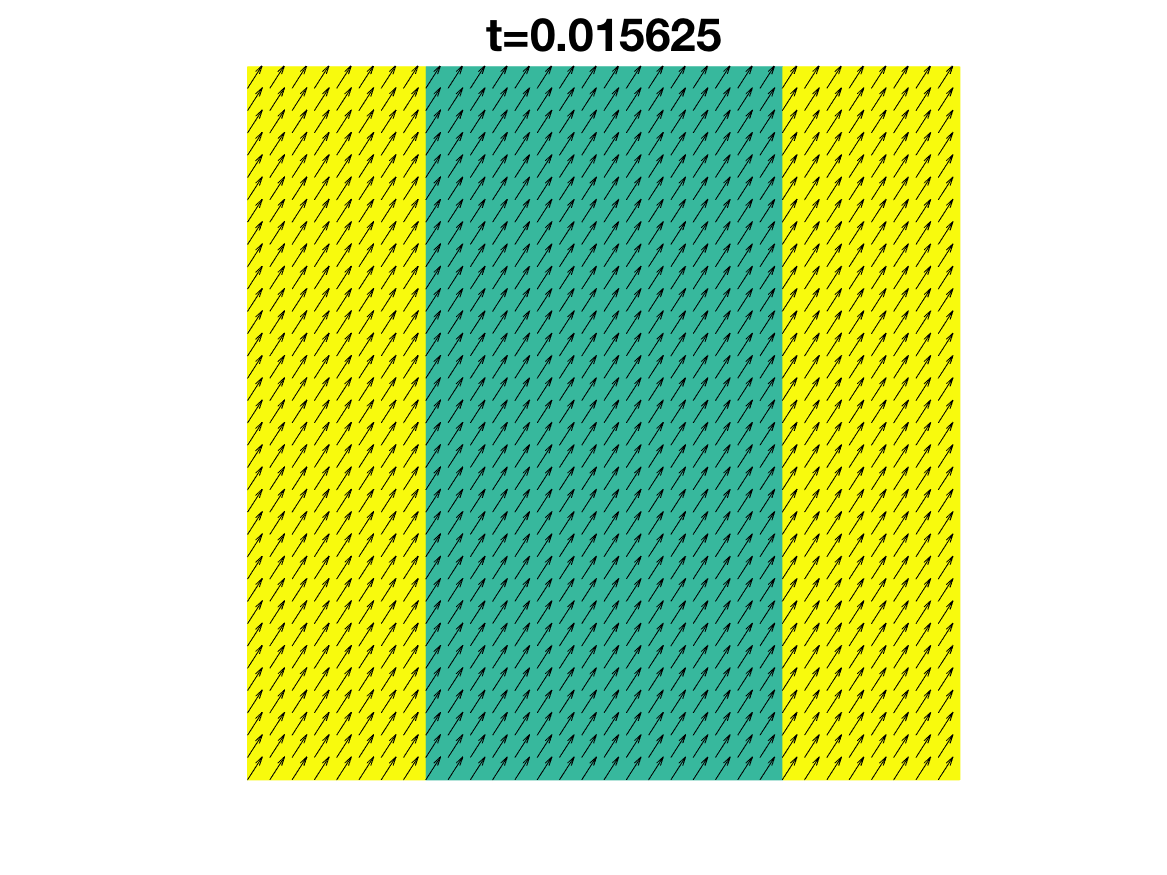} 
\includegraphics[width = 0.24 \textwidth,clip,trim= 2cm 1cm 3cm 0cm]{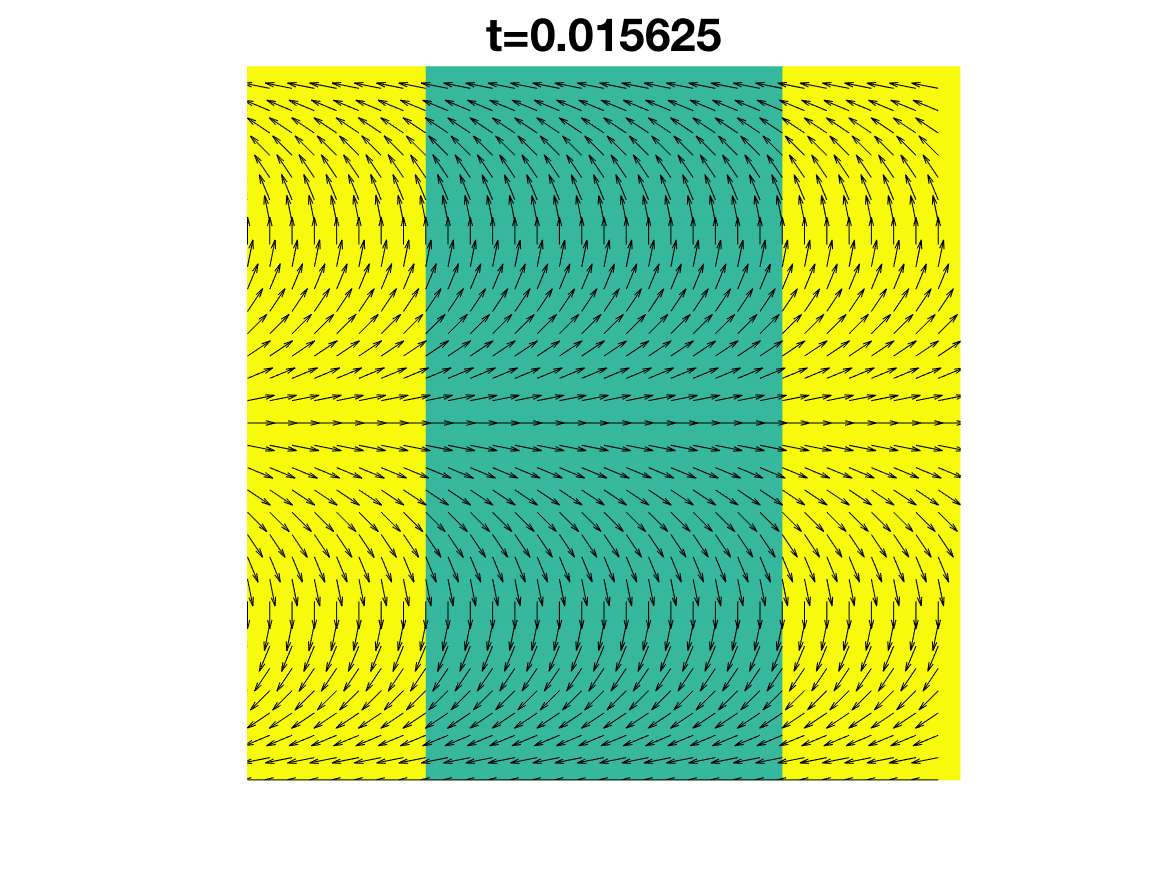} 
\includegraphics[width = 0.24 \textwidth,clip,trim= 2cm 1cm 3cm 0cm]{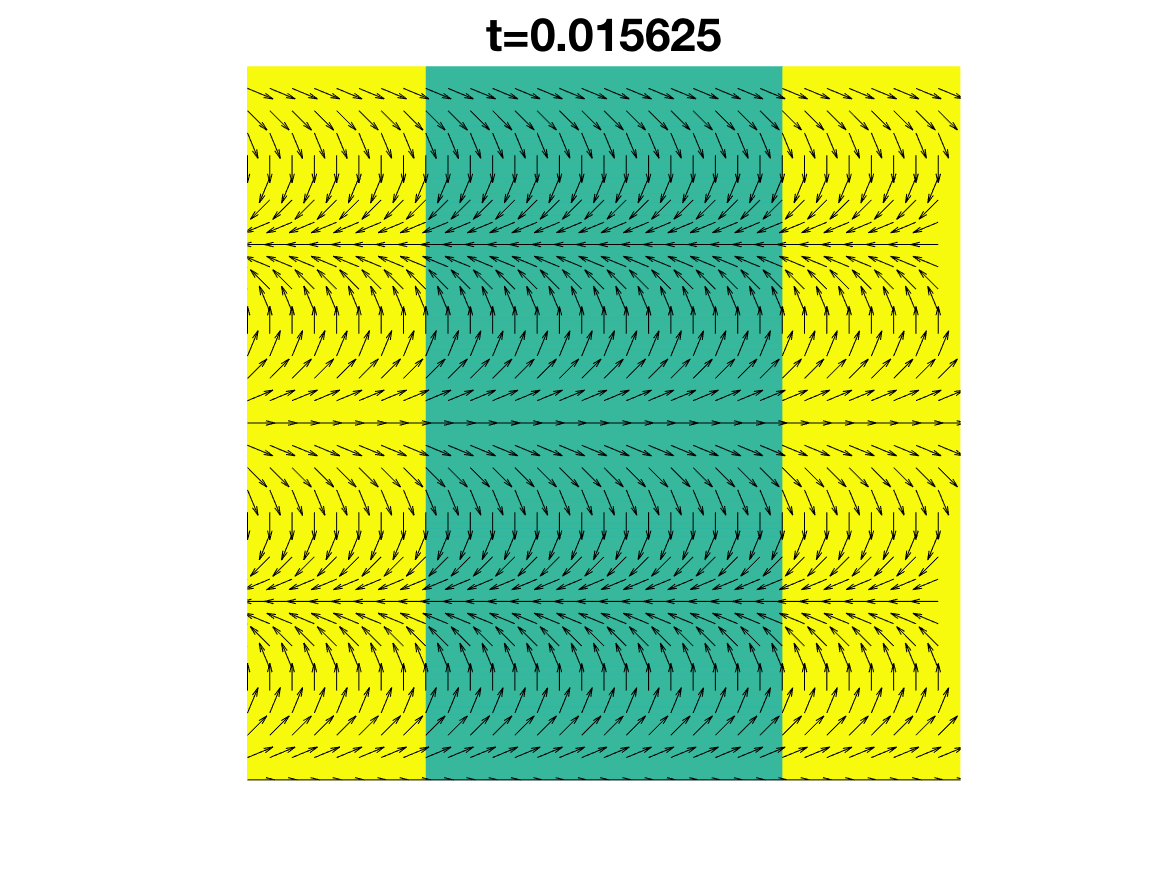} 
\includegraphics[width = 0.24 \textwidth,clip,trim= 2cm 1cm 3cm 0cm]{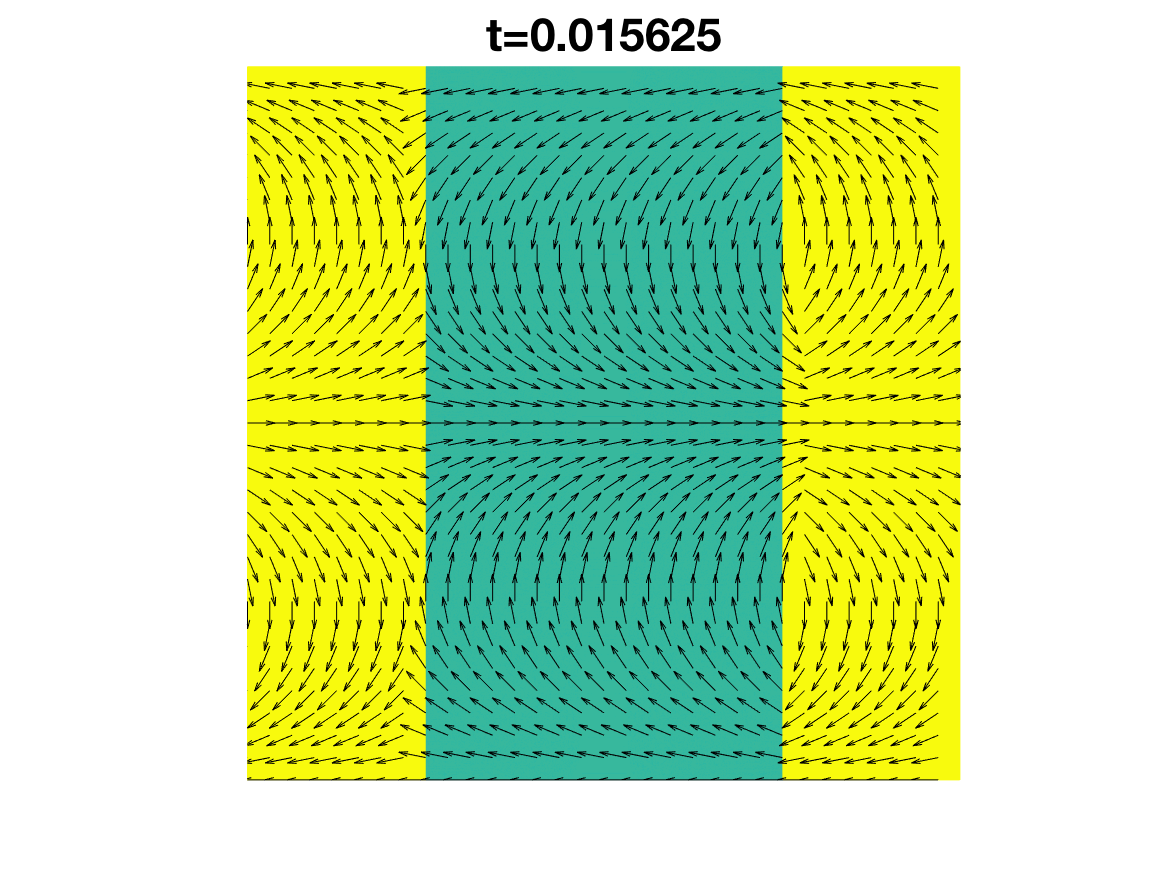}
\caption{Snapshots of the time evolution of an initial $O_2$ matrix-valued field with parallel line defects. From left to right, the initial fields are given in \eqref{eq:initialslow2} with $\eta_1(x_1,x_2) = \eta_2(x_1,x_2) = 1, 2\pi x_1, 4\pi x_1$, and $\eta_1(x_1,x_2) = 2 \pi x_1,  \eta_2(x_1,x_2) = -2 \pi x_1$. See Section~\ref{sec:doubleslow}.} \label{fig:4}
\end{figure}

Figures~\ref{fig:5} and \ref{fig:6} display snapshots of the time dynamics of the interface at different times for two choices of the phases $\eta_1$ and $\eta_2$. 
In Figure~\ref{fig:5}, we set the initial phases to be $\eta_1 = 2 \pi x_1$ and $\eta_2 = 4 \pi x_1$. 
In Figure~\ref{fig:6}, we set the initial phases to be $\eta_1 = 2 \pi x_1$ and $\eta_2 = 8 \pi x_1$. 
Thus, we have $[\eta_s^2]_\Gamma = -12 \pi^2$ in Figure~\ref{fig:5} and  $[\eta_s^2]_\Gamma = -60 \pi^2$ in Figure~\ref{fig:6}. In both figures, we observe that the straight line defects have nonzero speed along their normal directions. Comparing Figure~\ref{fig:5} and Figure~\ref{fig:6}, we observe that the speed of the line defect in Figure~\ref{fig:6} is about five  times the speed of the line defect in Figure~\ref{fig:5}. 
All these observations are consistent with our analytical results on the motion law \eqref{eq:motionlawslow} in Section~\ref{sec:doubleslow}.

Figure~\ref{fig:7} displays the snapshots of the time dynamics of the interface at different times where the initial phases are give by $\eta_1 = 8 \pi x_1$ and $\eta_2 = 2\pi x_1$. We observe the dynamics in Figure~\ref{fig:7} has opposite direction than in Figure~\ref{fig:6}.  This is also consistent with our analytical result on the motion law \eqref{eq:motionlawslow} in Section~\ref{sec:doubleslow}.

\begin{figure}[ht]
\includegraphics[width = 0.24 \textwidth,clip,trim= 5cm 1cm 5cm 0cm]{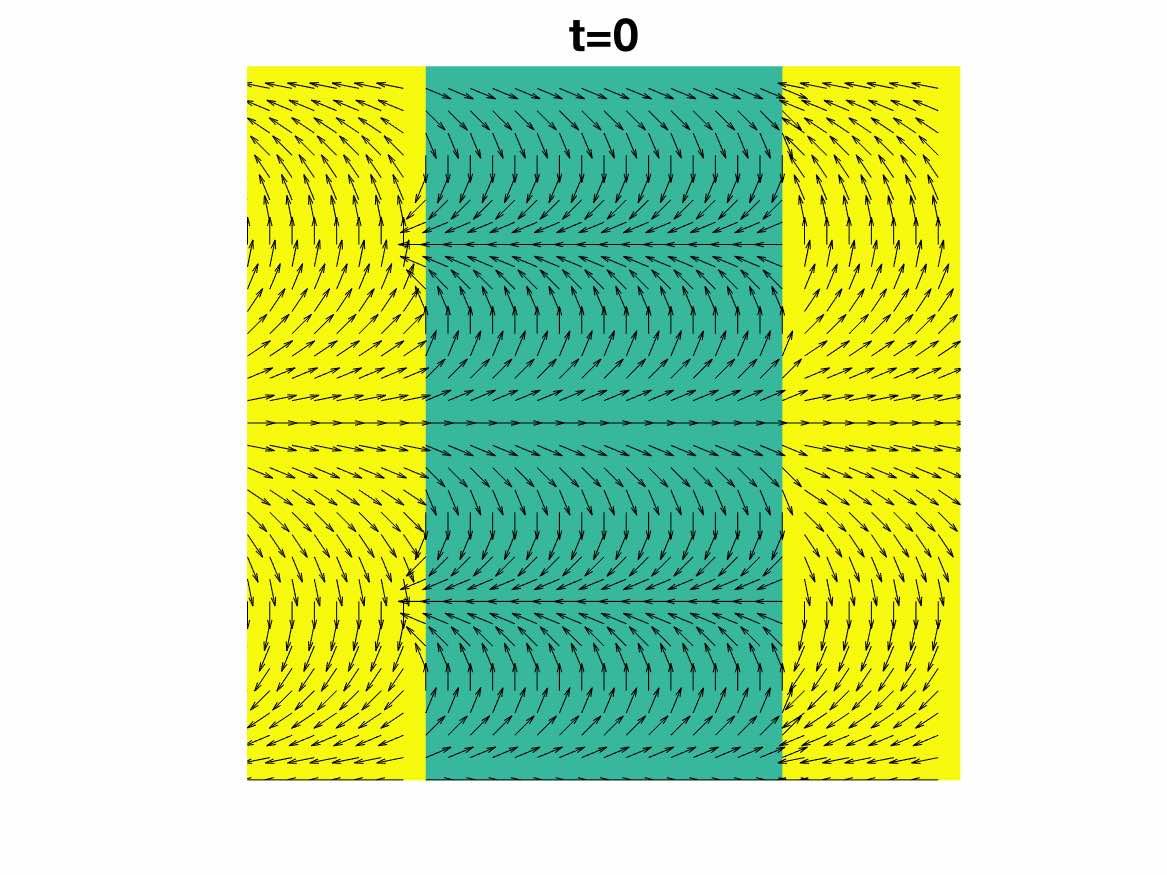}
\includegraphics[width = 0.24 \textwidth,clip,trim= 5cm 1cm 5cm 0cm]{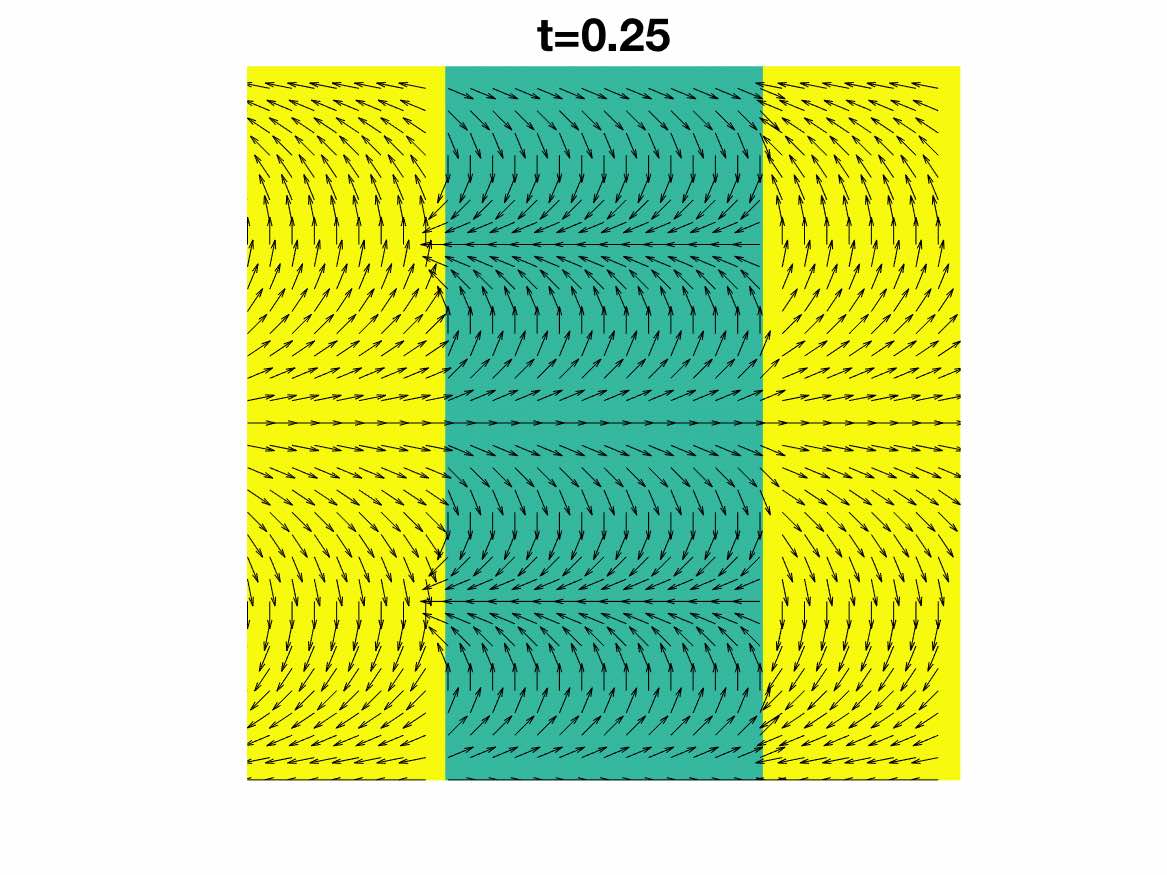}
\includegraphics[width = 0.24 \textwidth,clip,trim= 5cm 1cm 5cm 0cm]{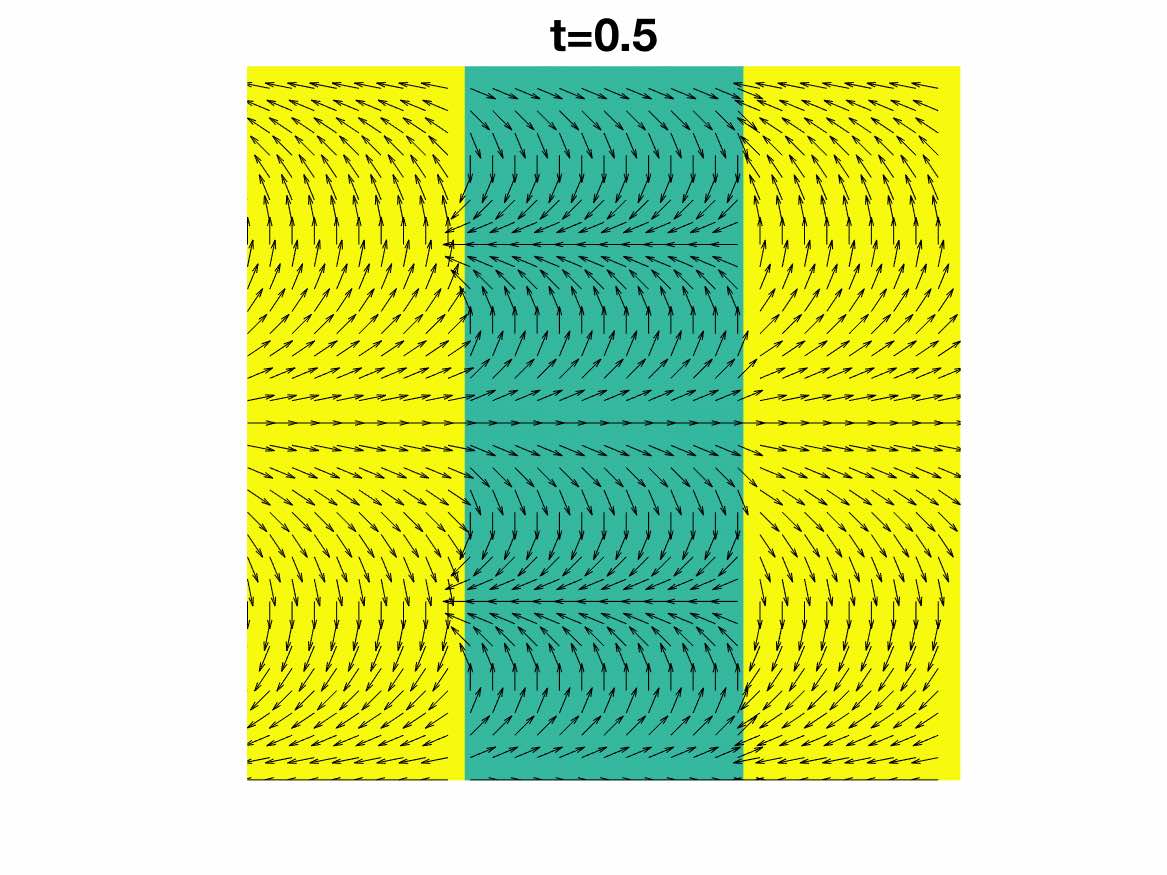}
\includegraphics[width = 0.24 \textwidth,clip,trim= 5cm 1cm 5cm 0cm]{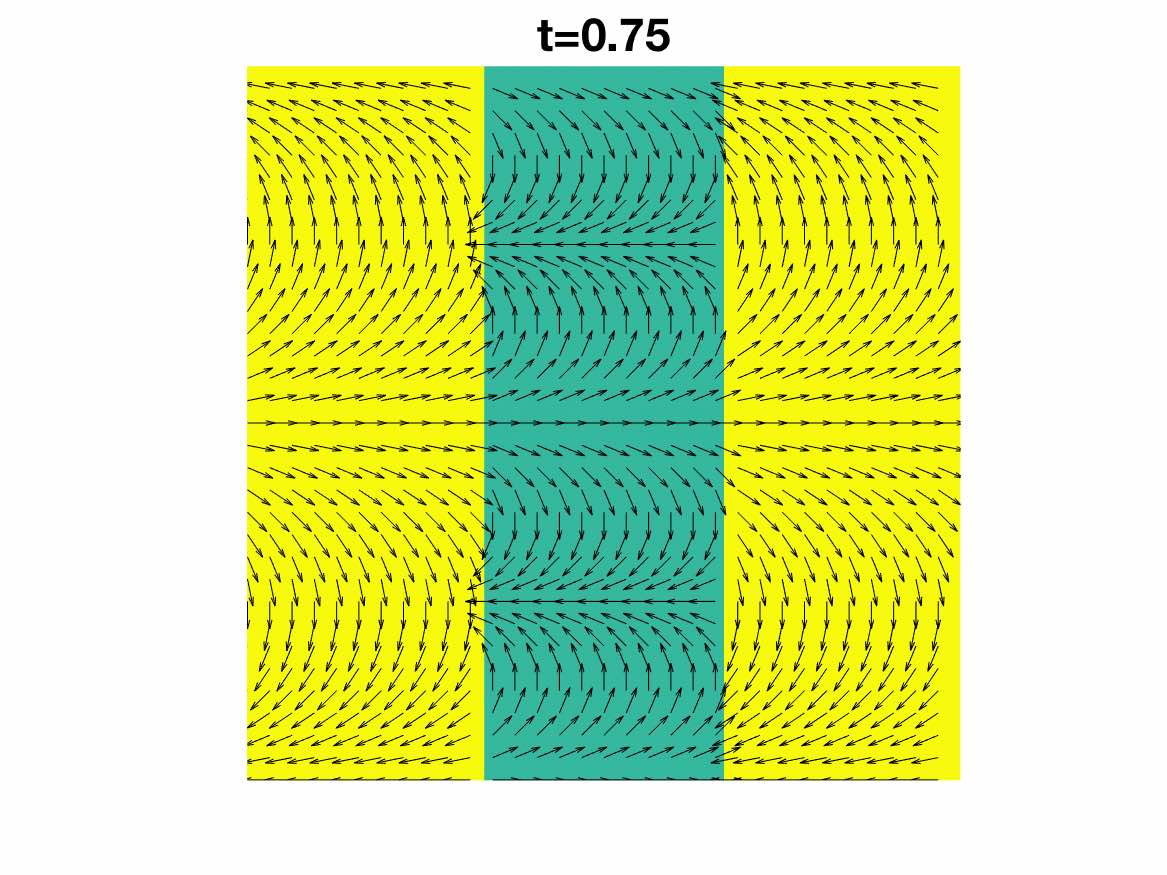}\\
\includegraphics[width = 0.24 \textwidth,clip,trim= 5cm 1cm 5cm 0cm]{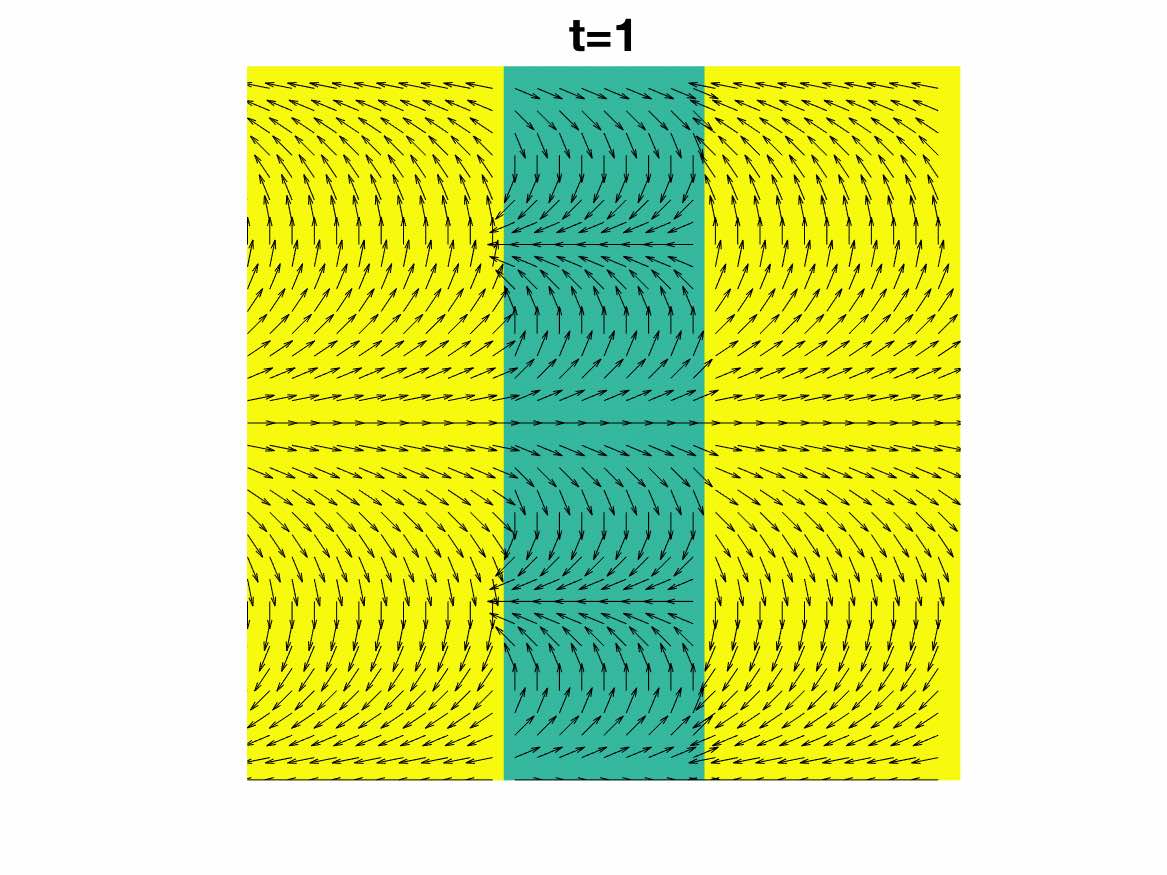}
\includegraphics[width = 0.24 \textwidth,clip,trim= 5cm 1cm 5cm 0cm]{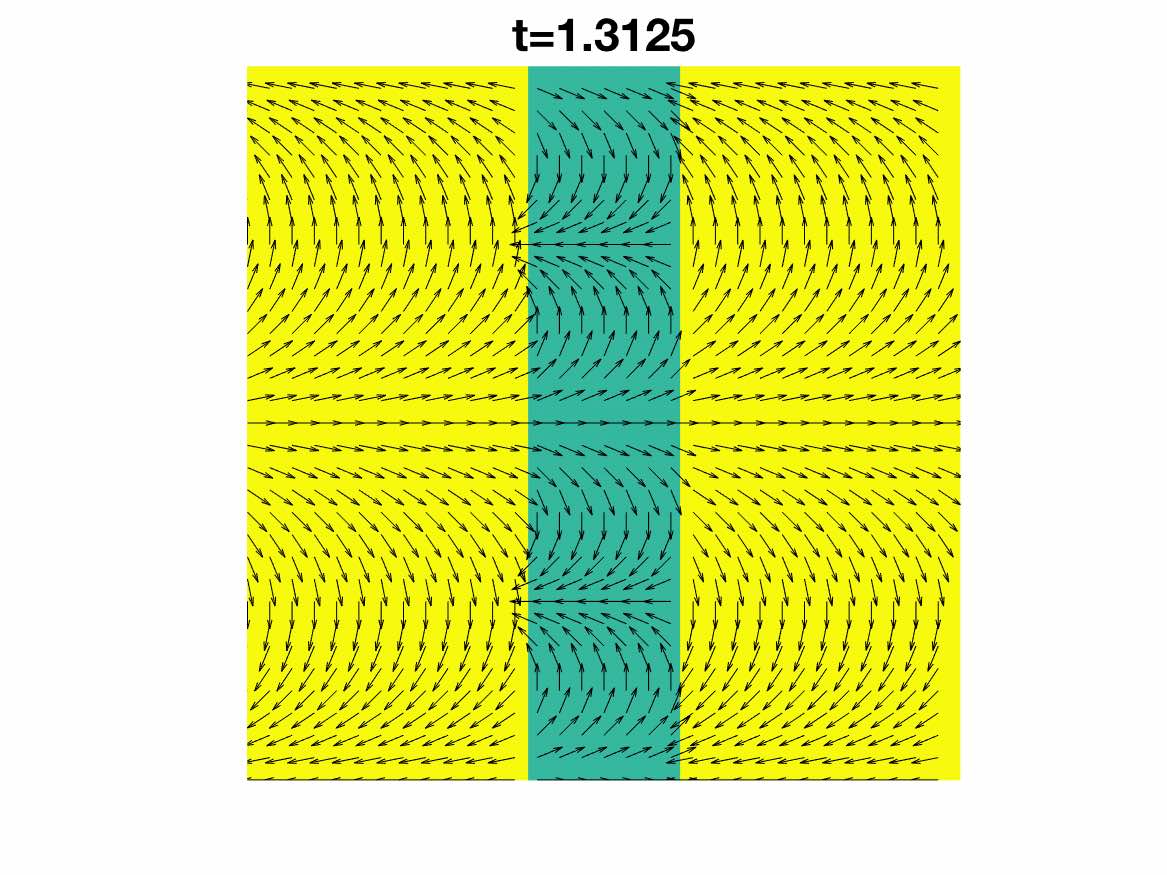}
\includegraphics[width = 0.24 \textwidth,clip,trim= 5cm 1cm 5cm 0cm]{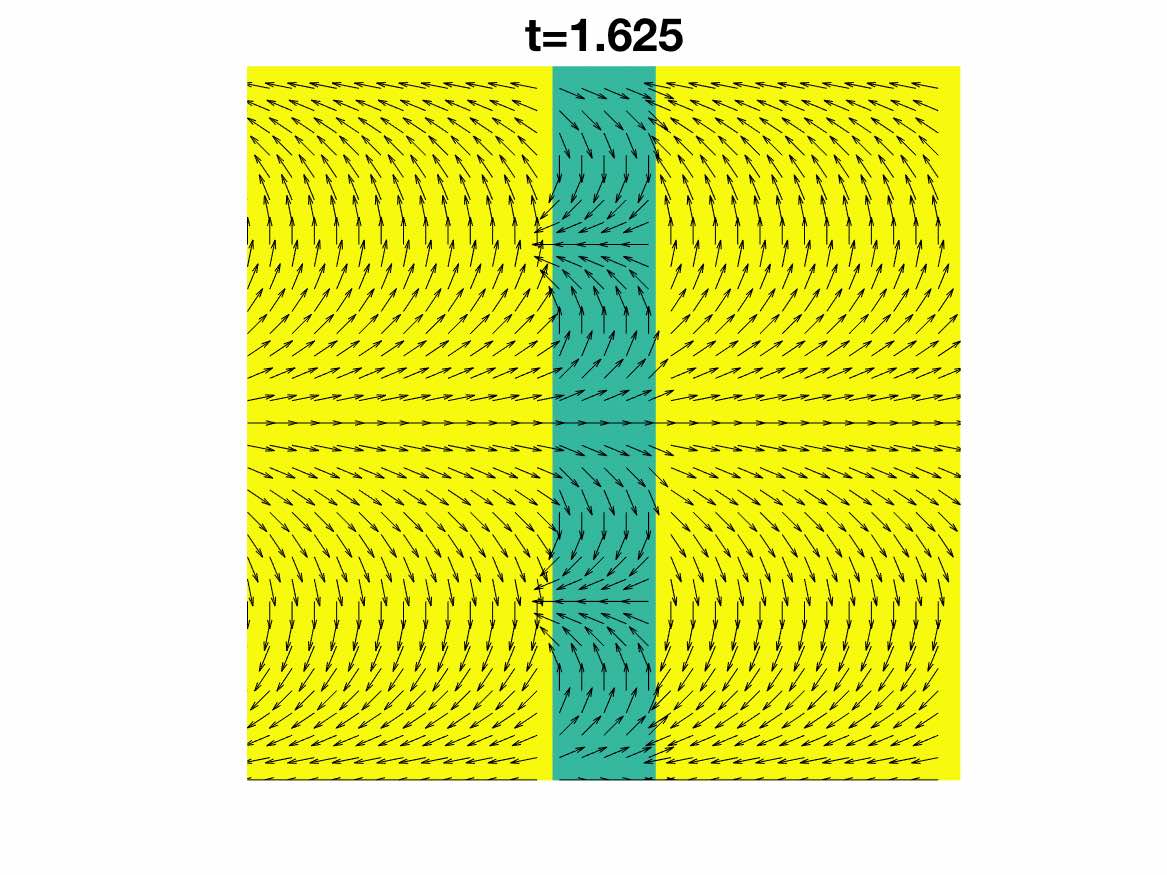}
\includegraphics[width = 0.24 \textwidth,clip,trim= 5cm 1cm 5cm 0cm]{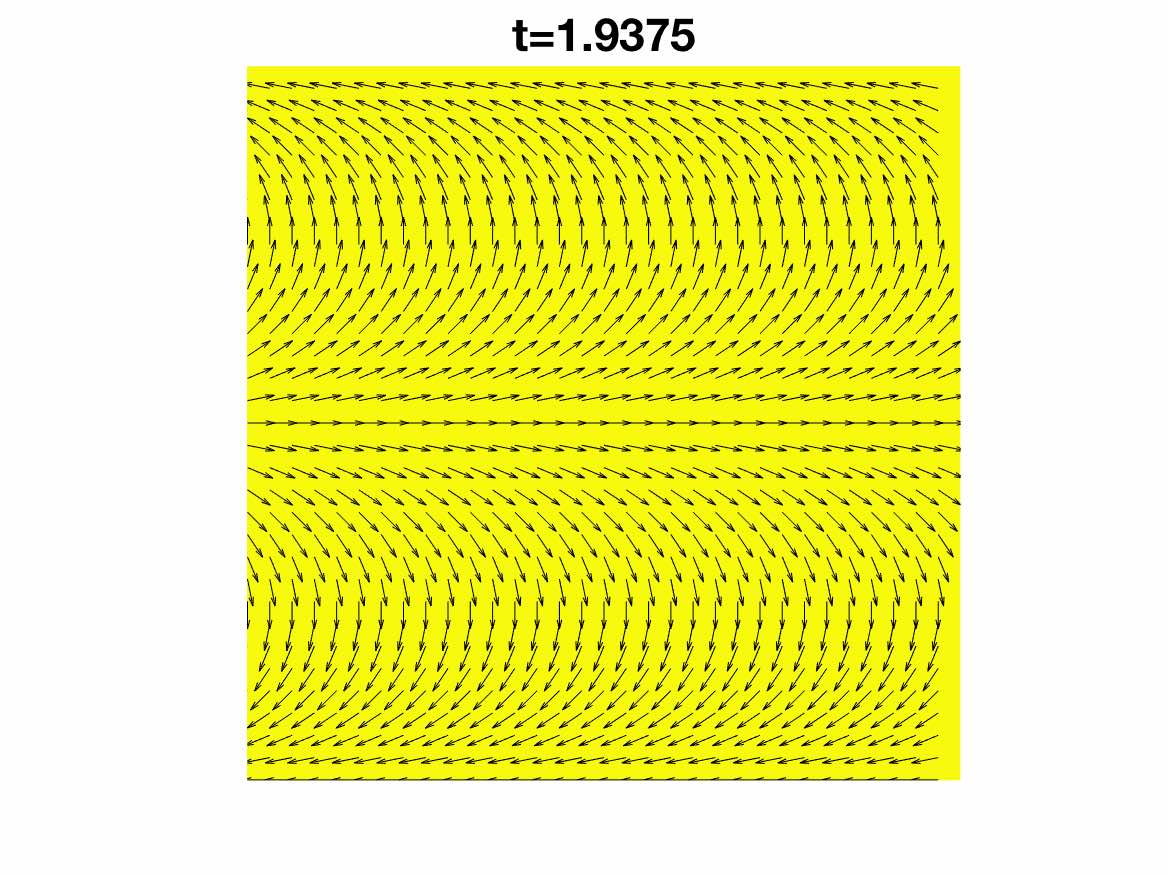}
\caption{Snapshots of the time evolution of an initial $O_2$ matrix-valued field with parallel line defects. The initial field is given in \eqref{eq:initialslow2} with $\eta_1 = 2 \pi x_1$ and $\eta_2 = 4\pi x_1$. See Section~\ref{sec:doubleslow}.}\label{fig:5}
\end{figure}

\begin{figure}[ht]
\includegraphics[width = 0.24 \textwidth,clip,trim= 5cm 1cm 5cm 0cm]{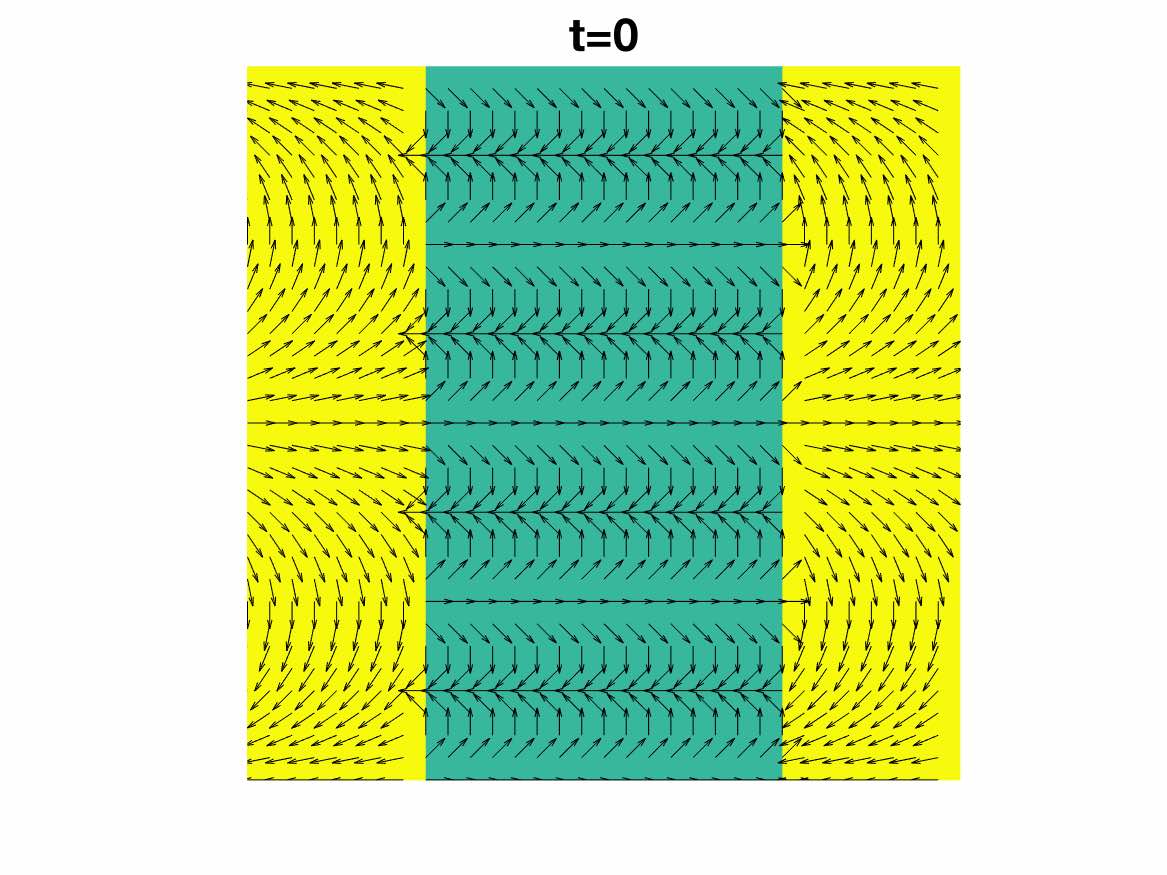}
\includegraphics[width = 0.24 \textwidth,clip,trim= 5cm 1cm 5cm 0cm]{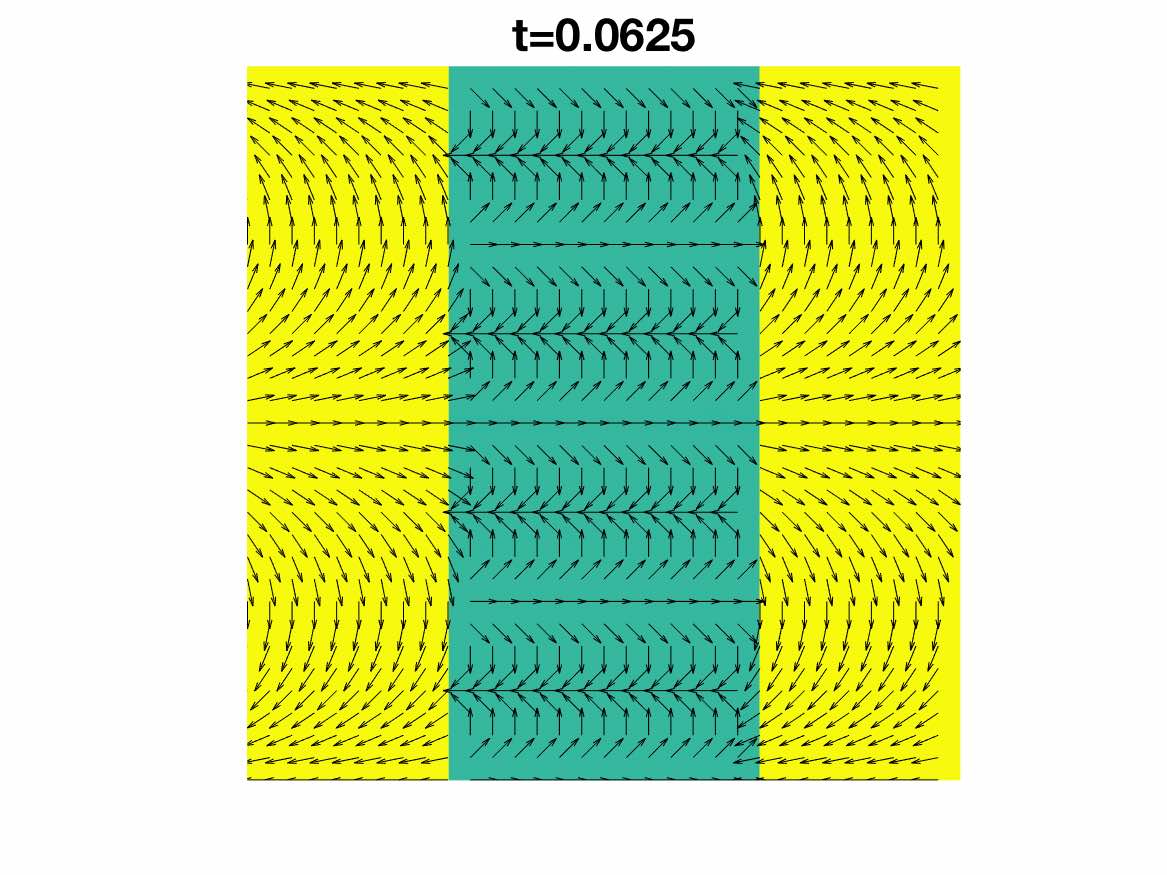}
\includegraphics[width = 0.24 \textwidth,clip,trim= 5cm 1cm 5cm 0cm]{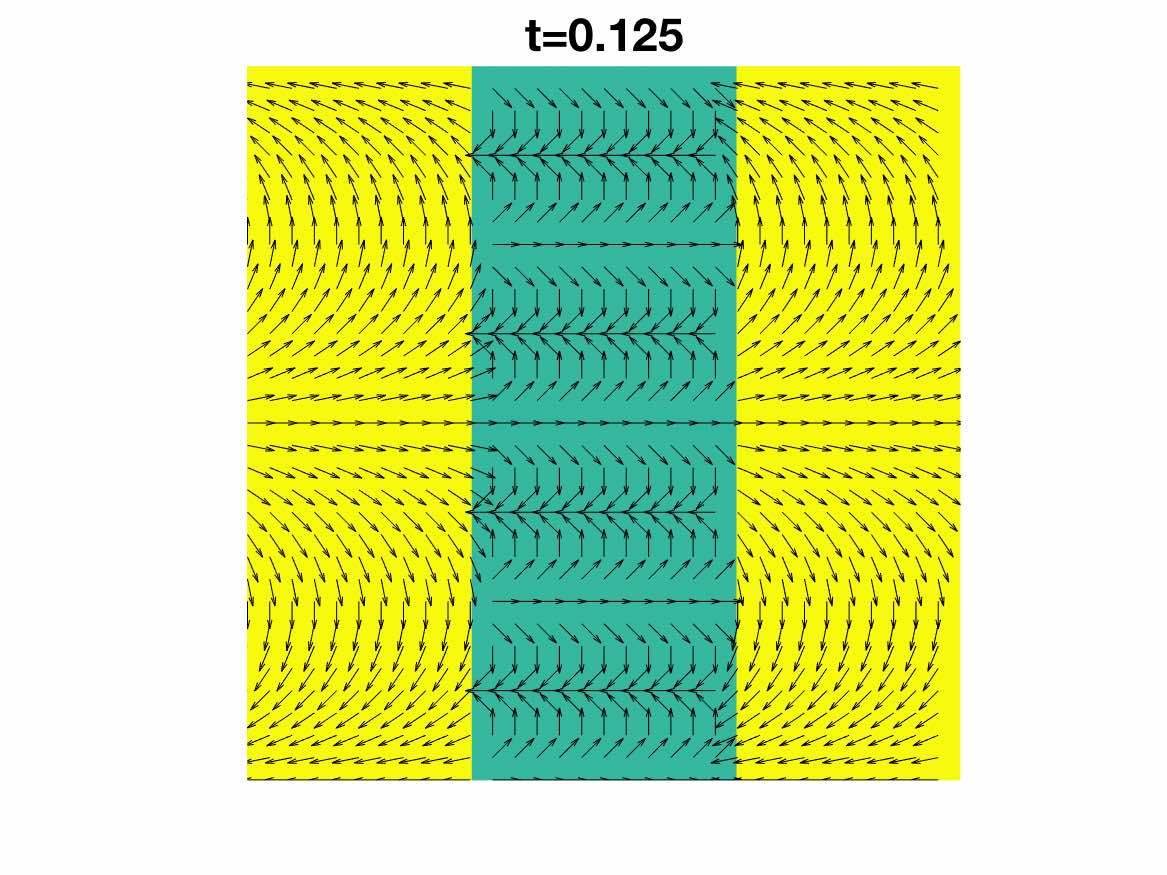}
\includegraphics[width = 0.24 \textwidth,clip,trim= 5cm 1cm 5cm 0cm]{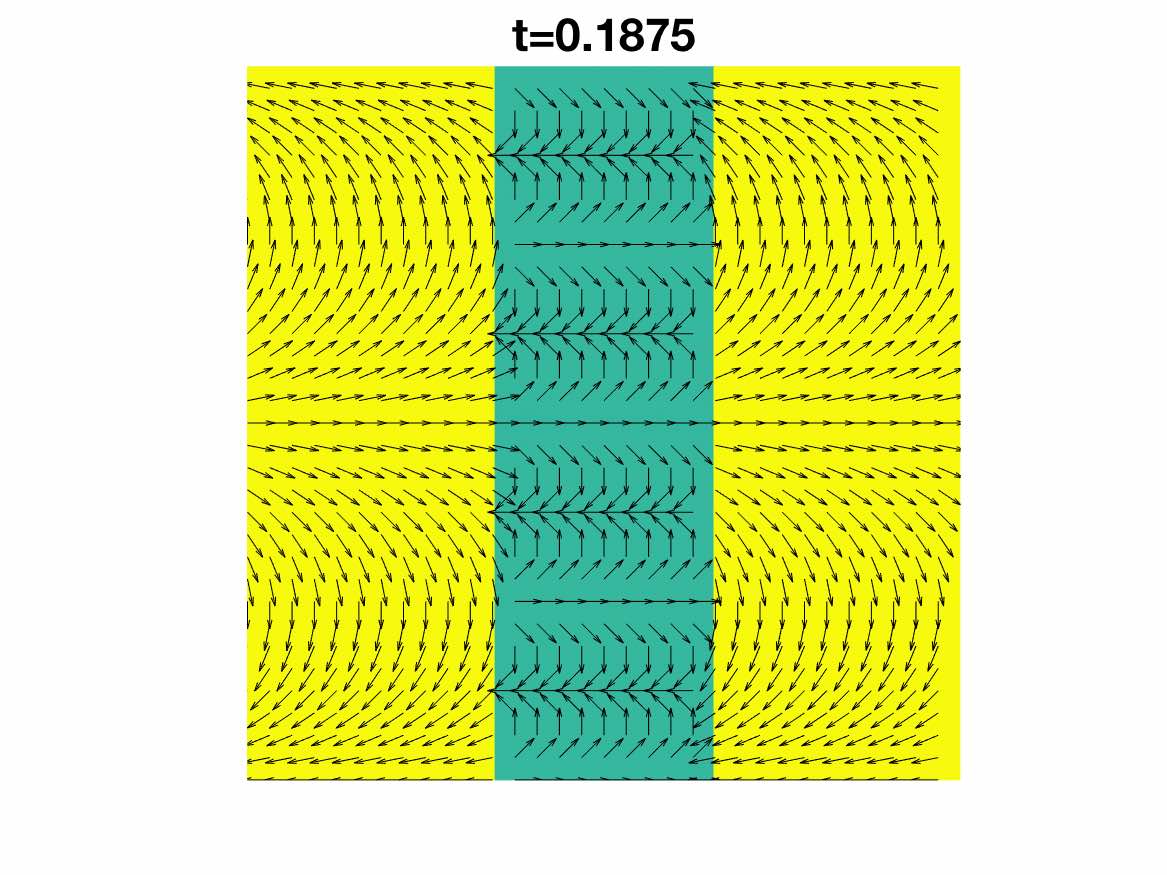}\\
\includegraphics[width = 0.24 \textwidth,clip,trim= 5cm 1cm 5cm 0cm]{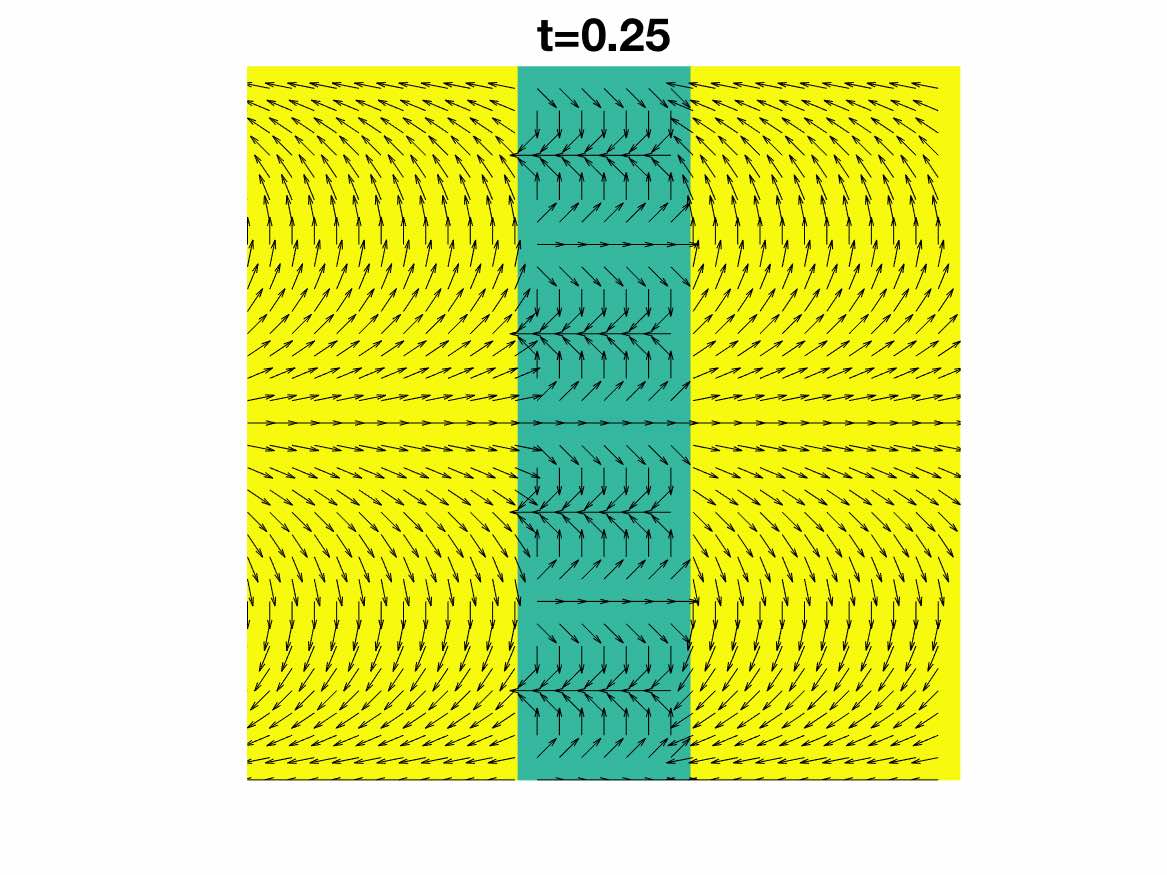}
\includegraphics[width = 0.24 \textwidth,clip,trim= 5cm 1cm 5cm 0cm]{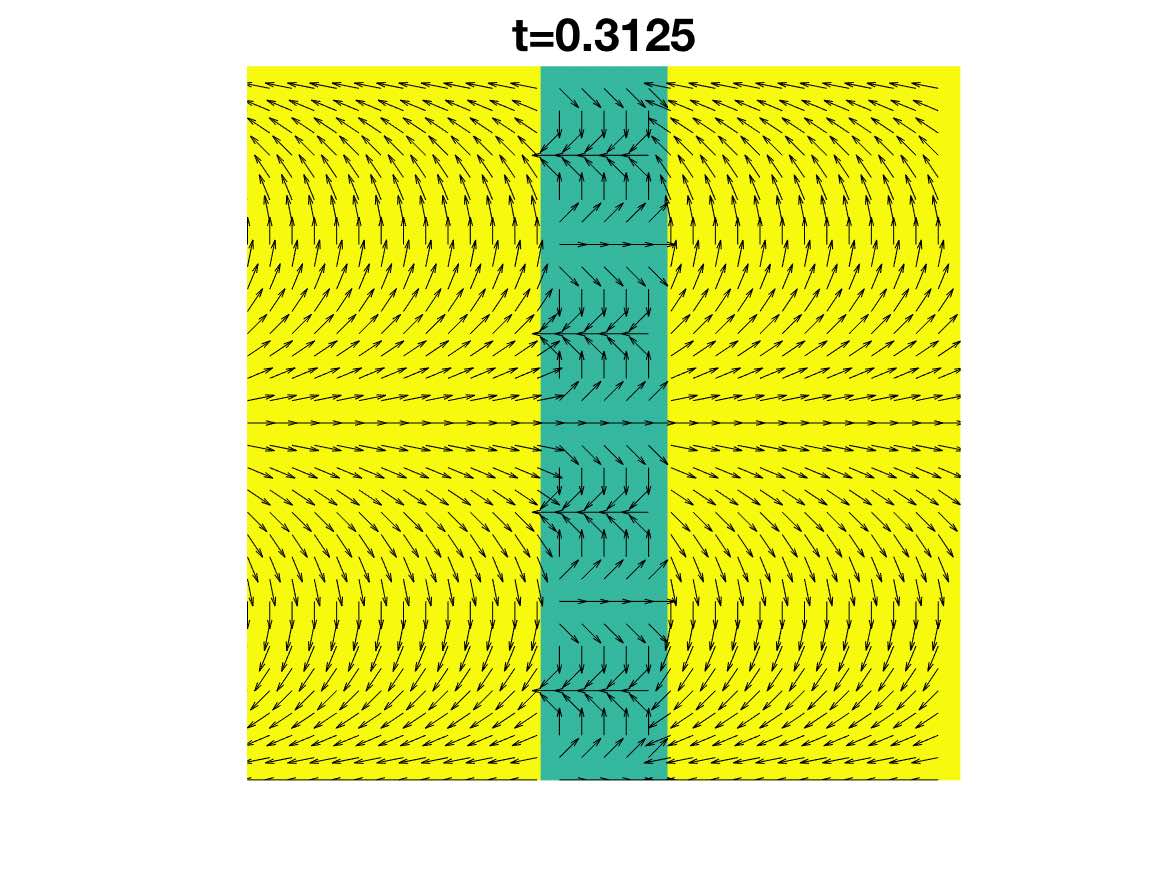}
\includegraphics[width = 0.24 \textwidth,clip,trim= 5cm 1cm 5cm 0cm]{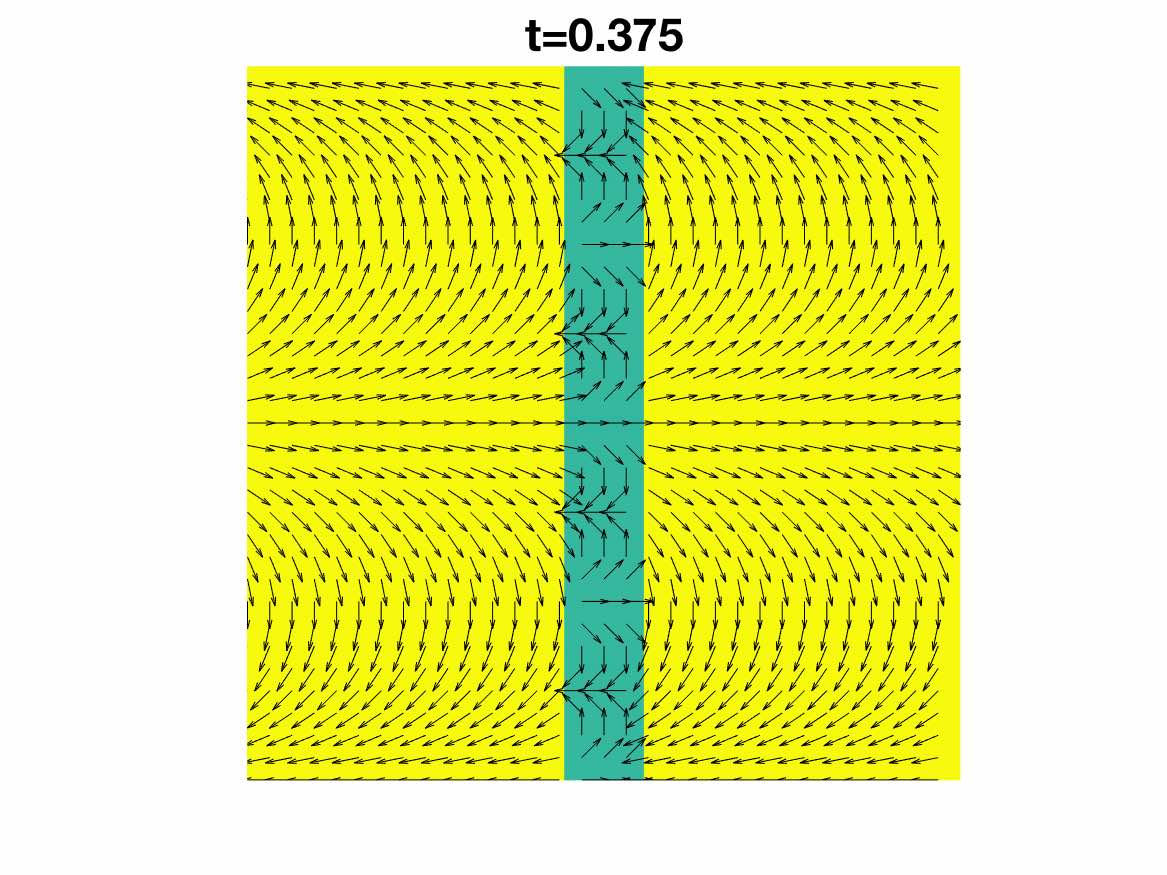}
\includegraphics[width = 0.24 \textwidth,clip,trim= 5cm 1cm 5cm 0cm]{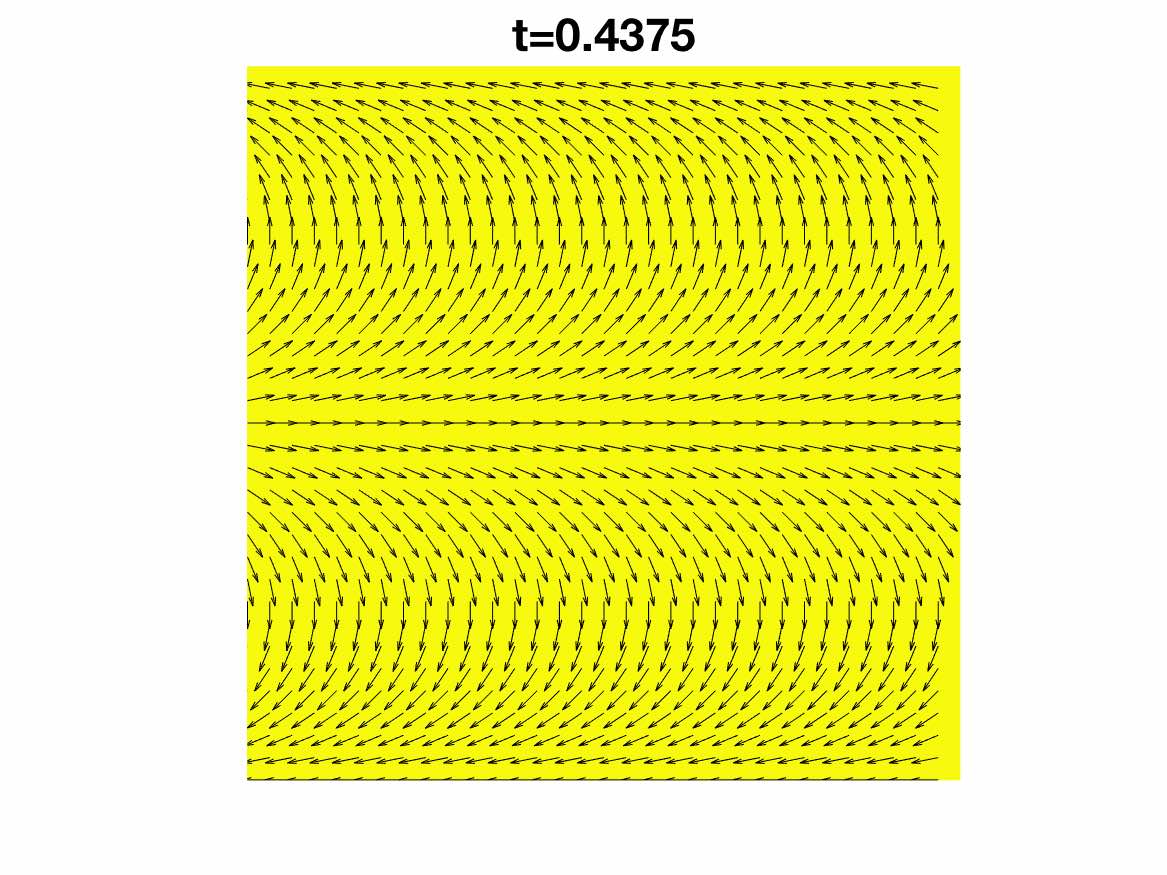}
\caption{Snapshots of the time evolution of an initial $O_2$ matrix-valued field with parallel line defects. The initial field is given in \eqref{eq:initialslow2} with $\eta_1 = 2 \pi x_1$ and $\eta_2 = 8\pi x_1$. See Section~\ref{sec:doubleslow}.}\label{fig:6}
\end{figure}

\begin{figure}[ht]
\includegraphics[width = 0.24 \textwidth,clip,trim= 5cm 1cm 5cm 0cm]{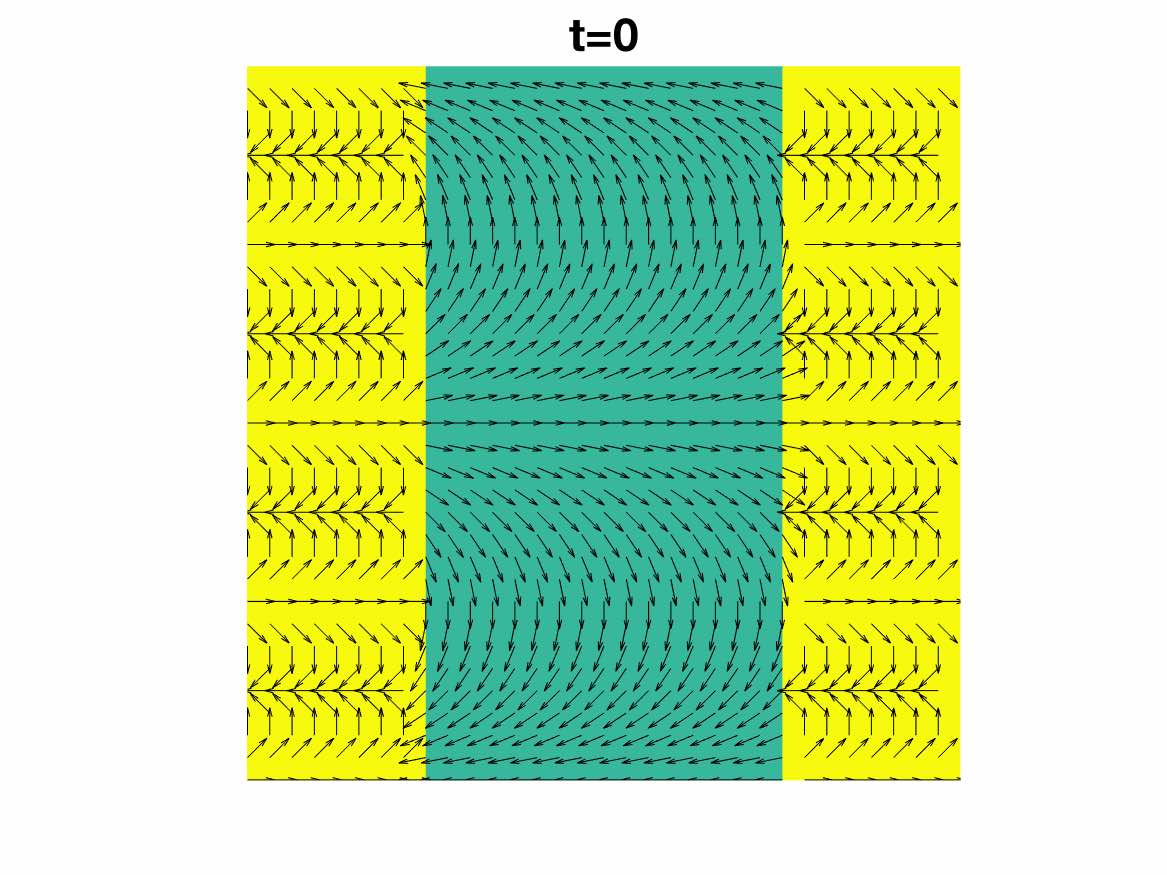}
\includegraphics[width = 0.24 \textwidth,clip,trim= 5cm 1cm 5cm 0cm]{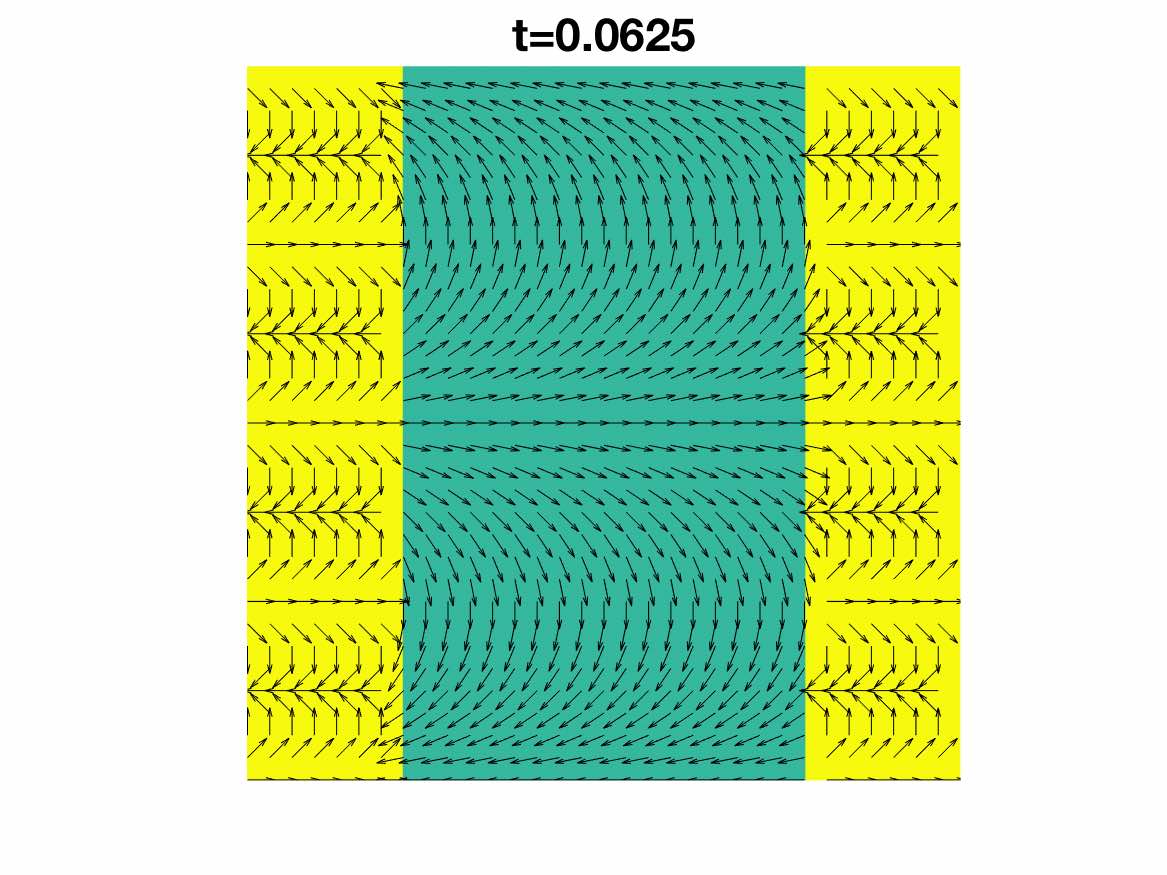}
\includegraphics[width = 0.24 \textwidth,clip,trim= 5cm 1cm 5cm 0cm]{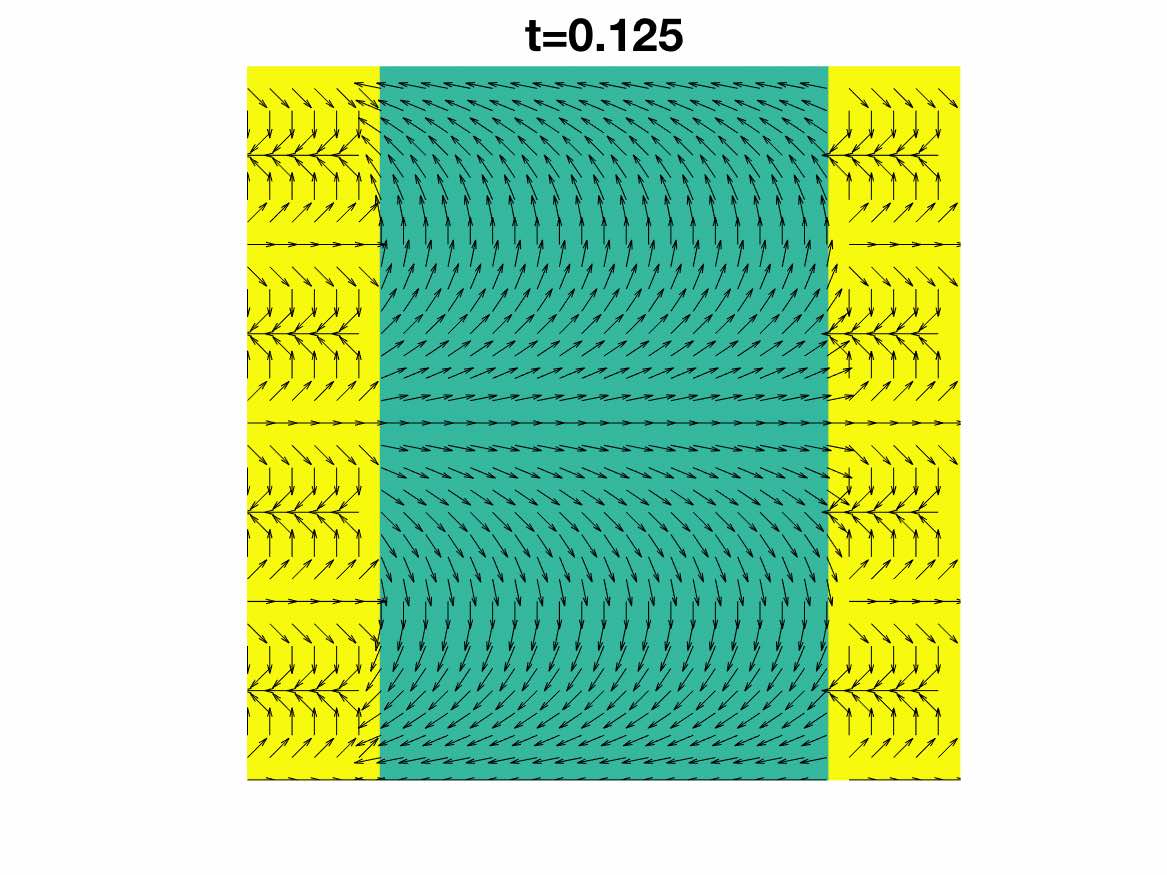}
\includegraphics[width = 0.24 \textwidth,clip,trim= 5cm 1cm 5cm 0cm]{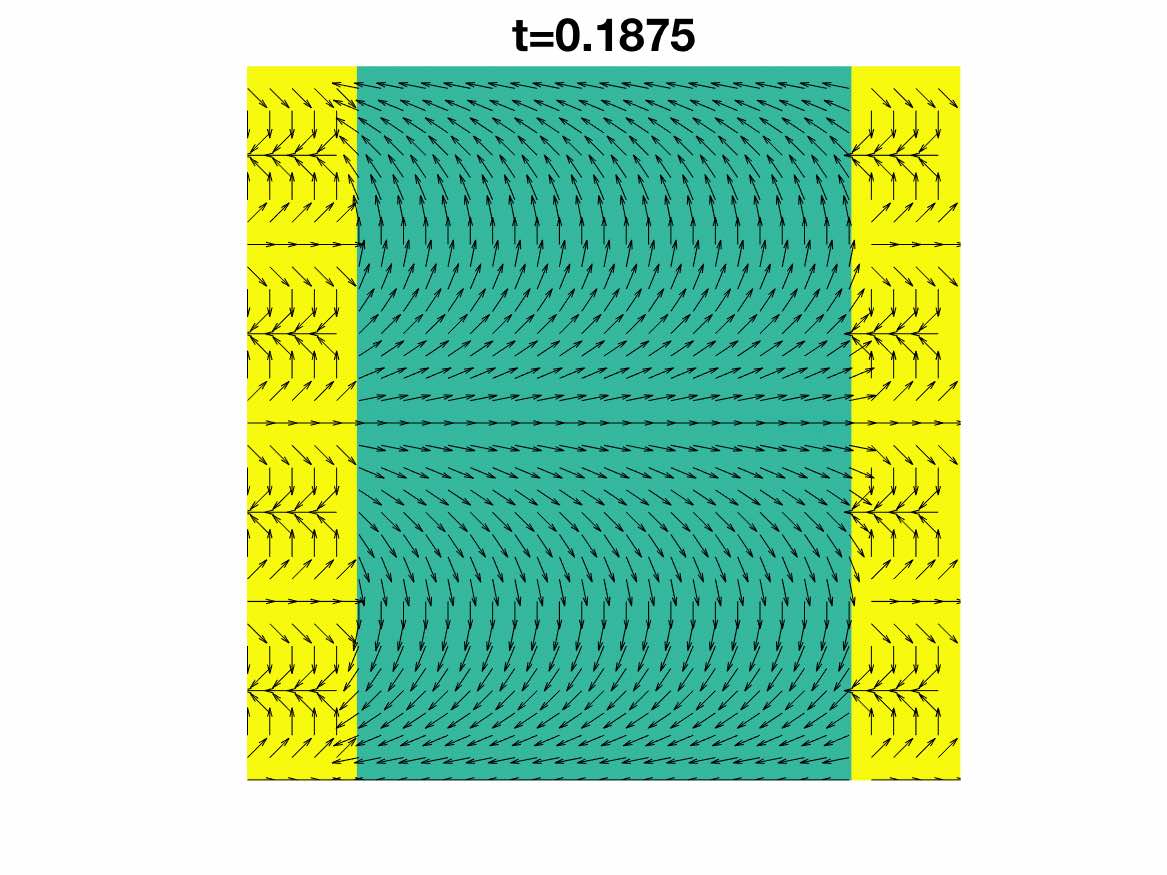}\\
\includegraphics[width = 0.24 \textwidth,clip,trim= 5cm 1cm 5cm 0cm]{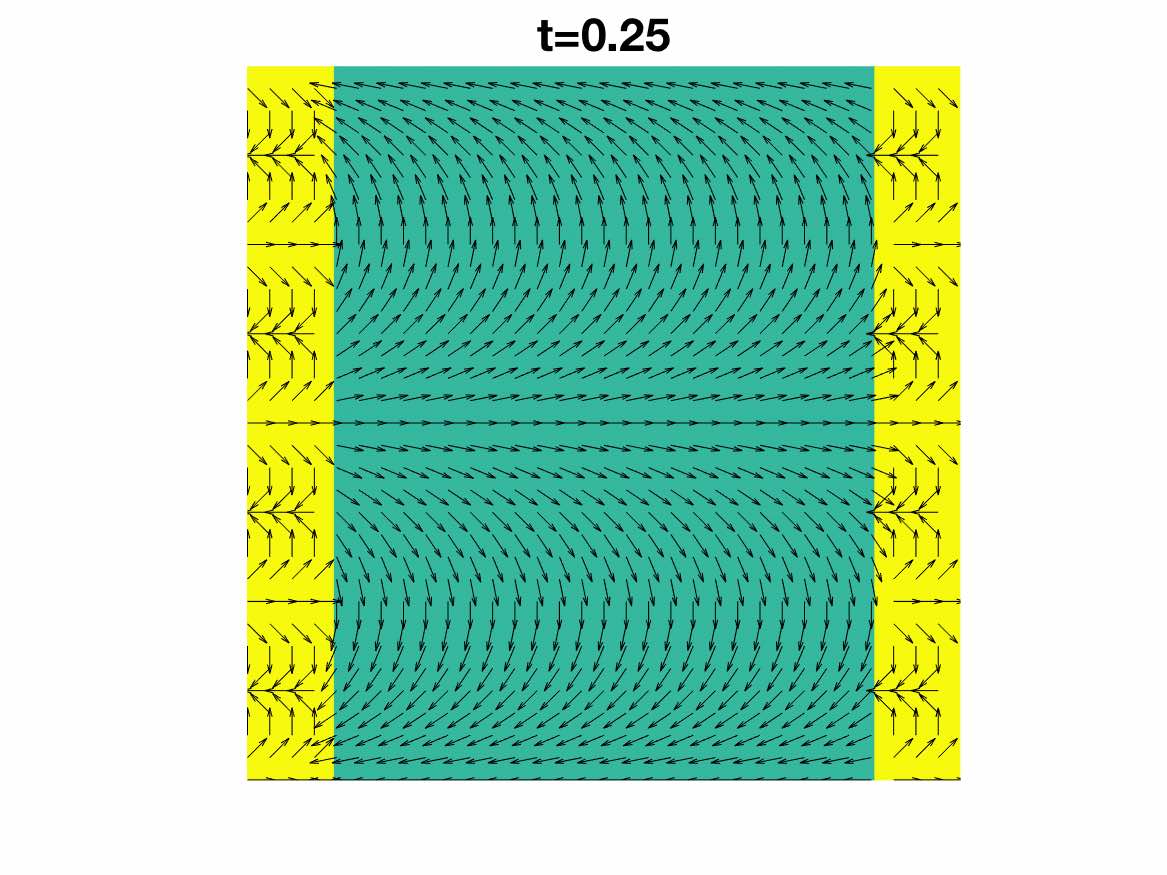}
\includegraphics[width = 0.24 \textwidth,clip,trim= 5cm 1cm 5cm 0cm]{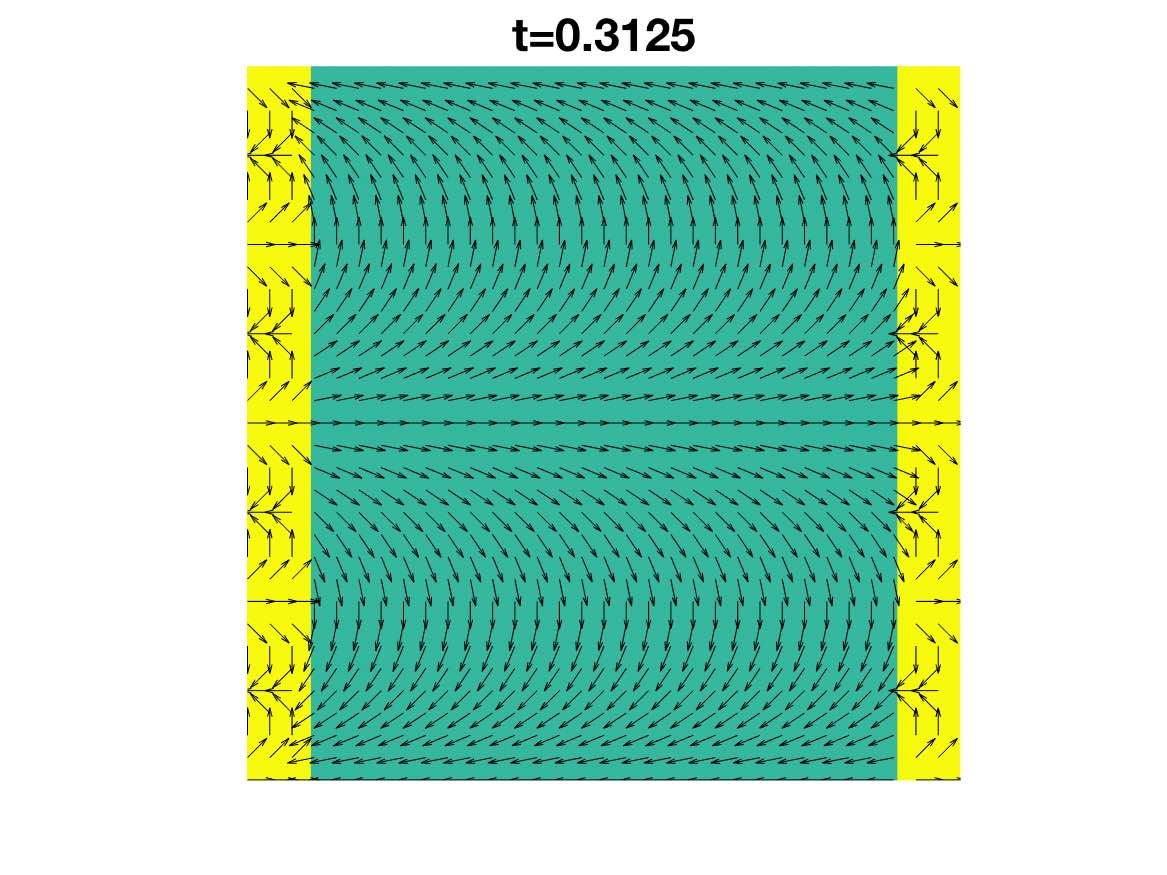}
\includegraphics[width = 0.24 \textwidth,clip,trim= 5cm 1cm 5cm 0cm]{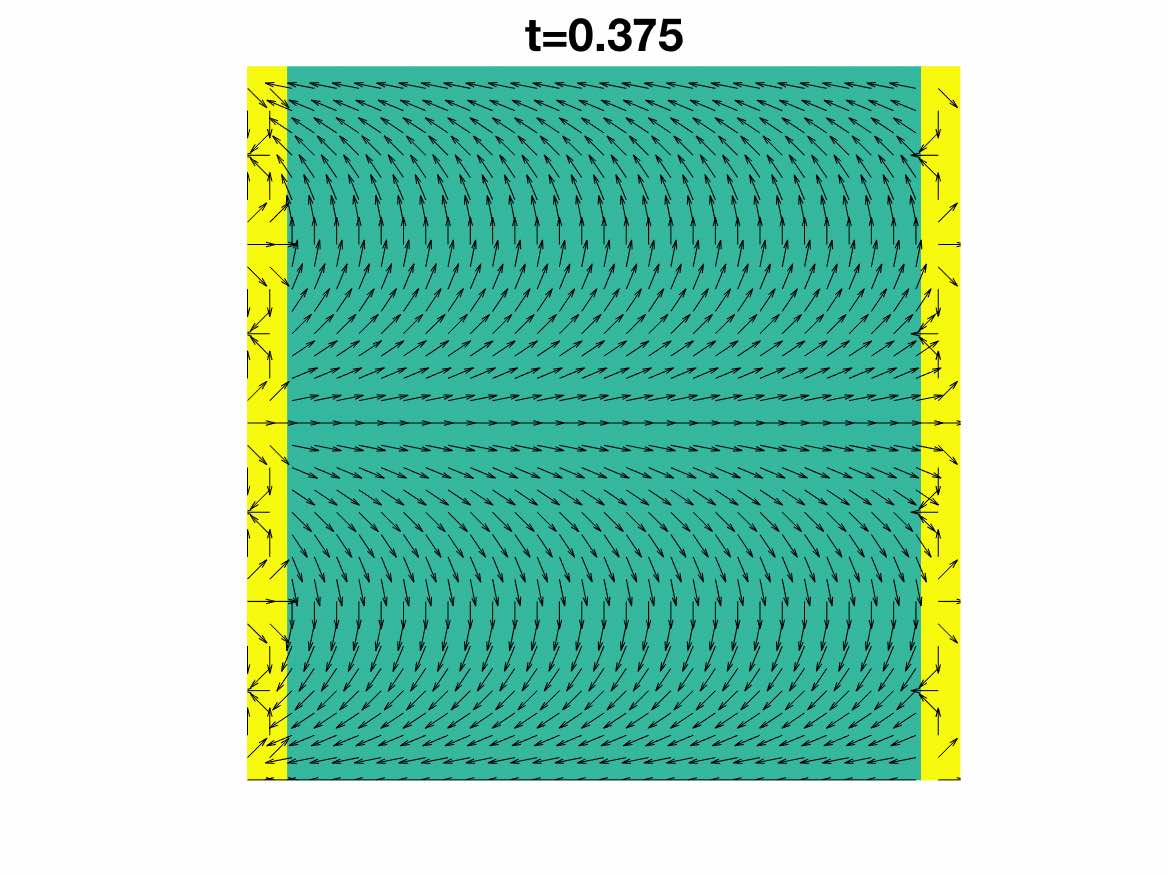}
\includegraphics[width = 0.24 \textwidth,clip,trim= 5cm 1cm 5cm 0cm]{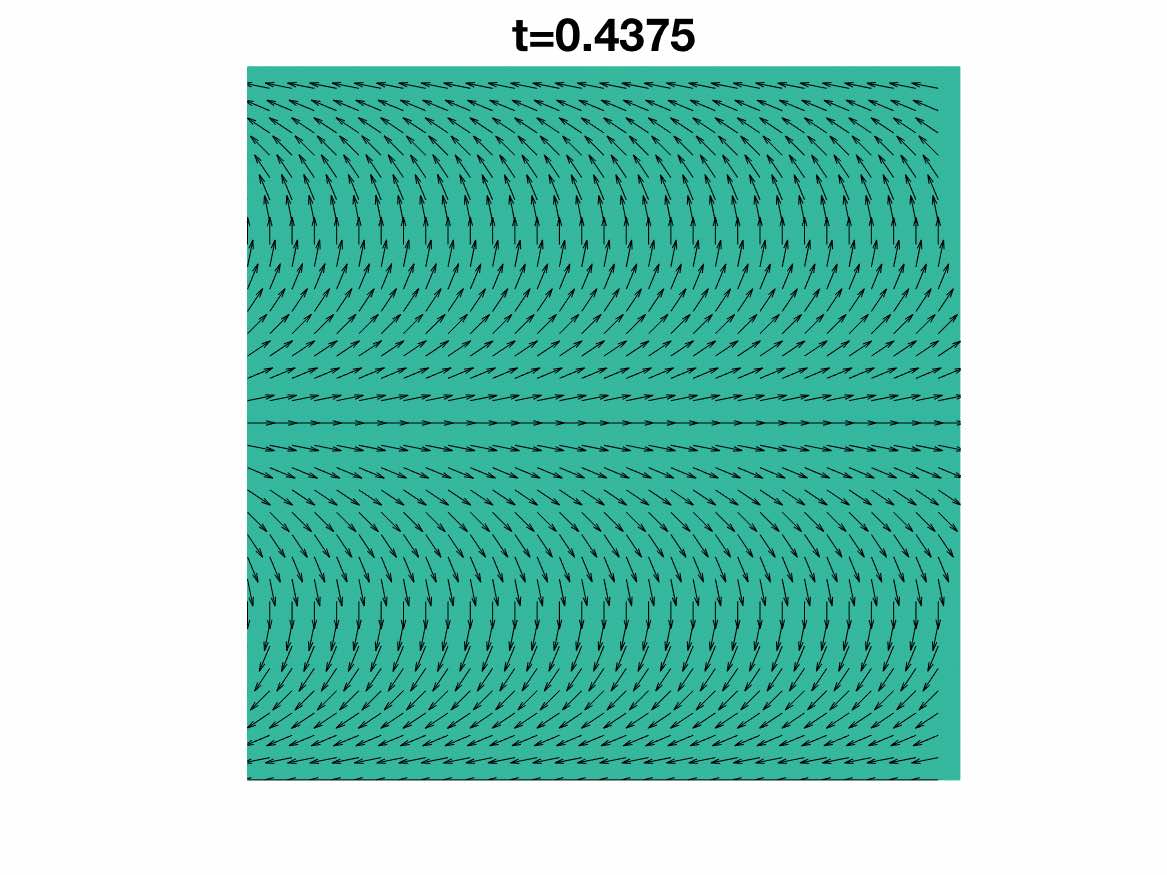}
\caption{Snapshots of the time evolution of an initial $O_2$ matrix-valued field with parallel line defects. The initial field is given in \eqref{eq:initialslow2} with $\eta_1 = 8 \pi x_1$ and $\eta_2 = 2\pi x_1$. See Section~\ref{sec:doubleslow}.}\label{fig:7}
\end{figure}

\section{Conclusion and discussion}\label{sec:con}
In this paper, we used  asymptotic methods to study the initial value problem for the generalized Allen-Cahn equation in \eqref{e:orig}. 
If the initial condition has single-signed determinant, at each point of the domain, at a fast $O(\varepsilon^{-2} t)$ time scale, the solution evolves towards the closest orthogonal matrix. Then, at the $O(t)$ time scale, the solution evolves according to the $O_n$ diffusion equation \eqref{e:OnDiffEqa}. Stationary solutions to the  $O_n$ diffusion equation were analyzed for $n=2$ in Section~\ref{s:Harmonic}. 
If the initial condition has regions where the determinant is positive and  negative, an interface develops. Away from the interface, in each region, the matrix-valued field behaves as in the single-signed determinant case. 
At the $O(t)$ time scale, the interface evolves in the normal direction by curvature. 
At a slow $O(\varepsilon t)$ time scale, for $n=2$, the interface is driven by curvature and the jump in the squared tangental derivative of the phase across the interface.
In Section~\ref{sec:num}, we conducted a variety of numerical experiments  to verify, support, and illustrate our analytical results. 

In this paper, we have focused on the two-dimensional problem. We expect that the asymptotic methods in \cite{Dai_2012,Dai_2014} could be used to study  higher-dimensional problems. In this paper, we also only focused on a square with periodic boundary conditions. We used this to derive harmonic orthogonal matrix-valued-fields in Section~\ref{s:Harmonic} and in the numerical examples in  Section~\ref{sec:num}. However, the asymptotic results from Sections~\ref{sec:part1} and \ref{sec:part2}  apply to other boundary conditions as well. 

There are several places where we focused on the $n=2$ case. In particular, 
in Section~\ref{s:Harmonic}, we derived explicit $O_2$ harmonic fields; 
in Proposition~\ref{p:MatrixTanhProfile} we explicitly derived the transition profile for the boundary layer; 
and in Proposition~\ref{p:SlowMotionLaw}, we were able to simplify the expression for the motion law at the slow  $O(\varepsilon t)$ time scale in terms of the jump in the squared tangental derivative of the phase across the interface.
It would be interesting to extend these results to $n\geq 3$. 

Here, we have considered the $L^2$ gradient flow \eqref{e:orig} of the energy $E$ in \eqref{e:energy}.  An interesting equation would arise from considering the $H^{-1}$ gradient flow of $E$, a generalization of the Cahn--Hilliard equation \cite{Pego_1989,Dai_2012,Dai_2014,Cahn_2013,Chen_2014,Wang_2017}. 
Another way to generalize  \eqref{e:orig} would be to consider multi-phase systems  as in \cite{Rubinstein_1989,Bronsard_1993}.

%\clearpage

\printbibliography
\end{document}